\documentclass[11pt,letterpaper]{amsart}
\usepackage[utf8]{inputenc}
\usepackage[english]{babel}
\usepackage[shortlabels]{enumitem}
\usepackage{amsmath}
\usepackage{amsfonts}
\usepackage{amsthm}
\usepackage{amssymb}
\usepackage{mathtools}
\usepackage{tikz}
\usepackage{tikz-cd}

\usepackage{hyperref}
\usepackage{cleveref}
\usepackage{url}
\usepackage{leftidx}
\usepackage[T1]{fontenc}
\usetikzlibrary{fit,backgrounds}
\usepackage{spectralsequences}
\usepackage{mathabx,epsfig}

\usepackage[a4paper, margin=3cm]{geometry}

\usepackage{CommonMath}

\hypersetup{
    colorlinks=true,
    linkcolor=blue,
    filecolor=magenta,      
    urlcolor=blue,
    pdftitle={Overleaf Example},
    pdfpagemode=FullScreen,
    }

\DeclareMathOperator{\Aut}{Aut}

\DeclareMathOperator{\id}{id}

\DeclareMathOperator{\Spec}{Spec}

\DeclareMathOperator{\sgn}{sgn}

\DeclareMathOperator{\im}{im}

\DeclareMathOperator{\Proj}{Proj}

\DeclareMathOperator{\Tor}{Tor}
\DeclareMathOperator{\Ext}{Ext}
\DeclareMathOperator{\Sym}{Sym}

\DeclareMathOperator{\GL}{GL}
\DeclareMathOperator{\SL}{SL}

\DeclareMathOperator{\orb}{orb}
\DeclareMathOperator{\br}{br}
\DeclareMathOperator{\univ}{univ}
\DeclareMathOperator{\sm}{sm}
\DeclareMathOperator{\Quot}{Quot}
\DeclareMathOperator{\gen}{gen}

\DeclareMathOperator{\Hur}{Hur}
\DeclareMathOperator{\Conf}{Conf}
\DeclareMathOperator{\PConf}{PConf}
\DeclareMathOperator{\stab}{stab}
\DeclareMathOperator{\std}{std}

\DeclareMathOperator{\Covers}{Covers}
\DeclareMathOperator{\Gor}{Gor}
\DeclareMathOperator{\Cov}{Cov}
\DeclareMathOperator{\fr}{fr}
\DeclareMathOperator{\coker}{coker}
\DeclareMathOperator{\CovRes}{CovRes}

\DeclareMathOperator{\length}{length}
\DeclareMathOperator{\loc}{loc}

\title{Hurwitz spaces, Nichols algebras, and Igusa zeta functions}
\author{Kevin Chang}

\begin{document}
\maketitle

\begin{abstract}
By constructing new quasimap compactifications of Hurwitz spaces of degrees 4 and 5, we establish a new connection between arithmetic statistics, quantum algebra, and geometry and answer a question of Ellenberg-Tran-Westerland and Kapranov-Schechtman. It follows from the geometry of our compactifications and a comparison theorem of Kapranov-Schechtman that we can precisely relate the following 3 quantities: (1) counts of $\mathbb{F}_q[t]$-algebras of degrees 3, 4, and 5 (2) the ``invariant'' part of the cohomology of certain special Nichols algebras (3) Igusa local zeta functions for certain prehomogeneous vector spaces.

Using Igusa's computation of the zeta function for the space of pairs of ternary quadratic forms, we compute the number of quartic $\mathbb{F}_q[t]$-algebras with cubic resolvent of discriminant $q^b$ and the part of the cohomology of a 576-dimensional Nichols algebra $\mathfrak{B}_4$ invariant under a natural $\mathbf{S}_4$-action. From the comparison for degree 3, we also obtain two answers to Venkatesh's question about the topological origin of the secondary term in the count of cubic fields.
\end{abstract}

\section{Introduction}\label{section:introduction}
An important problem in number theory is to estimate the number of extensions of a global field with fixed degree and bounded discriminant. So far, the main term of this estimate has been found by Bhargava-Shankar-Wang \cite{bsw_geometry} in degrees up to 5 for global fields of characteristic not 2. In the work of Bhargava-Shankar-Wang and their predecessors \cite{b_density_quartic,b_density_quintic,dw_density_cubic,dh_density_discriminants_cubic}, the strategy to count extensions of a global field $K$ has two steps: (1) count $\mathcal{O}_K$-algebras (2) sieve out the ones that do not correspond to rings of integers of extensions of $K$.

The reason this strategy works in degrees up to 5 is the following: for $3 \le d \le 5$, degree $d$ $\mathcal{O}_K$-algebras (with extra resolvent structure if $d \in \{4,5\}$) are parametrized by orbits in \emph{prehomogeneous vector spaces} $(G_d,V_d)$. Each $G_d$ is a reductive group acting linearly on a vector space $V_d$ with a Zariski dense open orbit $V_d^{ss}$ such that the generic stabilizer is the symmetric group $\mathbf{S}_d$. Each $V_d$ is equipped with a function $\Delta_d$ whose value at a point is the discriminant of the corresponding degree $d$ algebra. Thus, the problem of counting degree $d$ $\mathcal{O}_K$-algebras with bounded discriminant is reduced to counting lattice points $V_d(\mathcal{O}_K)$ with bounded $\Delta_d$. Because of the classification of prehomogeneous vector spaces, there are no parametrizations of higher degree algebras by prehomogeneous vector spaces.

From the geometric point of view, the reason why the degrees $d \le 5$ are special is that for these $d$, $B\mathbf{S}_d$ can be presented as a GIT quotient of a prehomogeneous vector space. For $3 \le d \le 5$, there is a unique (up to scalar) \emph{discriminant} function $\Delta_d$ on $V_d$ generating the ring of relative $G_d$-invariants. The dense open $G_d$-orbit in $V_d$ is precisely the subset $V_d^{ss} \subset V_d$ where $\Delta_d$ does not vanish, and the quotient $[V_d^{ss}/G_d]$ is isomorphic to $B\mathbf{S}_d$. By results of Wood \cite{w_quartic_arbitrary} and Landesman-Vakil-Wood \cite{lvw_low_degree_hurwitz}, the quotients $[V_d/G_d]$ are genuinely moduli stacks of covers. Thus, if we have a curve $C$, a map $C \to [V_d/G_d]$ that generically lands in $B\mathbf{S}_d \subset [V_d/G_d]$ corresponds to a degree $d$ cover that is generically \'{e}tale. Such maps are \emph{quasimaps} to $B\mathbf{S}_d$ in the sense of Ciocan-Fontanine-Kim-Maulik \cite{ckm_git} and Cheong-Ciocan-Fontanine-Kim \cite{cck_orbifold}.

One goal of this paper is to study the geometry of spaces of quasimaps to $B\mathbf{S}_d$. To do this, it will be useful to study the local problem of counting $G_d(\mathbb{F}_q[[t]])$-orbits in $V_d(\mathbb{F}_q[[t]])$, which can also be thought of as counting (possibly singular) degree $d$ covers of the formal disc over $\mathbb{F}_q$. This local counting problem is related to the computation of certain Igusa local zeta functions. In many works \cite{i_complex_powers,i_results_complex,i_functional_complex,i_b_functions,i_local_prehomog,i_stationary_phase}, Igusa studied $p$-adic integrals of the form $\int_V|\Delta(v)|^sdv$, where $\Delta$ is a compactly supported function on the $K$-vector space $V$. He computed these \emph{Igusa zeta functions} for the discriminants of many prehomogeneous vector spaces, including $(G_3,V_3)$ and $(G_4,V_4)$. It turns out that if $d \in \{3,4\}$ and the size of the residue field $q$ is coprime to 6, the Igusa zeta function $\int_{V_d}|\Delta_d(v)|^sdv$ is a power series in $q$ and $t = q^{-s}$ (in fact, a rational function). A similar fact is conjectured for all Igusa zeta functions and, in particular, for the zeta function for $(G_5,V_5)$.

As the other goal of this paper, we'll relate these zeta functions to generating functions for counts of $\mathbb{F}_q[t]$-algebras and cohomology of certain special braided Hopf algebras called \emph{Nichols algebras}. We hope this new connection will be interesting to both the arithmetic statistics community and the Nichols algebra community because each side seems to have some interesting implications for the other (see \Cref{section:arith_nichols}).

\subsection{Main results}\label{section:main_results}
In this paper, we construct smooth compactifications of low degree Hurwitz stacks as stacks of quasimaps from orbinodal curves to $B\mathbf{S}_d$. As stated above, Igusa zeta functions govern a good amount of the geometry of these compactifications.  providing bounds on fiber dimensions of \emph{branch morphisms} from these compactifications to symmetric powers of the base curve. Moreover, when our base curve is $\mathbb{A}^1$, we can use $\mathbb{G}_m$-localization to compute the cohomology of our compactifications and obtain counts of low degree $\mathbb{F}_q[t]$-algebras by reducing to the local counting problem of computing Igusa zeta functions. The cohomology of our compactifications for base curve $\mathbb{A}^1$ is also equal to the cohomology of certain Nichols algebras, so we are able to obtain new computations of Nichols algebra cohomology and reprove some earlier computations.

To state our results in a nice way, we define the generating functions $$I_d(q,t) \coloneqq C_d\int_{V_d}|\Delta_d(v)|^{s - 1}dv$$ for $d \in \{3,4\}$, where the $C_d$ are constants such that the $I_d$ have constant term 1. Throughout the paper, we will call these ``Igusa zeta functions,'' even though Igusa zeta functions do not traditionally have the shift $s \mapsto s - 1$. Their precise values were obtained by Igusa in \cite{i_stationary_phase}: \begin{align*}
I_3(q,t) &= \frac{1 + t + t^2 + t^3 + t^4}{(1 - t^2)(1 - qt^6)} \\
I_4(q,t) &= \frac{f(q,t)}{(1 - t)(1 - t^2)(1 - qt^6)(1 - q^2t^8)(1 - q^3t^{12})},
\end{align*}
where $f(q,t) = 1 + t^2 + t^3 + t^4 - 2t^5 + 2qt^6 + (q - 1)t^7 + qt^8 - qt^9 + (q - 1)qt^{10} - 2qt^{11} + 2q^2t^{12} - q^2t^{13} - q^2t^{14} - q^2t^{15} - q^2t^{17}$.

\subsubsection{Algebraic geometry}\label{section:ag_results}
The main geometric construction on which all our results are based is the \emph{big stack of degree $d$ covers with resolvents} $\mathcal{R}^d$, where $d \in \{3,4,5\}$. In this section, we give the key properties and ideas of our construction and an overview of how properties of our stacks lead to the other results of this paper. For this section, fix $d$, and let $\mathbb{S}_d = \Spec\mathbb{Z}\left[\frac{1}{d!}\right]$.

Roughly, for a scheme $S$ over $\mathbb{S}_d$, an $S$-point of $\mathcal{R}^d$ parametrizes an orbinodal curve $\mathcal{C} \to S$ along with a representable morphism $\chi: \mathcal{C} \to \mathcal{X}_d \coloneqq [V_d/G_d]$ sending all generic, marked, and nodal points to the \emph{\'{e}tale locus} $\mathcal{E}_d \coloneqq [V_d^{ss}/G_d] \cong B\mathbf{S}_d$. A morphism $\mathcal{C} \to \mathcal{X}_d$ with this condition parametrizes a degree $d$ cover $\mathcal{D} \to \mathcal{C}$ \'{e}tale above $\chi^{-1}(\mathcal{E}_d)$ along with extra resolvent data (if $d \in \{4,5\}$) above the points mapping outside $\mathcal{E}_d$, a.k.a. the branch points. Note that there are no conditions on whether $\mathcal{D}$ is connected.

To translate between morphisms $\mathcal{C} \to \mathcal{X}_d$ and covers of $\mathcal{C}$, we rely on results of Wood \cite{w_quartic_arbitrary} and Landesman-Vakil-Wood \cite[\S 3-4]{lvw_low_degree_hurwitz} that parametrize covers of arbitrary base schemes. These results are crucial because we need to deal with orbinodal curves over an arbitrary base $\mathcal{C} \to S$ to construct our stacks. A discussion of Wood's and Landesman-Vakil-Wood's parametrizations can be found in \Cref{section:prehomog_covs}.

The stack $\mathcal{R}^d$ comes equipped with a \emph{branch morphism} $\br: \mathcal{R}^d \to \mathcal{M}$, where $\mathcal{M}$ is the stack of pointed nodal curves equipped with a marked divisor disjoint from the nodes and marked points. Let $\mathcal{M}^{\sm}$ denote the open substack of smooth curves, and let $\mathcal{R}^{d,\sm}$ denote its preimage under $\br$. Let $\mathcal{M}^{b,n}$ denote the open and closed substack of $\mathcal{M}$ where the marked divisor has degree $b$ and there are $n$ marked points, and let $\mathcal{R}^{d;b,n}$ denote its preimage. We have a morphism $\mathcal{M} \to \mathcal{M}^{0,*}$ forgetting the marked divisor.

Our main result about the stack $\mathcal{R}^d$ is the following:

\begin{theorem}[\Cref{theorem:rd_nice}, \Cref{corollary:sm_smooth}]\label{theorem:intro_rd_nice}
\begin{enumerate}[(a)]
\item $\mathcal{R}^d$ is a smooth algebraic stack over $\mathbb{S}_d$.

\item The morphism $\br: \mathcal{R}^d \to \mathcal{M}$ is proper and representable by Deligne-Mumford stacks.

\item The composite $\mathcal{R}^{d,\sm} \xrightarrow[]{\br} \mathcal{M}^{\sm} \to \mathcal{M}^{\sm;0,*}$ is smooth. The restriction $\mathcal{R}^{d,\sm;b,*} \to \mathcal{M}^{\sm;0,*}$ has relative dimension $b$.
\end{enumerate}
\end{theorem}

Let $\mathcal{M}^{\circ} \subset \mathcal{M}$ denote the open substack where the marked divisor is reduced. We can then define the \emph{big Hurwitz stack of simply branched degree $d$ covers} $\mathcal{H}^d = \br^{-1}(\mathcal{M}^{\circ}) \subset \mathcal{R}^d$. It turns out that morphisms $\mathcal{C} \to \mathcal{X}_d$ corresponding to simply branched covers $\mathcal{D} \to \mathcal{C}$ have unique resolvent data above the branch points up to unique isomorphism. Thus, $\mathcal{H}^d$ genuinely parametrizes simply branched degree $d$ covers, as its name would suggest.

Given a smooth $n$-pointed curve $(C;\sigma)$ over a field $k$ (over $\mathbb{S}_d$), we get a map $$\br: \mathcal{R}_{(C;\sigma),b}^d \to \Sym^b(C^{\gen})$$ by taking the base change of $\mathcal{R}^{d;b,n} \xrightarrow[]{\br} \mathcal{M}^{b,n} \to \mathcal{M}^{0,n}$ along the map $\Spec k \to \mathcal{M}^{0,*}$ classifying $(C;\sigma)$. Here, $C^{\gen}$ is the complement of the marked points $\sigma$ in $C$. By \Cref{theorem:intro_rd_nice}, $\mathcal{R}_{(C;\sigma),b}^d$ is a smooth Deligne-Mumford stack, and $\br$ is proper. The open substack $\mathcal{H}_{(C;\sigma),b}^d = \br^{-1}(\Conf^b(C^{\gen}))$ is now the familiar Hurwitz stack of (possibly disconnected) degree $d$ covers of $C$ simply branched above $C^{\gen}$. From the deformation theory of simply branched covers, we know that $\pi: \mathcal{H}_{(C;\sigma),b}^d \to \Conf^b(C^{\gen})$ is \'{e}tale.

We now explain the main applications of our construction: \begin{itemize}
\item \textbf{Arithmetic}: To obtain our count of quartic $\mathbb{F}_q[t]$-algebras with cubic resolvent (\Cref{theorem:intro_fqt_alg_counts_34}), we consider the case $(C;\sigma) = (\mathbb{P}^1;\infty)$. In this case, there are $\mathbb{G}_m$-actions on all the stacks such that the map $$\br: \mathcal{R}_{(\mathbb{P}^1;\infty),b}^d \to \Sym^b(\mathbb{A}^1)$$ is $\mathbb{G}_m$-equivariant. Note that $\Sym^b(\mathbb{A}^1)$ is affine such that the $\mathbb{G}_m$-action contracts everything to 0. By a $\mathbb{G}_m$-localization lemma (\Cref{lemma:sawin_gm_local}), the $\mathbb{Q}_{\ell}$-cohomology of $\mathcal{R}_{(\mathbb{P}^1;\infty),b}^d$ is isomorphic to the cohomology of its central fiber $\mathcal{F}_b^d \coloneqq \br^{-1}(0)$. Because $\mathcal{R}_{(\mathbb{P}^1;\infty),b}^d$ is smooth and $\br$ is proper, the cohomology of both of them is pure in the sense of Deligne and can be computed by counting $\mathbb{F}_q$-points of $\mathcal{F}_b^d$. This point count can be computed using the Igusa zeta function; in fact, its generating function is exactly $I_d$. This tells us the cohomology of $\mathcal{R}_{(\mathbb{P}^1;\infty),b}^d$, and by smoothness, we can apply Poincar\'{e} duality to get the compactly supported cohomology and point count of $\mathcal{R}_{(\mathbb{P}^1;\infty),b}^d$. This gives us \Cref{theorem:intro_fqt_alg_counts_34}. From the method outlined above, it is clear that the local-global duality of the counts comes from the $\mathbb{G}_m$-action and Poincar\'{e} duality.

We take a moment to explain why we allow for stackiness above $\infty$ instead of just considering non-stacky $\mathbb{P}^1$. A cover $D \to \mathbb{A}^1$ can be uniquely extended to a cover $\overline{D} \to \mathbb{P}^1$ with $\overline{D}$ smooth, which can be made into a cover of orbicurves $\overline{\mathcal{D}} \to \mathbb{P}_r^1$ \'{e}tale over $\infty$ in a unique way such that the classifying map $\mathbb{P}_r^1 \to \Covers_d$ is representable. Because $\overline{\mathcal{D}}$ is smooth above $\infty$, there is no extra resolvent data above $\infty$. Thus, a point of $\mathcal{R}_{(\mathbb{P}^1;\infty),b}^d$ corresponds to a degree $d$ cover of $\mathbb{A}^1$ with resolvent. If we consider $\br: \mathcal{R}_{\mathbb{P}^1,b}^d \to \Sym^b(\mathbb{P}^1)$ instead, then the points of $\br^{-1}(\Sym^b(\mathbb{A}^1))$ parametrize degree $d$ covers of $\mathbb{A}^1$ with resolvent that are unramified above $\infty$. These could be counted if we knew what part of the Igusa zeta function counted covers with resolvents unramified above $\infty$, but for now, we only have counts without any local conditions above $\infty$.

\item \textbf{Intersection cohomology and Nichols algebra cohomology}: We can also use the Igusa zeta functions $I_d$ to prove dimension bounds on the fibers of $\br$. Using these bounds, we obtain the following result:

\begin{theorem}[\Cref{theorem:dense_small34}, \Cref{theorem:dense_small5}]\label{theorem:intro_dense_small34}
For $d \in \{3,4\}$, \begin{enumerate}[(a)]
\item $\mathcal{H}_{(C;\sigma),b}^d$ is dense in $\mathcal{R}_{(C;\sigma),b}^d$.

\item $\br$ is small.
\end{enumerate}
The same results hold for $d = 5$ assuming a precise conjecture about the relevant Igusa zeta function (\Cref{conjecture:igusa5_value}).
\end{theorem}

Density and smallness imply that if we let $j: \Conf^b(C^{\gen}) \xhookrightarrow{} \Sym^b(C^{\gen})$ denote the open immersion, then $$j_{!*}\pi_*\underline{\mathbb{Q}_{\ell}}_{\mathcal{H}_{(C;\sigma),b}^d} \simeq R\br_*\underline{\mathbb{Q}_{\ell}}_{\mathcal{R}_{(C;\sigma),b}^d}.$$ If $k = \mathbb{F}_q$, then by proper base change, the trace function of the right-hand side at a point $p \in \Sym^b(C^{\gen})(\mathbb{F}_q)$ counts degree $d$ covers $D \to C^{\gen}$ (with a cubic resolvent if $d = 4$) with branch divisor $p$. This provides a counting interpretation to the trace function of the intersection complex on the left-hand side, answering a question of Ellenberg-Tran-Westerland and Kapranov-Schechtman in this case (see \Cref{section:arith_nichols}). Moreover, the comparison theorem of Kapranov-Schechtman \cite[Corollary 3.3.5]{ks_shuffle_algs} lets us use the above quasi-isomorphism of complexes to compute the $\mathbf{S}_4$-invariant part of the Nichols algebra $\mathfrak{B}_4$ (\Cref{theorem:intro_inv_coh_bd_4}) and provide an arithmetic interpretation of the cohomology of the Nichols algebras $\mathfrak{B}_d$.
\end{itemize}

\subsubsection{Arithmetic statistics}\label{section:arith_results}
Our main arithmetic result is the following sharp count of quartic algebras with cubic resolvent over $\mathbb{F}_q[t]$. The definition of a quartic algebra with cubic resolvent is due to Wood \cite{w_quartic_arbitrary} and can be found in \Cref{section:prehomog_cov_res}. For this section, let $q$ be a prime power coprime to 6.

\begin{theorem}[\Cref{theorem:fqt_alg_counts_34}]\label{theorem:intro_fqt_alg_counts_34}
Let $N_{4,q,b}$ be the number of isomorphism classes of quartic $\mathbb{F}_q[t]$-algebras with cubic resolvent of discriminant $q^b$, inversely weighted by automorphisms over $\mathbb{F}_q[t]$. Then $$\sum_{b = 0}^{\infty}N_{4,q,b}t^b = I_4(q^{-1},qt) = \frac{f(q^{-1},qt)}{(1 - qt)(1 - q^2t^2)(1 - q^5t^6)(1 - q^6t^8)(1 - q^9t^{12})},$$ where $f(q^{-1},qt) = 1 + q^2t^2 + q^3t^3 + q^4t^4 - 2q^5t^5 + 2q^5t^6 - (q - 1)q^6t^7 + q^7t^8 - q^8t^9 - (q - 1)q^8t^{10} - 2q^{10}t^{11} + 2q^{10}t^{12} - q^{11}t^{13} - q^{12}t^{14} - q^{13}t^{15} - q^{15}t^{17}$.
\end{theorem}

\begin{corollary}[\Cref{corollary:fqt_count_terms_34}]\label{corollary:intro_fqt_count_terms_34}
There are rational functions $A_b,B_b,C_b,D_b,E_b \in \mathbb{Q}(q^{1/12})$ depending only on $b \pmod{24}$ such that $$N_{4,q,b} = (A_bb + B_b)q^b + C_bq^{5b/6} + (D_bb + E_b)q^{3b/4}$$ for all $q,b$.
\end{corollary}

If we let $N_{3,q,b}$ be the number of isomorphism classes of cubic $\mathbb{F}_q[t]$-algebras of discriminant $q^b$, also inversely weighted by automorphisms over $\mathbb{F}_q[t]$, then $$\sum_{b = 0}^{\infty}N_{3,q,b}t^b = I_3(q^{-1},qt) = \frac{1 + qt + q^2t^2 + q^3t^3 + q^4t^4}{(1 - q^2t^2)(1 - q^5t^6)}.$$ Though this can be computed using our method in the same way as $N_{4,q,b}$, there is a much easier and more direct method to compute $N_{3,q,b}$, which is described in \Cref{remark:direct_counting}.

Using $\mathbb{G}_m$-localization and Poincar\'{e} duality, we show that the above formulas for $N_{3,q,b}$ and $N_{4,q,b}$ come from the following local counts, which fall directly out of Igusa's computations of $I_d$.

\begin{proposition}[Special case of \Cref{corollary:local_cover_count}]\label{proposition:intro_local_cover_count}
\begin{enumerate}[(a)]
\item Let $N_{3,q,b}^{\loc}$ be the number of isomorphism classes of cubic $\mathbb{F}_q[[t]]$-algebras of discriminant $q^b$, inversely weighted by automorphisms over $\mathbb{F}_q[[t]]$. Then $$\sum_{b = 0}^{\infty}N_{3,q,b}^{\loc}t^b = I_3(q,t) = \frac{1 + t + t^2 + t^3 + t^4}{(1 - t^2)(1 - qt^6)}.$$

\item Let $N_{4,q,b}^{\loc}$ be the number of isomorphism classes of quartic $\mathbb{F}_q[[t]]$-algebras with cubic resolvent of discriminant $q^b$, inversely weighted by automorphisms over $\mathbb{F}_q[[t]]$. Then $$\sum_{b = 0}^{\infty}N_{4,q,b}^{\loc}t^b = I_4(q,t) = \frac{f(q,t)}{(1 - t)(1 - t^2)(1 - qt^6)(1 - q^2t^8)(1 - q^3t^{12})},$$ where $f(q,t) = 1 + t^2 + t^3 + t^4 - 2t^5 + 2qt^6 + (q - 1)t^7 + qt^8 - qt^9 + (q - 1)qt^{10} - 2qt^{11} + 2q^2t^{12} - q^2t^{13} - q^2t^{14} - q^2t^{15} - q^2t^{17}$.
\end{enumerate}
\end{proposition}

In \cite[Problem 5]{ggw_problems_arith_top}, Venkatesh asks the following question: ``What is the topological meaning of secondary terms appearing in asymptotic counts in number theory?'' The example he considers is the counting problem for cubic extensions of $\mathbb{Q}$, where the number of cubic fields of absolute discriminant less than $X$ is currently known to be $CX + KX^{5/6} + O(X^{2/3 + \epsilon})$ for constants $C,K$ by Bhargava-Taniguchi-Thorne \cite{btt_improved_error}. The terms $X$ and $X^{5/6}$ can be seen in the less refined count of cubic algebras of bounded discriminant, which was done much earlier by Shintani \cite{s_on_dirichlet}.

From the ``local-global duality'' above, we can provide one answer to Venkatesh's question: the lower order terms in the global counts correspond to the higher order terms in the local counts. To put this answer in more geometric terms, we will show (\Cref{proposition:fbd_k_pts}) that the Igusa zeta function $I_3$ is roughly the generating function for the point counts of the central fibers $\mathcal{F}_b^3$ of the branch maps $\mathcal{R}_{(\mathbb{P}^1;\infty),b}^3 \to \Sym^b(\mathbb{A}^1)$. The $1 - qt^6$ in the denominator of $I_3(q,t)$ tells us that the dimension of $\mathcal{F}_b^3$ is roughly $\frac{b}{6}$. This means that the leading term $q^{b/6}$ in the point counts of the central fibers corresponds under Poincar\'{e} duality to the $q^{5b/6}$ term in the count of cubic $\mathbb{F}_q[t]$-algebras of discriminant $q^b$ \cite{i_b_functions}.

\begin{remark}\label{remark:deopurkar_fb3}
The fact that the central fibers $\mathcal{F}_b^3$ have dimension roughly $\frac{b}{6}$ was actually observed by Deopurkar in \cite[\S4]{d_modular}, where he considers the stratification of $\mathcal{F}_b^3$ by Maroni invariant and shows that the strata are quotients of affine spaces by finite groups by describing all the possible generically \'{e}tale triple covers of the formal disc. This approach has the advantage of being much more explicit than our computation of the dimension, which just uses the point count given by the Igusa zeta function. However, it is much more difficult to generalize this explicit approach to describe all the points of $\mathcal{F}_b^4$, which correspond to generically \'{e}tale quartic covers of the formal disc with cubic resolvent. Our approach of expressing the generating function of the point count as an Igusa zeta function lets us compute not only the dimension of $\mathcal{F}_b^4$ (roughly $\frac{b}{4}$) but also the cohomology and the exact point count.
\end{remark}

\subsubsection{Quantum algebra}\label{section:quantum_results}
In contrast to the strategy of Ellenberg-Venkatesh-Westerland \cite{evw_homological_stability} and Ellenberg-Tran-Westerland \cite{etw_fnf_cells}, we use number theory and algebraic geometry to deduce something on the topological/quantum algebraic side. The theorem that allows us to make this translation is due to Kapranov-Schechtman:

\begin{theorem}[{\cite[Corollary 3.3.5]{ks_shuffle_algs}}]\label{theorem:intro_ks_nichols_ic}
There is an isomorphism $$\Ext_{T_{!*}(V(c,-1))}^{n - j,n}(k,k) \cong H^j(\Sym^n(\mathbb{A}^1),j_{!*}\pi_*\underline{k}_{\Hur_{G,n}^c})^{\vee}.$$
\end{theorem}

Here, $G$ is an arbitrary finite group, $T_{!*}(V(c,-1))$ is a braided graded Hopf algebra associated to a union of conjugacy classes $c \subset G$ known as the \emph{Nichols algebra}, and $\pi: \Hur_{G,n}^c \to \Conf^n(\mathbb{A}^1)$ is a covering space corresponding to a natural action of the $n$th braid group $\mathbf{B}_n$ on $c^{\times n}$. The Hurwitz space $\Hur_{G,n}^c$ parametrizes marked $G$-covers of the (closed) unit disc with $n$ branch points whose local monodromy is in $c$, where a marking is a $G$-equivariant isomorphism between the fiber above 1 and $G$. Our Hurwitz stacks $\mathcal{H}_{(C;\sigma),b}^d$, which correspond to the case $G = \mathbf{S}_d$ and $c = \tau_d$, the class of transpositions, are unmarked. To make better use of the theorem of Kapranov-Schechtman, we construct marked Hurwitz spaces as certain $\mathbf{S}_d$-torsors $\mathcal{H}_{M,(C;\sigma),b}^d \to \mathcal{H}_{(C;\sigma),b}^d$ along with marked compactifications $\mathcal{R}_{M,(C;\sigma),b}^d \to \mathcal{R}_{(C;\sigma),b}^d$, and we show that the space $\mathcal{H}_{M,(\mathbb{P}^1;\infty),b}^d(\mathbb{C})$ is homeomorphic to $\Hur_{\mathbf{S}_d,b}^{\tau_d}$.

Our result concerns the cohomology of the Nichols algebra $\mathfrak{B}_4$ associated to the Yetter-Drinfeld module $k\tau_4 \otimes \sgn \in \mathcal{YD}_{\mathbf{S}_4}^{\mathbf{S}_4}$, where $k$ is a field of characteristic 0 and $\sgn$ is the sign representation of $\mathbf{S}_4$. This 576-dimensional Nichols algebra is significant because it is one of very few known examples of finite-dimensional elementary Nichols algebras. The cohomology of the previous Nichols algebra on this list, the 12-dimensional $\mathfrak{B}_3$, was computed algebraically by \c{S}tefan-Vay in \cite{sv_cohomology_fk_alg}. We prove the following theorem using the Kapranov-Schechtman theorem and the density/smallness from \Cref{theorem:intro_dense_small34}. We can also recover the $\mathbf{S}_3$-invariant part of the \c{S}tefan-Vay computation by the same method and recover the entire computation by applying $\mathbb{G}_m$-localization to the marked compactification (\Cref{remark:sv_compare_future}).

\begin{theorem}[\Cref{theorem:inv_coh_bd_34} for $d = 4$]\label{theorem:intro_inv_coh_bd_4}
\begin{enumerate}[(a)]
\item There are $\mathbf{S}_4$-equivariant isomorphisms $$\Ext_{\mathfrak{B}_4}^{a,b}(k,k) \cong H^{b - a}(\mathcal{R}_{M,(\mathbb{P}^1;\infty),b}^4(\mathbb{C}),k)^{\vee} \cong H^{b - a}(\mathcal{F}_{M,b}^4(\mathbb{C}),k)^{\vee},$$ where the $\mathbf{S}_4$-action on the $\Ext$ groups is the geometric action.

\item Let $E^{a,b} = \dim\Ext_{\mathfrak{B}_4}^{a,b}(k,k)^{\mathbf{S}_4}$. Then $$\sum_{a = 0}^{\infty}\sum_{b = 0}^{\infty}E^{a,b}q^at^b = I_4(q^{-2},qt) = \frac{f(q^{-2},qt)}{(1 - qt)(1 - q^2t^2)(1 - q^4t^6)(1 - q^4t^8)(1 - q^6t^{12})},$$ where $f(q^{-2},qt) = 1 + q^2t^2 + q^3t^3 + q^4t^4 - 2q^5t^5 + 2q^4t^6 - (q^2 - 1)q^5t^7 + q^6t^8 - q^7t^9 - (q^2 - 1)q^6t^{10} - 2q^9t^{11} + 2q^8t^{12} - q^9t^{13} - q^{10}t^{14} - q^{11}t^{15} - q^{13}t^{17}$.
\end{enumerate}
\end{theorem}

We clarify what we mean by ``geometric action'' in part (a). There are two $\mathbf{S}_4$-actions on $k\tau_4 \otimes \sgn$. The first is the \emph{standard action}, which is the action on $k\tau_4 \otimes \sgn$ as a Yetter-Drinfeld module: for $t \in \tau_4$ and $g \in \mathbf{S}_4$, we have $(g,t) \mapsto \sgn(g)gtg^{-1}$. The second is the \emph{geometric action}: for $t \in \tau_4$ and $g \in \mathbf{S}_4$, we have $(g,t) \mapsto gtg^{-1}$. Each of these actions induces an $\mathbf{S}_4$-action on $\mathfrak{B}_4$. Although the standard action may seem more natural from an algebraic point of view, the geometric action is the one that appears in our work since it corresponds to the $\mathbf{S}_4$-action on $\mathcal{H}_{M,(C;\sigma),b}^d$.

Unfortunately, we cannot say anything about the Yoneda ring structure of $\Ext_{\mathfrak{B}_4}^{*,*}(k,k)^{\mathbf{S}_4}$ with our methods, as the multiplication on the Hurwitz space side is topological rather than algebro-geometric. However, we do note that the terms in the denominator suggest the presence of polynomial generators in bidegrees $(1,1)$, $(2,2)$, $(4,6)$, $(4,8)$, and $(6,12)$. In the $d = 3$ case, the terms in the denominator $(1 - q^2t^2)(1 - q^4t^6)$ indeed appear as polynomial generators in $\Ext_{\mathfrak{B}_3}^{2,2}$ and $\Ext_{\mathfrak{B}_3}^{4,6}$ by the computation of the cohomology ring by \c{S}tefan-Vay \cite{sv_cohomology_fk_alg}, so it is likely that the $\mathfrak{B}_4$ picture looks like this too.

We get another answer to Venkatesh's question about the topological origin of secondary terms in arithmetic counts: the secondary terms correspond to polynomial generators of the $\mathbf{S}_d$-invariant cohomology of $\mathfrak{B}_d$. In particular, the secondary term for the count of cubic fields/algebras comes from the generator in $\Ext_{\mathfrak{B}_3}^{4,6}$. Since the denominator terms in the Nichols algebra cohomology really come from the denominator of the Igusa zeta function, this answer is secretly the same as the one about local counts.

\subsection{Motivation: Arithmetic statistics and Nichols algebras}\label{section:arith_nichols}
This paper originated from a question asked by Ellenberg-Tran-Westerland (ETW) in their paper proving the upper bound in the Malle conjecture over $\mathbb{F}_q(t)$ \cite{etw_fnf_cells}. In that paper, the authors study the cohomology of the quantum shuffle algebra $T_*(V(c,-1))$, which contains the Nichols algebra $T_{!*}(V(c,-1))$ as a natural subalgebra. Whereas the Nichols algebra cohomology is dual to the cohomology of the intersection complex $j_{!*}\pi_*\underline{k}_{\Hur_{G,n}^c}$ on $\Sym^n(\mathbb{A}^1)$, the quantum shuffle algebra cohomology is dual to the cohomology of the Hurwitz space itself, so its associated trace function counts extensions of $\mathbb{F}_q(t)$. In contrast to the simplicity of its topological and arithmetic interpretation, the quantum shuffle algebra is generally much more complicated than the Nichols algebra and much harder to work with.

In \cite{etw_fnf_cells}, the authors ask whether the trace of Frobenius on the Nichols algebra/intersection cohomology counts anything, particularly in the case $(G,c) = (\mathbf{S}_3,\tau_3)$, where the cohomology has been computed by \c{S}tefan-Vay \cite{sv_cohomology_fk_alg}. In a similar vein, in \cite{ks_shuffle_algs}, Kapranov-Schechtman remark that it would be interesting to describe the intersection cohomology as the cohomology of some natural (partial) compactification of $\Hur_{G,n}^c$. By general facts about perverse sheaves, the cohomology of such a compactification is the same as the intersection cohomology as long as it is smooth and admits a small extension of the branch map to $\Sym^n(\mathbb{A}^1)$. There are many previously studied Hurwitz space compactifications such as twisted admissible covers compactifications \cite{acv_twisted_bundles} and Deopurkar's quasimap compactifications \cite{d_compact}, but we have to construct new compactifications in this paper because these prior compactifications either fail to be small (twisted admissible covers) or smooth (Deopurkar's for $d \ge 4$).

We note that the Nichols algebra $T_{!*}(V(c,-1))$ and corresponding intersection complexes $j_{!*}\pi_*\underline{k}_{\Hur_{G,n}^c}$ on $\Sym^n(\mathbb{A}^1)$ only depend on the pair $(G,c)$. This means that the cohomology and $\mathbb{F}_q$-point counts of a smooth small compactification are intrinsic to the pair $(G,c)$ as well. In other words, whatever counting problem is governed by a smooth small compactification of $\Hur_{G,n}^c$ is a canonical approximation to the problem of counting $G$-extensions of $\mathbb{F}_q(t)$. By constructing the smooth small compactifications in this paper, we show that the approximations for the cases $(G,c) = (\mathbf{S}_d,\tau_d)$ with $3 \le d \le 5$ are rank $d$ algebras with resolvent. In other words, the parametrizations of low degree algebras used in arithmetic statistics, which seem to fall out of the classification of prehomogeneous vector spaces, are actually intrinsic to the groups $\mathbf{S}_3,\mathbf{S}_4,\mathbf{S}_5$. This suggests some heuristics relating arithmetic statistics, Nichols algebras, and geometry: \begin{itemize}
\item Heuristic 1: Constructing smooth small compactifications of $\Hur_{G,n}^c$ may tell us how to parametrize $G$-extensions. It would be very interesting to try to find smooth small compactifications for other $(G,c)$ or to prove that such compactifications cannot exist. We remark that the approach taken in this paper of compactifying Hurwitz spaces by spaces of quasimaps to $BG$ cannot lead to anything beyond what we do in this paper with prehomogeneous vector spaces (see \Cref{remark:quotient_impossible} for details), but it may still be possible to find smooth small compactifications by other means.

\item Heuristic 2: The complexity of the count of $G$-extensions should match the complexity of the relevant Nichols algebras over $G$. In particular, when the relevant Nichols algebra is finite-dimensional, the count should be simpler in some way. For instance, for $G$ abelian, the Nichols algebras are exterior algebras, which are finite-dimensional and have extremely simple cohomology.

In this paper, we are concerned with Nichols algebras over symmetric groups $\mathbf{S}_d$. One observation that we have no explanation for is the similarity between the following two facts: \begin{enumerate}[(1)]
\item $\mathbf{S}_d$ can be realized as the generic stabilizer of a prehomogeneous vector space precisely when $d \le 5$.

\item The Nichols algebras $\mathfrak{B}_d$ for $(G,c) = (\mathbf{S}_d,\tau_d)$ are finite-dimensional when $d \le 5$ and conjectured to be infinite-dimensional for $d \ge 6$.
\end{enumerate}
This paper matches prehomogeneous vector spaces and Nichols algebras by equating the coefficients of the Igusa zeta function for $(G_d,V_d)$ with the $\mathbf{S}_d$-invariant cohomology of $\mathfrak{B}_d$. Given this connection, the non-existence of prehomogeneous vector spaces with generic stabilizer $\mathbf{S}_d$ for $d \ge 6$ can be seen as evidence for the conjectured infinite-dimensionality of $\mathfrak{B}_d$ for $d \ge 6$. However, our connection is rather indirect, so many aspects remain mysterious, e.g. what it means geometrically if the Nichols algebra is finite-dimensional. For more context about the algebras $\mathfrak{B}_d$, see \Cref{section:nichols_bd}.
\end{itemize}

\subsection{Relation to previous work}\label{section:relation_prev}
We owe a large debt to Anand Deopurkar and his thesis work. Much of this paper came from trying to interpret the geometry of his stacks of triple covers arithmetically and find higher degree compactifications with similar properties. Many details of our construction of stacks of degree $d$ covers with resolvents carry over identically from his construction of stacks of degree $d$ covers \cite{d_compact}. Both of us construct stacks of quasimaps to $B\mathbf{S}_d$. The difference is that we consider $B\mathbf{S}_d$ as an open substack of $\mathcal{X}_d$ for $d \in \{3,4,5\}$, whereas Deopurkar uses the stacks $\Covers_d$, which exist for all $d$ and fail to be smooth for $d \ge 4$. There are maps $\mathcal{X}_d \to \Covers_d$ inducing maps from our stacks to Deopurkar's stacks; this map is an isomorphism for $d = 3$, so we end up constructing the same triple cover compactification in this case. When $d \ge 4$, Deopurkar's stacks are singular, and for our applications, we need our compactifications to be smooth. The contributions of this paper are to use resolvent data to construct resolutions of Deopurkar's stacks for $d \in \{4,5\}$ and to connect the geometry of these low degree Hurwitz spaces to the broader picture of prehomogeneous vector spaces and Nichols algebras.

Because we use the presentation of $B\mathbf{S}_d$ as the GIT quotient stack $[V_d^{ss}/G_d]$, our construction over $\mathbb{C}$ is also a special case of the stacks of orbifold quasimaps constructed by Cheong-Ciocan-Fontanine-Kim \cite{cck_orbifold}. However, we do everything in mixed characteristic, so we rewrite everything.

We point out the recent work of Han-Park \cite{hp_enumerating_hyperelliptic} and Bejleri-Park-Satriano \cite{bps_height_moduli}, where the authors provide alternative answers to Venkatesh's question about secondary terms in separate contexts. In particular, Han-Park treat the $d = 2$ of the problem we study in this paper by enumerating quasi-admissible odd-degree hyperelliptic curves over $\mathbb{P}^1$. It would be interesting to see if their explanation for their lower order terms has the same underlying reason as our lower order terms. However, the techniques in both of these papers are totally different from ours. Their counting problems are equivalent to counting points of a mapping stack from $\mathbb{P}^1$ or a stacky $\mathbb{P}^1$ to a weighted projective space, so they are able to compute the motives of these stacks and extract the point counts. On the other hand, we count maps to $\mathcal{X}_d = [V_d/G_d]$, and our point counts are based on the point counts of the central fiber $\mathcal{F}_b^d$, which encode local data (e.g. the Igusa zeta function).

\subsection{Further questions}\label{section:further_questions}
We hope that the present work will serve as motivation for interested readers to compute the Igusa zeta function for $(G_5,V_5)$, as we will get many results for $d = 5$ if $I_5$ is computed (see \Cref{section:conj_igusa5} for specific conjectures). Igusa was unable to do so, but maybe the advances in both mathematics and technology will make this computation feasible nowadays. We point to the recent work of Tajima-Yukie \cite{ty_git_strat_iii} on the GIT stratification of $V_5$.

Another result that would be useful to study the case $d = 5$ would be a parametrization of quintic algebras with sextic resolvents over a general base analogous to Wood's theorem \cite{w_quartic_arbitrary} parametrizing quartic algebras with cubic resolvents over a general base. Because we have parametrizations of cubics and quartics over a general base, $\mathcal{R}^3$ and $\mathcal{R}^4$ genuinely are moduli stacks of cubic and quartic covers with resolvents. Right now, we do not know that $\mathcal{R}^5$ is actually a moduli stack of quintic covers with sextic resolvents.

\subsection{Outline of the paper}\label{section:outline_paper}
Readers who are interested only in the arithmetic or algebro-geometric aspects of this paper should start at \Cref{section:prehomog_covs}.

\Cref{section:nichols_algs} is a brief introduction to Nichols algebras, focusing on the examples $\mathfrak{B}_d$. \Cref{section:top_hurwitz} is a topological introduction to Hurwitz spaces, recapping the material of \cite[\S2]{evw_homological_stability} and stating the theorem of Kapranov-Schechtman. \Cref{section:prehomog_covs} discusses the prehomogeneous vector spaces $(G_d,V_d)$ and the relationship between $\mathcal{X}_d \coloneqq [V_d/G_d]$ and the stack of degree $d$ covers $\Covers_d$. \Cref{section:prehomog_local} translates the values of Igusa zeta functions to counts of $G_d(\mathcal{O}_K)$-orbits in $V_d(\mathcal{O}_K)$ for a local field $K$. \Cref{section:orbinodal_curves} is an introduction to orbinodal curves. Many important stacks and their properties can be found here. \Cref{section:covers_resolvents} contains the main geometric content of this paper. In it, we construct our stacks $\mathcal{R}^d$ and prove their important properties (properness of the branch morphism, smoothness, relative dimension over $\mathcal{M}$, etc.). \Cref{section:fibers_of_br} uses Igusa zeta functions to bound the dimensions of the fibers of the branch morphism $\br$. \Cref{section:related_stacks} constructs and proves properties of several important geometric objects derived from $\mathcal{R}^d$. These include Hurwitz stacks $\mathcal{H}^d$, marked variants of $\mathcal{R}^d$ and $\mathcal{H}^d$, and stacks $\mathcal{R}_{(C;\sigma),b}^d$ for a single pointed curve $(C;\sigma)$. We present some applications of our constructions to smoothability of curves in \Cref{section:app_smoothability}. \Cref{section:covers_of_a1} specializes to the case of $(C;\sigma) = (\mathbb{P}^1;\infty)$. This is where we use $\mathbb{G}_m$-localization to prove all of our main results about counts of $\mathbb{F}_q[t]$-algebras and cohomology of Nichols algebras.

\subsection{Acknowledgments}\label{section:acknowledgments}
I am indebted to my advisor Will Sawin for suggesting the problem that led to this paper and for introducing me to the beautiful theory of prehomogeneous vector spaces and parametrizations of low degree covers. I am grateful to Will, Aaron Landesman, Jordan Ellenberg, and Craig Westerland for providing helpful comments on a draft of this paper. I would like to thank Melissa Liu for introducing me to quasimaps and the utility of $\mathbb{G}_m$-localization early on in graduate school.  Finally, I thank Will, Aaron, Jordan, Craig, Andres Fernandez Herrero, Johan de Jong, and Anand Deopurkar for very helpful conversations related to the work and Alan Peng for valuable typesetting advice.

\section{Notation and conventions}\label{section:notation_conventions}
\begin{enumerate}[(1)]
\item $\mathbf{S}_d$ will denote the $d$th symmetric group. In the context of $\mathbf{S}_d$, we will denote the conjugacy class of transpositions in $\mathbf{S}_d$ by $\tau_d$.

\item $\mu_n$ will denote the group scheme of $n$th roots of unity.

\item Given a groupoid $X$, let $X/\sim$ denote its set of isomorphism classes. If $X$ has finitely many isomorphism classes, $\# X$ will denote the sum $\sum_x\frac{1}{\#\Aut(x)}$, where $x$ ranges over the isomorphism classes of $X$. We will often call $\# X$ the count \emph{inversely weighted by automorphisms}.

\item Given a finite rank locally free sheaf $\mathcal{E}$ on a scheme $X$, let $\mathbb{P}\mathcal{E} \coloneqq \Proj\Sym^{\bullet}\mathcal{E}$.

\item We will use the term \emph{finite cover} to refer to a finite locally free morphism of constant positive degree. We will use \emph{degree $d$ cover} to refer to a finite cover of degree $d$. For a finite cover $f: Y \to S$, we will often write $\mathcal{O}_Y$ in place of $f_*\mathcal{O}_Y$.

\item Whenever we have a \emph{curve over a scheme $S$} written $C \to S$, we mean that $C$ is an algebraic space and $C \to S$ is a flat proper morphism whose geometric fibers are purely 1-dimensional.

\item For a separated finite type scheme $X/S$, $\PConf^n(X)$ will denote the ordered configuration space given by removing the image of all the diagonal embeddings $X^{n - 1} \xhookrightarrow{} X^n$ (a.k.a. the \emph{big diagonal}) from $X^n$,where the fiber product is over $S$). Then $\Conf^n(X)$ will denote the unordered configuration space $\PConf^n(X)/S_n$.

\item We can work with $\ell$-adic and perverse sheaves on Deligne-Mumford stacks the same way we would on schemes, thanks to the work of Olsson \cite{o_sheaves_artin} and Laszlo-Olsson \cite{lo_six_operations_i,lo_six_operations_ii,lo_perverse_artin}. We always take $\ell$ to be a prime different from the characteristic of the base field.

\item For an open immersion of a smooth subvariety $j: U \xhookrightarrow{} X$ and a local system $\mathcal{L}$ on $U$, $j_{!*}\mathcal{L}$ will denote the shifted perverse sheaf $j_{!*}(\mathcal{L}[\dim X])[-\dim X]$.
\end{enumerate}

\section{Nichols algebras}\label{section:nichols_algs}
This section is a brief introduction to Nichols algebras. For a more detailed introduction, see \cite[\S2]{etw_fnf_cells} and \cite[\S2-3]{ks_shuffle_algs}. For a general survey, see \cite{as_pointed_hopf}.

For this section, let $k$ be a field of characteristic 0. All tensor products of $k$-vector spaces will be over $k$.

\subsection{Braided vector spaces and braided monoidal categories}\label{section:braided_vs_mc}
A \emph{braided vector space} is a pair $(V,c)$, where $V$ is a finite-dimensional $k$-vector space and $c: V \otimes V \xrightarrow[]{\sim} V \otimes V$ is an automorphism satisfying the braid equation $$(c \otimes \id)(\id \otimes c)(c \otimes \id) = (\id \otimes c)(c \otimes \id)(\id \otimes c)$$ as automorphisms of $V \otimes V \otimes V$.

Recall that the \emph{$n$th braid group} $\mathbf{B}_n$ is a group with presentation \begin{align*}
\mathbf{B}_n \coloneqq \langle\sigma_1,\ldots,\sigma_{n - 1}|\sigma_i\sigma_{i + 1}\sigma_i &= \sigma_{i + 1}\sigma_i\sigma_{i + 1} ,\forall 1 \le i \le n - 2; \\
\sigma_i\sigma_j &= \sigma_j\sigma_i,\forall |i - j| \ge 2\rangle.
\end{align*}
A braiding $c$ on a vector space $V$ naturally makes $V^{\otimes n}$ a $\mathbf{B}_n$-representation for each $n$: each $\sigma_i$ acts on $V^{\otimes n}$ as $1 \otimes \cdots \otimes 1 \otimes c \otimes 1 \otimes 1$ with $c$ acting on the $i$th and $(i + 1)$th copies of $V$, and the braid equation ensures that the braid relations between the $\sigma_i$ are satisfied.

Let $(\mathcal{V},\otimes,1,R)$ be a braided monoidal $k$-linear abelian category where $\otimes$ is exact with respect to both entries. Here, $(\mathcal{V},\otimes,1)$ is the underlying monoidal category, and for any $V_1,V_2 \in \mathcal{V}$, we have the $R$-matrix $$R_{V_1,V_2}: V_1 \otimes V_2 \xrightarrow[]{\sim} V_2 \otimes V_1$$ satisfying the braiding axioms \cite[Definition 3.2.1]{hs_hopf_algs}. The braiding axioms imply that $R_{V,V}: V \otimes V \xrightarrow[]{\sim} V \otimes V$ satisfies the braid equation. If $\mathcal{V}$ is equipped with a faithful $k$-linear strict monoidal functor to the category of $k$-vector spaces, then the underlying vector space of an object $V$ is a braided vector space with braiding $c = R_{V,V}$.

The main examples of braided monoidal categories relevant to us are categories of Yetter-Drinfeld modules:

\begin{definition}\label{definition:yd_module}
For a group $G$, a (right, right) \emph{Yetter-Drinfeld module}\footnote{Yetter-Drinfeld modules can be defined over any Hopf algebra \cite[\S3.4]{hs_hopf_algs}, but we will only deal with group algebras (of finite groups).} $M$ is a $G$-graded right $G$-module with the property that $M_gh \subset M_{h^{-1}gh}$ for all $g,h \in G$. A map of Yetter-Drinfeld modules is a $G$-equivariant map respecting the $G$-grading.

A \emph{graded Yetter-Drinfeld module} is a $\mathbb{Z}$-graded collection of Yetter-Drinfeld modules. A map of graded Yetter-Drinfeld modules is a map of Yetter-Drinfeld modules respecting the $\mathbb{Z}$-grading.
\end{definition}

Note that if we have a Yetter-Drinfeld module $M$ and a character $\psi: G \to k^{\times}$, we can define the twisted Yetter-Drinfeld module $M^{\psi}$, in which the $G$-action is multiplied by $\psi$.

\begin{example}\label{example:yd_c_times_n}
The most important example of a Yetter-Drinfeld module we will work with is $M = kc$, where $c$ is a union of conjugacy classes in $G$. Here, the basis element corresponding to an element $g \in c$ is in $M_g$.
\end{example}

\begin{definition}\label{definition:yd_braidings}
We define two braided monoidal structures on the category of Yetter-Drinfeld modules:
\begin{enumerate}[(a)]
\item For $m \in M$ and $n \in N_g$, we will write $m^n \coloneqq mg$. Then the $R$-matrix $R_{M,N}: M \otimes N \to N \otimes M$ is defined by $R_{M,N}(m \otimes n) = n \otimes m^n$, where $m \in M$ and $n \in N_g$ for some $g \in G$. We will use $\mathcal{YD}_G^G$ to denote the braided monoidal category of Yetter-Drinfeld modules with braiding $R$.

\item We define an alternate braided monoidal structure on $\mathcal{YD}_G^G$: $R_{M,N}^{\epsilon}(m \otimes n) = -n \otimes m^n$. We will use $\mathcal{YD}_{G,\epsilon}^G$ to denote the braided monoidal category obtained from the braiding $R^{\epsilon}$.
\end{enumerate}
\end{definition}

\subsection{Braided bialgebras}\label{section:bialgs_braided}
With the above examples in mind, we return to the setting of an arbitrary braided monoidal $k$-linear abelian category $\mathcal{V}$. We will refer to monoid objects $(A,\mu,\eta)$ as algebras and comonoid objects $(A,\Delta,\epsilon)$ as coalgebras. Given two algebras $A,B \in \mathcal{V}$, we define $A \underline{\otimes} B$ as the algebra having underlying object $A \otimes B$, multiplication $$\mu_{A \underline{\otimes} B}: A \otimes B \otimes A \otimes B \xrightarrow[]{1 \otimes R_{B,A} \otimes 1} A \otimes A \otimes B \otimes B \xrightarrow[]{\mu_A \otimes \mu_B} A \otimes B$$ and unit $$\eta_{A \underline{\otimes} B}: 1 \xrightarrow[]{\sim} 1 \otimes 1 \xrightarrow[]{\eta_A \otimes \eta_B} A \otimes B.$$

\begin{definition}\label{definition:braided_bialg}
A \emph{braided bialgebra} in $\mathcal{V}$ is a tuple $(A,\mu,\eta,\Delta,\epsilon)$, where \begin{enumerate}[(1)]
\item $(A,\mu,\eta)$ is an algebra.

\item $(A,\Delta,\epsilon)$ is a coalgebra.

\item $\Delta: A \to A \underline{\otimes} A$ and $\epsilon: 1 \to A$ are morphisms of algebras.
\end{enumerate}
A braided bialgebra is a \emph{braided Hopf algebra} if there exists a morphism $\chi: A \to A$ such that $$\eta\epsilon = \mu(1 \otimes \chi)\Delta = \mu(\chi \otimes 1)\Delta.$$ Such a $\chi$ is called an antipode, and it is unique if it exists.

A braided graded bialgebra is a braided bialgebra $A$ equipped with an $\mathbb{N}$-grading $A = \bigoplus_{n = 0}^{\infty}A^n$ that makes $(A,\mu,\eta)$ a graded algebra and $(A,\Delta,\epsilon)$ a graded bialgebra.
\end{definition}

Given any object $V$ of a braided monoidal category $\mathcal{V}$, there are three braided graded Hopf algebras naturally associated to it: \begin{enumerate}[(1)]
\item The \emph{tensor algebra} $T_!(V)$ has underlying object $$T_!(V) = \bigoplus_{n = 0}^{\infty}V^{\otimes n}$$ with multiplication induced by the identity map $$V^{\otimes m} \otimes V^{\otimes n} \to V^{\otimes m + n}.$$ We give $T_!(V)$ the unique coalgebra structure that makes $T_!^1(V) = V$ \emph{primitive}, i.e. $\Delta|_{T_!^1(v)} = 1 \otimes \id + \id \otimes 1: V \to V \otimes V$.

\item The \emph{cotensor algebra} $T_*(V)$ has underlying object $$T_*(V) = \bigoplus_{n = 0}^{\infty}V^{\otimes n}$$ with comultiplication induced by the identity map $$V^{\otimes m + n} \to V^{\otimes m} \otimes V^{\otimes n}.$$ To define the multiplication, we first recall that a \emph{shuffle} $w: \{1,\ldots,m\} \sqcup \{1,\ldots,n\} \to \{1,\ldots,m + n\}$ is a bijection that respects the orders on $\{1,\ldots,m\}$ and $\{1,\ldots,n\}$. For each shuffle $w$, there is a natural morphism $R_w: V^{\otimes m} \otimes V^{\otimes n} \to V^{\otimes m + n}$ that uses $R_{V,V}$ to shuffle the copies of $V$ according to $w$. The multiplication $\star$ on $T_*(V)$, called the \emph{shuffle product}, is defined as $\star = \sum_wR_w$. In \cite{etw_fnf_cells}, $T_*(V)$ is called the \emph{quantum shuffle algebra}, and its cohomology is bounded, related to the cohomology of Hurwitz spaces (\Cref{theorem:braid_qs_coh}), and used to prove the upper bound in Malle's conjecture over $\mathbb{F}_q(t)$.

\item Because $T_!(V)$ is the free algebra on $V$ (or alternatively, because $T_*(V)$ is the free coalgebra on $V$), there is a natural map $T_!(V) \to T_*(V)$. This is a map of coalgebras because $T_*(V)^1 = V$ is primitive and $T_!(V)$ is generated by the primitives $T_!(V)^1 = V$. Hence, we have a map of braided Hopf algebras $T_!(V) \to T_*(V)$. We define the \emph{Nichols algebra} $T_{!*}(V)$ as the braided Hopf algebra $$T_{!*}(V) = \im(T_!(V) \to T_*(V)).$$ 
\end{enumerate}
Details (e.g. proofs that these are braided Hopf algebras) can be found in \cite[\S2.4, \S3.1]{ks_shuffle_algs}. Of course, if $\mathcal{V}$ is a category of vector spaces (e.g. $\mathcal{YD}_G^G$), then all the above braided Hopf algebras have underlying $k$-algebras/coalgebras. It is possible to define a Nichols algebra of any braided vector space in the same way as above, and the underlying $k$-algebra/coalgebra of a Nichols algebra $T_{!*}(V) \in \mathcal{V}$ is determined by the underlying braided vector space of $V$.

The notation comes from the correspondence of Kapranov-Schechtman \cite[Theorem 3.3.1, Theorem 3.3.3]{ks_shuffle_algs} between certain bialgebras and certain sequence of perverse sheaves on $\Sym^n(\mathbb{A}^1)$. According to this correspondence, an object $V$ corresponds to a sequence of local systems $\mathcal{L}_n(V[1])$ on $\Conf^n(\mathbb{A}^1)$. If we let $j_n: \Conf^n(\mathbb{A}^1) \xhookrightarrow{} \Sym^n(\mathbb{A}^1)$ denote the usual open immersion, then $T_!(V)$ (resp. $T_*(V)$) corresponds to the sequence of perverse sheaves $j_{n!}\mathcal{L}_n(V[1])[n]$ (resp. $j_{n*}\mathcal{L}_n(V[1])[n]$), and the Nichols algebra $T_{!*}(V)$ corresponds to the intermediate extension $$j_{!*}\mathcal{L}_n(V[1])[n] = \im(j_{n!}\mathcal{L}_n(V[1])[n] \to j_{n*}\mathcal{L}_n(V[1])[n]).$$ We will be interested in the cohomological consequences of this correspondence (\Cref{theorem:ks_nichols_ic}).

\subsection{Quantum shuffle algebras and homology of braid groups}\label{section:qsa_homology}
The Nichols algebras and quantum shuffle algebras that are relevant to Hurwitz spaces and arithmetic statistics are those arising from the Yetter-Drinfeld modules $V(c,-1) \coloneqq kc \in \mathcal{YD}_{G,\epsilon}^G$ (note the sign $\epsilon$). The reason is the following key result, proven independently by Ellenberg-Tran-Westerland \cite[Corollary 3.7]{etw_fnf_cells} and Kapranov-Schechtman \cite[Corollary 3.3.4]{ks_shuffle_algs}:

\begin{theorem}[{\cite{etw_fnf_cells,ks_shuffle_algs}}]\label{theorem:braid_qs_coh}
Consider a finite-dimensional Yetter-Drinfeld module $V_{\epsilon} \in \mathcal{YD}_{G,\epsilon}^G$. Then $$H_j(\mathbf{B}_n,(V_{\epsilon}^{\vee})^{\otimes n} \otimes \sgn) \cong \Ext_{T_*(V_{\epsilon})}^{n - j,n}(k,k),$$ where $\sgn: \mathbf{B}_n \to \{\pm 1\}$ is the sign representation. 
\end{theorem}

When $V_{\epsilon} = V(c,-1)$ for some $c \subset G$, then the left-hand side $H_j(\mathbf{B}_n,(V_{\epsilon}^{\vee})^{\otimes n} \otimes \sgn) \cong H_j(\mathbf{B}_n,kc^{\times n})$ is the homology of the Hurwitz space $\Hur_{G,n}^c$, which we will discuss in \Cref{section:top_hurwitz}. The quantum shuffle algebras $T_*(V(c,-1))$ are very complicated in general but their cohomology groups have a straightforward arithmetic interpretation in terms of point counts of Hurwitz spaces. On the other hand, the Nichols algebras $T_{!*}(V(c,-1))$ are simpler, but we do not yet know when their cohomology groups admit point counting interpretations.

\subsection{The Nichols algebras $\mathfrak{B}_d$}\label{section:nichols_bd}
The main objects from quantum algebra featured in this paper are the Nichols algebras $\mathfrak{B}_d \coloneqq T_{!*}(V(\tau_d,-1))$ for $3 \le d \le 5$, which are rather special.

First off, for arbitrary $d \ge 3$, we can consider $\mathfrak{B}_d$ as the underlying $k$-algebra of the Nichols algebra $T_{!*}(k\tau_d \otimes \sgn) \in \mathcal{YD}_{\mathbf{S}_d}^{\mathbf{S}_d}$, where $\sgn: \mathbf{S}_d \to \{\pm 1\}$ is the sign representation, since $V(\tau_d,-1) = k\tau_d \in \mathcal{YD}_{\mathbf{S}_d,\epsilon}^{\mathbf{S}_d}$ and $k\tau_d \otimes \sgn \in \mathcal{YD}_{\mathbf{S}_d}^{\mathbf{S}_d}$ have the same underlying braided vector space. In particular, $\mathfrak{B}_d$ is an \emph{elementary} Nichols algebra: the Yetter-Drinfeld module $k\tau_d \otimes \sgn$ is finite-dimensional and absolutely irreducible, and its support (the subset of $\mathbf{S}_d$ with nonzero graded pieces) generates $\mathbf{S}_d$. For $3 \le d \le 5$, $\mathfrak{B}_d$ is finite-dimensional, and for $d \ge 6$, it is unknown but expected that $\mathfrak{B}_d$ is infinite-dimensional.

The fact that $\mathfrak{B}_d$ is finite-dimensional for $d \in \{3,4,5\}$ is significant, as the list of known finite-dimensional elementary Nichols algebras \cite[Table 9.1]{hlv_nichols_cubic} is extremely short. Letting $(a)_t = 1 + t + \cdots + t^{a - 1}$, we present some of the numerical data associated to $\mathfrak{B}_d$ for $3 \le d \le 5$: \begin{itemize}
\item $\mathfrak{B}_3$ has dimension 12, Hilbert series $(2)_t^2(3)_t$, and top degree 4.

\item $\mathfrak{B}_4$ has dimension 576, Hilbert series $(2)_t^2(3)_t^2(4)_t^2$, and top degree 12.

\item $\mathfrak{B}_5$ has dimension 8294400, Hilbert series $(4)_t^4(5)_t^2(6)_t^4$, and top degree 40.
\end{itemize}
We will see in \Cref{section:prehomog_covs} that the numbers 4, 12, and 40 appear as the dimensions of certain prehomogeneous vector spaces. We do not know why these numbers match or what the other numerical data of $\mathfrak{B}_d$ correspond to on the prehomogeneous side. These numbers appear as the number of indecomposable representations of the preprojective algebra of type $A_{d - 1}$, as noted by Majid \cite[\S3]{m_noncommutative_differentials}, and equivalently as the number of cluster variables in the coordinate ring of the group of upper unitriangular matrices $N_d \subset \SL_d$ by work of Gei\ss-Leclerc-Schr\"{o}er \cite{gls_rigid_preprojective}.

Another distinguishing feature of $\mathfrak{B}_d$ for $d \in \{3,4,5\}$ is that the algebras are known to be defined by quadratic relations in this range and not known to be quadratic outside it, though they are expected to be. For concreteness, we give the defining relations in the generators $(ij) \in \mathfrak{B}_d^1 = k\tau_d$: \begin{align*}
(ij)^2 &= 0,\quad \forall i,j \\
(ij)(kl) + (kl)(ij) &= 0,\quad \forall \{i,j\} \cap \{k,l\} = \emptyset \\
(ij)(jk) + (jk)(ik) + (ik)(ij) &= 0,\quad \forall i < j < k \\
(jk)(ij) + (ik)(jk) + (ij)(ik) &= 0,\quad \forall i < j < k.
\end{align*}

We end by noting that there are two different $\mathbf{S}_d$-actions on $\mathfrak{B}_d$, one coming from $k\tau_d \otimes \sgn \in \mathcal{YD}_{\mathbf{S}_d}^{\mathbf{S}_d}$ and one coming from $k\tau_d \in \mathcal{YD}_{\mathbf{S}_d,\epsilon}^{\mathbf{S}_d}$. We call the former the \emph{standard $\mathbf{S}_d$-action} and the latter the \emph{geometric $\mathbf{S}_d$-action}.

\section{Topology of Hurwitz spaces}\label{section:top_hurwitz}
In this section, we will review the topological models for Hurwitz spaces studied in \cite{etw_fnf_cells,evw_homological_stability} and state the Kapranov-Schechtman theorem comparing the cohomology of certain Nichols algebra and the cohomology of certain intersection cohomology complexes on $\Sym^n(\mathbb{A}^1)$ (\Cref{theorem:ks_nichols_ic}). For this section, $G$ will denote a finite group, and $c$ will denote a union of conjugacy classes of $G$. Again, $k$ will be a field of characteristic 0 for this section. Just for this section, all schemes will be over $\mathbb{C}$ and will be identified with their spaces of $\mathbb{C}$-points.

\subsection{Hurwitz spaces of the disc}\label{section:hurwitz_disc}
Following \cite[\S2]{evw_homological_stability}, we will define Hurwitz spaces $\Hur_{G,n}^c$ parametrizing marked $G$-covers of the disc with monodromy in $c$ around $n$ branch points.

Let $\overline{D} \subset \mathbb{C}$ be the closed unit disc, and let $D \subset \overline{D}$ be the open unit disc. The configuration space $\Conf^n(D)$ is the space of $n$ unordered points in $D$. Let $c_n$ be the set of $n$ points $\{p_1,\ldots,p_n\}$ where $p_j = \frac{j - 1}{n}i \in D$. There is a well-known description of the braid group $\mathbf{B}_n = \pi_1(\Conf^n(D),c_n)$, where $\sigma_j$ corresponds to the loop in $\Conf^n(D)$ where all the points remain the same except $p_j$ and $p_{j + 1}$, which move according to \begin{align*}
\widetilde{p}_j(t) &= \frac{2j - 1}{n}i - \frac{1}{2n}ie^{\pi it} \\
\widetilde{p}_{j + 1}(t) &= \frac{2j - 1}{n}i + \frac{1}{2n}ie^{\pi it}
\end{align*}
for $0 \le t \le 1$. Note that \begin{align*}
\widetilde{p}_j(0) &= p_j \\
\widetilde{p}_j(1) &= p_{j + 1} \\
\widetilde{p}_{j + 1}(0) &= p_{j + 1} \\
\widetilde{p}_{j + 1}(1) &= p_j,
\end{align*}
so $\sigma_j$ corresponds to the loop moving $p_j$ and $p_{j + 1}$ counterclockwise around each other until they swap positions.

Because $\pi_1(\Conf^n(D),c_n) = \mathbf{B}_n$, any $\mathbf{B}_n$-set $T$ gives rise to a covering space of $\Conf^n(D)$ by taking $\widetilde{\Conf}^n(D) \times_{\mathbf{B}_n} T \to \Conf^n(D)$, where $\widetilde{\Conf}^n(D) \to \Conf^n(D)$ is the universal cover. In particular, we will consider $S = c^{\times n}$, where $\mathbf{B}_n$ acts on $\underline{g} = (g_1,\ldots,g_n) \in c^{\times n}$ by $$\sigma_i(\underline{g})_j = \begin{cases}g_{j + 1} &\text{if }j = i \\ g_{j + 1}^{-1}g_jg_{j + 1} &\text{if }j = i + 1 \\ g_j &\text{otherwise}.\end{cases}$$

\begin{definition}\label{definition:hurwitz_disc}
The \emph{$n$th Hurwitz space of the disc associated to $(G,c)$} is $$\Hur_{G,n}^c \coloneqq \widetilde{\Conf}^n(D) \times_{\mathbf{B}_n} c^{\times n}.$$ If $c = G$, we will simply write $\Hur_{G,n}$.
\end{definition}

The Hurwitz space $\Hur_{G,n}^c$ has a natural moduli description in terms of marked covers of the disc. Note that in the following definition, we do not restrict to connected covers.

\begin{definition}\label{definition:marked_branched_cover}
A \emph{marked $n$-branched $G$-cover of the disc} is a quintuple $(Y,p,\bullet,S,\alpha)$, where \begin{enumerate}[(1)]
\item $S \subset D$ is a set of $n$ distinct points.

\item $p: Y \to \overline{D} - S$ is a covering map.

\item $\alpha: G \to \Aut(p)$ is a homomorphism inducing a simply transitive action of $G$ on each fiber of $p$.

\item $\bullet$ is a point in the fiber $p^{-1}(1)$.
\end{enumerate}
\end{definition}

\begin{proposition}[{\cite[\S2.3]{evw_homological_stability}}]\label{proposition:tophur_moduli}
There are bijections between the following sets: \begin{enumerate}[(a)]
\item Points of $\Hur_{G,n}$.

\item Pairs $(S,f)$, where $S \in \Conf^n(D)$ and $f: \pi_1(\overline{D} - S,1) \to G$ is a homomorphism.

\item Isomorphism classes of marked $n$-branched $G$-covers of the disc.
\end{enumerate}
\end{proposition}

Given a pair $(S,f)$ as in (b) and a point $s \in S$, we can get a well-defined conjugacy class in $G$, which we call the \emph{monodromy around $s$}. This is defined by taking a small counterclockwise loop around $s$, identifying it with a conjugacy class in $\pi_1(\overline{D} - S,1)$ by picking a path from 1 to a point on the loop, and taking the image of this conjugacy class in $G$ under $f$. Under the bijections above, $\Hur_{G,n}^c \subset \Hur_{G,n}$ corresponds to pairs $(S,f)$ with monodromy in $c$ around each $s \in S$.

Note that there is a $G$-action on $\Hur_{G,n}^c$ because the $G$-action and $\mathbf{B}_n$-action on $c^{\times n}$ commute. This $G$-action corresponds to the action of $G$ on the marking $\bullet \in p^{-1}(1)$. Thus, the points of the quotient $\Hur_{G,n}^c/G = \widetilde{\Conf}^n(D) \times_{\mathbf{B}_n} (c^{\times n}/G)$ correspond to unmarked $n$-branched $G$-covers of the disc.

\subsubsection{The case $G = \mathbf{S}_d$}\label{section:hurwitz_sd}
Recall the equivalence between $\mathbf{S}_d$-covering spaces and degree $d$ covering spaces. Starting with an $\mathbf{S}_d$-covering space $Y \to S$, we get a degree $d$ covering space $Y/\mathbf{S}_{d - 1} \to S$. Starting with a degree $d$ covering space $X \to S$, we get a $\mathbf{S}_d$-covering space by taking the $d$th fiber product $X^d \to S$ and deleting the big diagonal.

With this equivalence in mind, we make the following definition:

\begin{definition}\label{definition:marked_degree_d}
A \emph{marked $n$-branched degree $d$ cover of the disc} is a quintuple $(X,p,\beta,S)$, where \begin{enumerate}[(1)]
\item $S \subset D$ is a set of $n$ distinct points.

\item $p: X \to \overline{D} - S$ is a degree $d$ covering map.

\item $\beta: p^{-1}(1) \xrightarrow[]{\sim} \{1,\ldots,d\}$ is an isomorphism of sets.
\end{enumerate}
\end{definition}

Note that $\beta$ plays the role of the marking $\bullet$. From the equivalence of $\mathbf{S}_d$-covering spaces and degree $d$ covering spaces, we get the following corollary of \Cref{proposition:tophur_moduli}:

\begin{corollary}\label{corollary:tophur_degree_d}
There is a bijection between points of $\Hur_{G,n}$ and isomorphism classes of marked $n$-branched degree $d$ covers of the disc.
\end{corollary}

For a marked $n$-branched degree $d$-cover of the disc $(X,p,\beta,S)$, each point $s \in S$ gives us a well-defined conjugacy class of $\mathbf{S}_d$: take a small counterclockwise loop around $s$, and see how it permutes a nearby fiber. Hence, the points $\Hur_{\mathbf{S}_d,n}^c$ parametrize marked $n$-branched degree $d$ covers of the disc with monodromy in $c$ around each point in $S$. We will be particularly interested in the case $c = \tau_d$, the conjugacy class of transpositions, as $\Hur_{\mathbf{S}_d,n}^{\tau_d}$ parametrizes simply branched covers.

\subsubsection{Hurwitz spaces of the affine line}\label{section:hurwitz_a1}
For the rest of this paper, fix a radial homeomorphism $D \cong \mathbb{A}^1$ induced by a homeomorphism $[0,1) \xrightarrow[]{\sim} [0,\infty)$. This induces homeomorphisms $\Conf^n(D) \cong \Conf^n(\mathbb{A}^1)$, so that $\Hur_{G,n}^c \to \Conf^n(\mathbb{A}^1)$ and $\Hur_{G,n}^c/G \to \Conf^n(\mathbb{A}^1)$ are covering spaces. In fact, by Riemann's existence theorem, $\Hur_{G,n}^c$ and $\Hur_{G,n}^c/G$ are the spaces of $\mathbb{C}$-points of finite \'{e}tale covers of $\Conf^n(\mathbb{A}^1)$. For the rest of \Cref{section:top_hurwitz}, we will not distinguish between the spaces $\Hur_{G,n}^c$ and the corresponding varieties over $\mathbb{C}$.

\subsection{Nichols algebra cohomology and intersection cohomology}\label{section:nichols_ic}
Because $$\pi: \Hur_{G,n}^c \to \Conf^n(\mathbb{A}^1)$$ is the finite \'{e}tale cover (or equivalently, finite covering space) corresponding to the $\mathbf{B}_n$-set $c^{\times n}$, the pushforward $\pi_*\underline{k}$ is the local system on $\Conf^n(\mathbb{A}^1)$ corresponding to the $\mathbf{B}_n$-representation $kc^{\times n}$.

Note that there in an open immersion $j: \Conf^n(\mathbb{A}^1) \xhookrightarrow{} \Sym^n(\mathbb{A}^1)$. The following theorem of Kapranov-Schechtman is a major motivation for this paper. We present it in its dual formulation, as the $\Ext$ algebra seems more natural to us than the $\Tor$ coalgebra \cite[Remark 3.3.6]{ks_shuffle_algs}.

\begin{theorem}[{\cite[Corollary 3.3.5]{ks_shuffle_algs}}]\label{theorem:ks_nichols_ic}
There is an isomorphism $$\Ext_{T_{!*}(V(c,-1))}^{n - j,n}(k,k) \cong H^j(\Sym^n(\mathbb{A}^1),j_{!*}\pi_*\underline{k}_{\Hur_{G,n}^c})^{\vee}.$$
\end{theorem}

Note that $\pi$ factors as $\Hur_{G,n}^c \to \Hur_{G,n}^c/G \xrightarrow[]{\pi^{\prime}} \Conf^n(\mathbb{A}^1)$. Then $$\pi_*^{\prime}\underline{k}_{\Hur_{G,n}^c/G} \cong (\pi_*\underline{k}_{\Hur_{G,n}^c})^G.$$ Because $k$ has characteristic 0, $(\pi_*\underline{k}_{\Hur_{G,n}^c})^G$ is naturally a direct summand of $\pi_*\underline{k}_{\Hur_{G,n}^c}$, so $$j_{!*}\pi_*^{\prime}\underline{k}_{\Hur_{G,n}^c/G} \cong (j_{!*}\pi_*\underline{k}_{\Hur_{G,n}^c})^G.$$ From the proof of \cite[Theorem 3.3.1]{ks_shuffle_algs}, we see that this action corresponds to the $G$-action on $T_{!*}(V(c,-1))$, so we get the following corollary.

\begin{corollary}\label{corollary:nichols_quotient_ic}
There is an isomorphism $$\Ext_{T_{!*}(V(c,-1))}^{n - j,n}(k,k)^G \cong H^j(\Sym^n(\mathbb{A}^1),j_{!*}\pi_*^{\prime}\underline{k}_{\underline{k}_{\Hur_{G,n}^c/G}})^{\vee}.$$
\end{corollary}

\begin{remark}\label{remark:std_vs_geo}
Note that when $G = \mathbf{S}_d$ and $c = \tau_d$ for $d \in \{3,4,5\}$, the $\mathbf{S}_d$-action in \Cref{corollary:nichols_quotient_ic} comes from the geometric $\mathbf{S}_d$-action on $\mathfrak{B}_d$, as the name suggests. On the other hand, by \cite{sv_cohomology_fk_alg}, the invariants of the cohomology of $\mathfrak{B}_d$ under the standard $\mathbf{S}_d$-action correspond to the cohomology of a Hopf algebra, the bosonization $\mathfrak{B}_d \# k[\mathbf{S}_d]$.

We find it a bit unsatisfying that these two actions for which we take invariants are not the same, and we wonder if there is any quantum algebraic interpretation of the geometric $\mathbf{S}_d$-invariants or geometric interpretation of the standard $\mathbf{S}_d$-invariants. On the bright side, the actions agree on $\Ext_{\mathfrak{B}_d}^{n - j,n}(k,k)$ with $n$ even. In \Cref{theorem:inv_coh_bd_34}, we will compute the geometric $\mathbf{S}_4$-invariant cohomology of $\mathfrak{B}_4$, which agrees with the cohomology of $\mathfrak{B}_4 \# k[\mathbf{S}_4]$ in even internal degrees by \Cref{corollary:bosonization_even_b}.
\end{remark}

\section{Prehomogeneous vector spaces and low degree covers}\label{section:prehomog_covs}
In this section, we will review some of the connections between prehomogeneous vector spaces, parametrizations of low degree covers, and arithmetic statistics that we will need to construct our stacks of covers with resolvents.

Recall that a \emph{prehomogeneous vector space} (over some base field $k$) is a triple $(G,\rho,V)$ such that $G$ is a connected linear algebraic group, $V$ is a vector space, and $\rho: G \to \GL(V)$ is a representation such that there is a Zariski dense orbit $V^{ss} \subset V$.

Throughout this paper, we will only ever consider three prehomogeneous vector spaces, which we will denote $(G_d,V_d)$ for $d \in \{3,4,5\}$. We use $\std_n$ to denote the standard $n$-dimensional representation of $\GL_n$: \begin{itemize}
\item $d = 3$: $G_3 = \GL_2$, $V_3 = \Sym^3(\std_2) \otimes \det\std_2^{\vee}$.

\item $d = 4$: $G_4 = \{(g_3,g_2) \in \GL_3 \times \GL_2|\det g_3 = \det g_2\}$, $V_4 = \Sym^2(\std_3) \otimes \std_2^{\vee}$.

\item $d = 5$: $G_5 = \{(g_4,g_5) \in \GL_4 \times \GL_5|(\det g_4)^2 = \det g_5\}$, $V_5 = \std_4 \otimes \det\std_4^{\vee} \otimes \Lambda^2\std_5$.
\end{itemize}
If $d \in \{4,5\}$, we will write a $G_d$-bundle as a tuple $(\mathcal{F},\mathcal{F}^{\prime},\alpha)$, where \begin{itemize}
\item $d = 4$: $\mathcal{F}$ is a rank 3 vector bundle, $\mathcal{F}^{\prime}$ is a rank 2 vector bundle, and $\alpha: \det(\mathcal{F}) \xrightarrow[]{\sim} \det(\mathcal{F}^{\prime})$ is an isomorphism.

\item $d = 5$: $\mathcal{F}$ is a rank 4 vector bundle, $\mathcal{F}^{\prime}$ is a rank 5 vector bundle, and $\alpha: \det(\mathcal{F})^2 \xrightarrow[]{\sim} \det(\mathcal{F}^{\prime})$ is an isomorphism.
\end{itemize}

For the rest of this paper, $d$ will denote either 3, 4, or 5. We will use $G_d$ and $V_d$ to denote the corresponding group schemes and affine spaces over $\mathbb{Z}$. We will use $\mathfrak{g}_d$ to denote the Lie algebra of $G_d$ as an affine space over $\mathbb{Z}$. Define the quotient stack $\mathcal{X}_d = [V_d/G_d]$.

Each $V_d$ comes with a (unique up to a scalar) homogeneous polynomial $\Delta_d: V_d \to S$ called the \emph{discriminant} that scales with respect to the $G_d$-action according to the square of a generator $\chi_d: G_d \to \mathbb{G}_m$ of the character group of $G_d$: \begin{itemize}
\item $d = 3$: For a binary cubic form $f(x,y) = ax^3 + bx^2y + cxy^2 + dy^3 \otimes (x \wedge y)^{\vee} \in V_3$, we have $$\Delta_3(f) = b^2c^2 - 4ac^3 - 4b^3d - 27a^2d^2 + 18abcd.$$ For $g \in G_3$, we have $\chi_3(g) = \det(g)$.

\item $d = 4$: For a pair of ternary quadratic forms $(\sum_{0 \le i \le j \le 2}a_{ij}x_ix_j,\sum_{0 \le i \le j \le 2}b_{ij}x_ix_j)$ represented as a pair of $3 \times 3$ symmetric matrices $$(A,B) = \left(\begin{pmatrix}a_{00} & \frac{1}{2}a_{01} & \frac{1}{2}a_{02} \\ \frac{1}{2}a_{01} & a_{11} & \frac{1}{2}a_{12} \\ \frac{1}{2}a_{02} & \frac{1}{2}a_{12} & a_{22}\end{pmatrix},\begin{pmatrix}b_{00} & \frac{1}{2}b_{01} & \frac{1}{2}b_{02} \\ \frac{1}{2}b_{01} & b_{11} & \frac{1}{2}b_{12} \\ \frac{1}{2}b_{02} & \frac{1}{2}b_{12} & b_{22}\end{pmatrix}\right),$$ we have\footnote{Even though there are $\frac{1}{2}$'s in the expressions for $A,B$, the expression $4\det(Ax - By)$ has only integer coefficients. Thus, it makes sense even when 2 is not invertible.} $$\Delta_4(A,B) = \Delta_3(4\det(Ax - By) \otimes (x \wedge y)^{\vee}).$$ For $(g_3,g_2) \in G_4$, we have $\chi_4(g_3,g_2) = \det(g_3)$.

\item $d = 5$: The function $\Delta_5$ is a complicated degree 40 polynomial, so we refer the reader to \cite[Example 2.11]{k_prehomog} and \cite[\S4]{b_hcl_iv}. For $(g_4,g_5) \in G_5$, we have $\chi_5(g_4,g_5) = \det(g_4)$.
\end{itemize}
We will call the open subset of $V_d$ where $\Delta_d$ does not vanish the \emph{semistable locus} $V_d^{ss}$. It is well-known (e.g. \cite[Example 2.4, Example 2.8, Example 2.11]{k_prehomog}) that in each case, $\Delta_d(gv) = \chi_d(g)^2v$. Thus, $\Delta_d$ descends to a section of a line bundle $\mathcal{L}_d$ on $\mathcal{X}_d$, which we call the \emph{discriminant line bundle}. Imitating \cite[\S 2.1]{d_compact}, we will call the zero locus in $\mathcal{X}_d$ of $\Delta_d$ the \emph{branch locus} $\Sigma_d$ and its complement the \emph{\'{e}tale locus} $\mathcal{E}_d$. We will justify this terminology later.

For a field $k$ of characteristic not dividing $d!$, the ring of relative invariant functions on $(V_d)_k$ is a polynomial ring generated by $(\Delta_d)_k$. This means that $(V_d^{ss})_k$ really is the $\chi_d$-semistable locus in the sense of GIT (see \cite[\S 2]{k_moduli_of_reps} for definitions). In fact, $(V_d^{ss})_k$ is also the $\chi_d$-stable locus because $(G_d)_k$ acts transitively on $(V_d^{ss})_k$ with finite stabilizers. Thus, we will sometimes write $V_d^s$ for $V_d^{ss}$.

\subsection{Covers and resolvents}\label{section:prehomog_cov_res}
Our motivation for studying prehomogeneous vector spaces comes from their relationship with parametrizations of low degree covers and arithmetic statistics. In the work of Davenport-Heilbronn \cite{dh_density_discriminants_cubic} and Bhargava \cite{b_hcl_iii,b_hcl_iv}, the authors count degree $d$ number fields (for $3 \le d \le 5$) by counting certain unrefined objects and sieving out the ones that do not contribute to the number field count. These unrefined objects are parametrized by orbits in prehomogeneous vector spaces: \begin{itemize}
\item $V_3(\mathbb{Z})/G_3(\mathbb{Z})$ parametrizes cubic $\mathbb{Z}$-algebras.

\item $V_4(\mathbb{Z})/G_4(\mathbb{Z})$ parametrizes quartic $\mathbb{Z}$-algebras with cubic resolvents.

\item $V_5(\mathbb{Z})/G_5(\mathbb{Z})$ parametrizes quintic $\mathbb{Z}$-algebras with sextic resolvents.
\end{itemize}
The idea is that the quotient stacks $\mathcal{X}_d$ should parametrize versions of these objects for general base. For a fixed base scheme $S$, we make the following definitions: \begin{itemize}
\item A \emph{cubic cover of $S$} is a degree 3 finite locally free morphism $C \to S$ (to define an \emph{$n$-ic cover} or \emph{degree $n$ cover}, replace 3 with $n$). It is shown in \cite[\S2]{w_quartic_arbitrary} that the stack of cubic covers is isomorphic to the stack $\mathcal{X}_3$.

\item A \emph{quartic cover of $S$ with cubic resolvent}, as defined by Wood in \cite[\S3]{w_quartic_arbitrary}, is a tuple $(Q,C,\phi,\delta)$, where \begin{itemize}
\item $Q$ is a quartic cover of $S$.

\item $C$ is a cubic cover of $S$.

\item The \emph{resolvent map} $\phi$ is a map $\Sym_2(\mathcal{O}_Q/\mathcal{O}_S) \to \mathcal{O}_C/\mathcal{O}_S$ of vector bundles on $S$, which we can view as a quadratic map $\mathcal{O}_Q/\mathcal{O}_S \to \mathcal{O}_C/\mathcal{O}_S$. Here, as in \cite{w_quartic_arbitrary}, for a locally free sheaf $\mathcal{F}$, we use $\Sym_n\mathcal{F}$ to denote the subsheaf of symmetric elements of $\mathcal{F}^{\otimes n}$, as opposed to the usual quotient $\Sym^n\mathcal{F}$ of $\mathcal{F}^{\otimes n}$. By \cite[Lemma A.4]{w_quartic_arbitrary}, $(\Sym_n\mathcal{F})^{\vee} \cong \Sym^n\mathcal{F}^{\vee}$. If 2 is invertible (as it will be for most of this paper), the map $\Sym_2\mathcal{F} \to \Sym^2\mathcal{F}$ is an isomorphism, so there is no real distinction.

\item The \emph{orientation} $\delta$ is an isomorphism $\det(\mathcal{O}_Q/\mathcal{O}_S) \xrightarrow[]{\sim} \det(\mathcal{O}_C/\mathcal{O}_S)$.
\end{itemize}
such that \begin{enumerate}[(1)]
\item For local sections $x,y$ of $\mathcal{O}_Q/\mathcal{O}_S$, we have\footnote{Here, $xy$ does not make sense as a section of $\mathcal{O}_Q/\mathcal{O}_S$, but it does make sense when wedged with $x \wedge y$.} $\delta(x \wedge y \wedge xy) = \phi(x) \wedge \phi(y)$.

\item If we pick local frames for our vector bundles, we can write $\phi$ as a pair of $3 \times 3$ symmetric matrices $(A,B)$ and define the binary cubic form $(x,y) \mapsto 4\det(Ax - By) \otimes (x \wedge y)^{\vee}$. These local expressions glue to a global section $\det\phi \in H^0(S,\Sym^3(\mathcal{O}_C/\mathcal{O}_S)^{\vee} \otimes \det(\mathcal{O}_C/\mathcal{O}_S))$. We require that $C$ is the cubic cover of $S$ corresponding to the binary cubic form $\det\phi$ with underlying bundle $(\mathcal{O}_C/\mathcal{O}_S)^{\vee}$.
\end{enumerate}
The main result of \cite{w_quartic_arbitrary} is that there is an isomorphism from the stack of quartic covers with cubic resolvents to the stack $\mathcal{X}_4$ that sends $(Q,C,\phi,\delta)$ to the section $\phi$ of the $V_4$-bundle associated to the $G_4$-bundle $((\mathcal{O}_Q/\mathcal{O}_S)^{\vee},(\mathcal{O}_C/\mathcal{O}_S)^{\vee},(\delta^{\vee})^{-1})$.

\item \emph{Quintic covers with sextic resolvent} have only been studied for Dedekind domains. In \cite{o_quintic_dedekind}, it is shown that for a Dedekind domain $R$, there is a bijection between $\mathcal{X}_5(R)/\sim$ and isomorphism classes of quintic $R$-algebras with sextic resolvents, and it is also shown that each quintic $R$-algebra admits a sextic resolvent. However, unlike the cases $d \in \{3,4\}$, an isomorphism between $\mathcal{X}_5$ and a suitable stack of quintic covers with sextic resolvents has not been found. The closest result is the Casnati-Ekedahl/Landesman-Vakil-Wood parametrization of quintic Gorenstein covers, which we will discuss in \Cref{section:gorenstein_covers}.
\end{itemize}

\subsection{The stack of covers}\label{section:stack_covers}
\begin{definition}\label{definition:stack_of_covers}
The \emph{stack of degree $d$ covers} $\Covers_d$ is the stack over $\mathbb{Z}$ whose objects over $S$ are degree $d$ finite locally free maps $C \to S$.
\end{definition}

The following facts about $\Covers_d$ are well-known and can be found in \cite{p_commalg}: \begin{itemize}
\item $\Covers_d$ has a presentation as a quotient stack $[B_d/\GL_d]$, where $B_d$ is the finite type affine scheme over $\mathbb{Z}$ parametrizing \emph{based degree $d$ covers}: the points $B_d(S)$ correspond to $\mathcal{O}_S$-algebra structures on $\mathcal{O}_S^{\oplus d}$, and $\GL_d(S)$ acts on $B_d(S)$ by acting on $\mathcal{O}_S^{\oplus d}$. The coordinates of $B_d$ can be thought of as the coordinates of the identity as a linear combination of the standard basis of $\mathcal{O}_S^{\oplus d}$ and the elements in the multiplication table for the algebra structure in terms of the standard basis. $B_d$ is cut out of the affine space with these coordinates by equations corresponding to the identity, associativity, and commutativity axioms.

\item There is a function $\Delta_d^{\Cov}$ on $B_d$ called the \emph{discriminant} which computes the usual discriminant of a based degree $d$ algebra. This function has the property that for $g \in \GL_d(S)$ and $b \in B_d(S)$, $\Delta_d^{\Cov}(gb) = \det(g)^{-2}\Delta_d^{\Cov}(b)$. Thus, $\Delta_d^{\Cov}$ descends to a section of a line bundle $\mathcal{L}_d^{\Cov}$ on $\Covers_d$. We will use $\Delta_d^{\Cov}$ to denote this section (in addition to the function above). We call $\mathcal{L}_d^{\Cov}$ the \emph{discriminant line bundle}.

\item The complement of the zero locus of $\Delta_d^{\Cov}$ is called the \emph{\'{e}tale locus}, and it is naturally isomorphic to $B\mathbf{S}_d$ because it parametrizes degree $d$ finite \'{e}tale covers.
\end{itemize}
Note that we have used $\Delta_d$ for the discriminant on both $\mathcal{X}_d$ and $\Covers_d$. We will justify this by exhibiting morphisms $\mathcal{X}_d \to \Covers_d$ and then showing that the pullback of the discriminant on $\Covers_d$ agrees with the discriminant on $\mathcal{X}_d$.

\subsection{The morphisms $\phi_d: \mathcal{X}_d \to \Covers_d$}\label{section:phi_d}
We define natural morphisms $\phi_d: \mathcal{X}_d \to \Covers_d$. Fix a base scheme $S$ throughout. Whenever we have a projective bundle $\mathbb{P}\mathcal{F} \to S$, let $\pi$ denote the projection. The cases $d \in \{3,4\}$ are done in \cite{w_quartic_arbitrary}, so we will recall the construction in those cases without providing too many details. This discussion is similar to \cite[\S3]{lvw_low_degree_hurwitz}, except that they consider only the Gorenstein case: \begin{itemize}
\item $\phi_3: \mathcal{X}_3 \xrightarrow[]{\sim} \Covers_3$: Let $\mathcal{F}_2$ be a rank 2 bundle on $S$, along with a section $\eta \in H^0(S,\Sym^3\mathcal{F}_2 \otimes \det\mathcal{F}_2^{\vee})$. There is a natural isomorphism $\Phi_3: H^0(S,\Sym^3\mathcal{F}_2 \otimes \det\mathcal{F}_2^{\vee}) \xrightarrow[]{\sim} H^0(\mathbb{P}\mathcal{F}_2,\pi^*\det\mathcal{F}_2^{\vee}(3))$. Using $\Phi_3(\eta)$, we can form a cdga $K$ of bundles on $\mathbb{P}\mathcal{F}_2$: \begin{align}
\pi^*\det\mathcal{F}_2(-3) &\to \mathcal{O}_{\mathbb{P}\mathcal{F}_2} \label{equation:complex_d=3}
\end{align}
Here $\mathcal{O}_{\mathbb{P}\mathcal{F}_2}$ is in degree 0. Then $\mathcal{O}_C \coloneqq R^0\pi_*K$ is the structure sheaf of a cubic cover $C \to S$. This construction commutes with base change on $S$, and we can recover $\mathcal{F}_2$ from $C$ as $\mathcal{F}_2 = (\mathcal{O}_C/\mathcal{O}_S)^{\vee}$. In this case (and only this case), $\phi_3$ is an isomorphism.

\item $\phi_4: \mathcal{X}_4 \to \Covers_4$: Let $\mathcal{F}_3$ and $\mathcal{F}_2$ be bundles on $S$ of respective ranks 3 and 2, along with an isomorphism $\delta: \det\mathcal{F}_3 \xrightarrow[]{\sim} \det\mathcal{F}_2$ and a section $\eta \in H^0(S,\Sym^2\mathcal{F}_3 \otimes \mathcal{F}_2^{\vee})$. There is a natural isomorphism $\Phi_4: H^0(S,\Sym^2\mathcal{F}_3 \otimes \mathcal{F}_2^{\vee}) \xrightarrow[]{\sim} H^0(\mathbb{P}\mathcal{F}_3,\pi^*\mathcal{F}_2^{\vee}(2))$. Taking the Koszul complex of $\Phi_4(\eta)$, we can form a cdga $K$ of bundles on $\mathbb{P}\mathcal{F}_3$: \begin{align}
\pi^*\det\mathcal{F}_2(-4) \to \pi^*\mathcal{F}_2(-2) &\to \mathcal{O}_{\mathbb{P}\mathcal{F}_3} \label{equation:complex_d=4}
\end{align}
Here $\mathcal{O}_{\mathbb{P}\mathcal{F}_3}$ is in degree 0. Then $\mathcal{O}_Q \coloneqq R^0\pi_*K$ is the structure sheaf of a quartic cover $Q \to S$. This construction commutes with base change on $S$, and we can recover $\mathcal{F}_3$ from $Q$ as $\mathcal{F}_3 = (\mathcal{O}_Q/\mathcal{O}_S)^{\vee}$.

\item $\phi_5: \mathcal{X}_5 \to \Covers_5$: Let $\mathcal{F}_4$ and $\mathcal{F}_5$ be bundles on $S$ of respective ranks 4 and 5, along with an isomorphism $\delta: (\det\mathcal{F}_4)^2 \xrightarrow[]{\sim} \det\mathcal{F}_5$ and a section $\eta \in H^0(S,\mathcal{F}_4 \otimes \det\mathcal{F}_4^{\vee} \otimes \Lambda^2\mathcal{F}_5)$. There is a natural isomorphism $\Phi_5: H^0(S,\mathcal{F}_4 \otimes \det\mathcal{F}_4^{\vee} \otimes \Lambda^2\mathcal{F}_5) \xrightarrow[]{\sim} H^0(\mathbb{P}\mathcal{F}_3,\pi^*\det\mathcal{F}_4^{\vee} \otimes \Lambda^2\pi^*\mathcal{F}_5(1))$. We can form a complex of bundles $K$ on $\mathbb{P}\mathcal{F}_4$ as in \cite[Lemma 3.15]{lvw_low_degree_hurwitz}: \begin{align}
\pi^*\det\mathcal{F}_4(-5) \xrightarrow[]{\beta_3} \pi^*\det\mathcal{F}_4 \otimes \pi^*\mathcal{F}_5^{\vee}(-3) \xrightarrow[]{\beta_2} \pi^*\mathcal{F}_5(-2) &\xrightarrow[]{\beta_1} \mathcal{O}_{\mathbb{P}\mathcal{F}_4} \label{equation:complex_d=5}
\end{align}
Here, $\mathcal{O}_{\mathbb{P}\mathcal{F}_4}$ is in degree 0. The map $\beta_2$ comes from the section $\Phi_5(\eta)$. If we use $\delta$ to write $\pi^*\det\mathcal{F}_4(-5) \cong \pi^*\det\mathcal{F}_4^{\vee} \otimes \pi^*\det\mathcal{F}_5(-5)$, then $\beta_3$ is obtained by taking the five $4 \times 4$ Pfaffians of $\beta_2$. We obtain $\beta_1$ by taking the dual of $\beta_3$ and tensoring with $\pi^*\det\mathcal{F}_4(-5)$. We now describe the $\mathcal{O}_{\mathbb{P}\mathcal{F}_4}$-algebra structure on $K$ term by term: \begin{itemize}
\item $\mathcal{O}_{\mathbb{P}\mathcal{F}_4} \otimes K \to K$: This is just the $\mathcal{O}_{\mathbb{P}\mathcal{F}_4}$-module structure on $K$.

\item $\pi^*\mathcal{F}_5(-2) \otimes \pi^*\mathcal{F}_5(-2) \to \pi^*\det\mathcal{F}_4 \otimes \pi^*\mathcal{F}_5^{\vee}(-3)$: Because of our orientation $\delta$, there is a natural isomorphism $$\pi^*\det\mathcal{F}_4 \otimes \pi^*\mathcal{F}_5^{\vee} \cong \pi^*\det\mathcal{F}_4^{\vee} \otimes \pi^*\det\mathcal{F}_5 \otimes \pi^*\mathcal{F}_5^{\vee}$$ Moving the $\pi^*\mathcal{F}_5^{\vee}$ to the other side, we see that we want to provide a map $$\pi^*\mathcal{F}_5^{\otimes 3}(-4) \to \pi^*\det\mathcal{F}_4^{\vee} \otimes \pi^*\det\mathcal{F}_5$$ This map is given by $(x,y,z) \mapsto \Phi_5(\eta) \wedge (x \wedge y \wedge z)$.

\item $\pi^*\mathcal{F}_5(-2) \otimes (\pi^*\det\mathcal{F}_4 \otimes \pi^*\mathcal{F}_5^{\vee}(-3)) \to \pi^*\det\mathcal{F}_4(-5)$: This comes from the pairing $\pi^*\mathcal{F}_5 \otimes \pi^*\mathcal{F}_5^{\vee}(-3) \to \mathcal{O}_{\mathbb{P}\mathcal{F}_4}$.

\item All the other pairings are 0 for degree reasons.
\end{itemize}
By construction, our multiplication on $K$ is graded commutative. The proof of \cite[Theorem 4.1]{be_algebra_structures} shows that the multiplication we have defined is associative and compatible with the differential. Thus, we have given $K$ a cdga structure, and its derived pushforward $\mathcal{O}_Q \coloneqq R^0\pi_*K$ is an $\mathcal{O}_S$-algebra. By cohomology and base change, we have natural quasi-isomorphisms \begin{align*}
R\pi_*\mathcal{O}_{\mathbb{P}\mathcal{F}_4} &\simeq \mathcal{O}_S[0] \\
R\pi_*\left(\pi^*\mathcal{F}_5(-2)\right) &\simeq 0 \\
R\pi_*\left(\pi^*\det\mathcal{F}_4 \otimes \pi^*\mathcal{F}_5^{\vee}(-3)\right) &\simeq 0 \\
R\pi_*\left(\pi^*\det\mathcal{F}_4(-5)\right) &\simeq \mathcal{F}_4^{\vee}[-3]
\end{align*}
The hypercohomology spectral sequence then shows that $R\pi_*K \simeq \mathcal{O}_Q[0]$ and that there is a natural short exact sequence $$0 \to \mathcal{O}_S \to \mathcal{O}_Q \to \mathcal{F}_4^{\vee} \to 0$$ Thus, $\mathcal{O}_Q$ is the structure sheaf of a quintic cover $Q \to S$. Because all the $R^i\pi_*K$ are flat, the cohomology and base change theorem implies that the formation of $Q$ is compatible with base change. We also have a natural isomorphism $\mathcal{F}_4 \cong (\mathcal{O}_Q/\mathcal{O}_S)^{\vee}$.
\end{itemize}

We will need the following definition when we discuss Gorenstein covers in \Cref{section:gorenstein_covers}.

\begin{definition}[{\cite[Definition 3.12]{lvw_low_degree_hurwitz}}]\label{definition:right_codim}
Fix a base scheme $S$ and a map $S \to \mathcal{X}_d$. Let $\eta$ be the corresponding section of \begin{itemize}
\item $H^0(S,\Sym^3\mathcal{F}_2 \otimes \det\mathcal{F}_2^{\vee})$ if $d = 3$.

\item $H^0(S,\Sym^2\mathcal{F}_3 \otimes \mathcal{F}_2^{\vee})$ if $d = 4$.

\item $H^0(S,\mathcal{F}_4 \otimes \det\mathcal{F}_4^{\vee} \otimes \Lambda^2\mathcal{F}_5)$ if $d = 5$.
\end{itemize}
Define a closed subscheme $\Psi_d(\eta) \subset \mathbb{P}\mathcal{F}_{d - 1}$ as follows. Here, the maps to $\mathcal{O}_{\mathbb{P}\mathcal{F}_{d - 1}}$ are the ones obtained from $\eta$ in the construction of $\phi_d$ above: \begin{itemize}
\item $\mathcal{O}_{\Psi_3(\eta)} \coloneqq \coker(\pi^*\det\mathcal{F}_2(-3) \to \mathcal{O}_{\mathbb{P}\mathcal{F}_2})$ if $d = 3$.

\item $\mathcal{O}_{\Psi_4(\eta)} \coloneqq \coker(\pi^*\mathcal{F}_2(-2) \to \mathcal{O}_{\mathbb{P}\mathcal{F}_3})$ if $d = 4$.

\item $\mathcal{O}_{\Psi_5(\eta)} \coloneqq \coker(\pi^*\mathcal{F}_5(-2) \to \mathcal{O}_{\mathbb{P}\mathcal{F}_4})$ if $d = 5$.
\end{itemize}
We say that $\eta$ has the \emph{right codimension} at a point $s \in S$ if the fiber $\Psi_d(\eta)_s$ has dimension 0. If $\eta$ has the right codimension at all points $s \in S$, we say that $\eta$ has the \emph{right codimension}.
\end{definition}

\subsection{The discriminants agree}\label{section:discs_agree}
We check that the discriminants on $\mathcal{X}_d$ and $\Covers_d$ agree.

\begin{proposition}\label{proposition:disc_agrees}
The pullback $\phi_d^*\Delta_d^{\Cov}$ agrees with $\Delta_d$. More precisely, there is an isomorphism of the relevant line bundles on $\mathcal{X}_d$ that takes one section to the other.
\end{proposition}

\begin{proof}
Recall that in all the above morphisms $\mathcal{X}_d \to \Covers_d$, we have a natural isomorphism $\mathcal{F}_{d - 1} \cong (\mathcal{O}_Y/\mathcal{O}_S)^{\vee}$. Thus, $\det\mathcal{O}_Y^{\univ}$ and $(\det\mathcal{F}_{d - 1}^{\univ})^{\vee}$ are isomorphic, and the pullback $\phi_d^*\mathcal{L}_d^{\Cov} = \phi^*(\det\mathcal{O}_Y^{\univ})^{-2}$ is isomorphic to the discriminant line bundle $\mathcal{L}_d = (\det\mathcal{F}_{d - 1}^{\univ})^2$ on $\mathcal{X}_d$. Here, $\mathcal{O}_Y^{\univ}$ denotes the structure sheaf of the universal degree $d$ cover on $\Covers_d$, and $\mathcal{F}_{d - 1}^{\univ}$ denotes the universal $\mathcal{F}_{d - 1}$-bundle on $\mathcal{X}_d$.

It remains to show that the $\phi_d^*\Delta_d^{\Cov} = \Delta_d$ as sections of $\mathcal{L}_d$. By the definition of $\mathcal{L}_d$, we can identify the global sections $H^0(\mathcal{X}_d,\mathcal{L}_d)$ with the functions $f$ on $V_d$ such that for $g \in G_d(S)$ and $v \in V_d(S)$, $f(gv) = \chi_d(g)^2f(v)$. Recall that $\Delta_d$ generates the 1-dimensional $\mathbb{C}$-vector space of functions with this property on $(V_d)_{\mathbb{C}}$. Because the coordinate ring of $V_d$ injects into that of $(V_d)_{\mathbb{C}}$, we see that $\phi_d^*\Delta_d^{\Cov} = a\Delta_d$ for some $a \in \mathbb{C}$. To show that $a = 1$, we just need to check that $\phi_d^*\Delta_d^{\Cov}$ and $\Delta_d$ agree and are nonzero when evaluated at a single element $v \in V_d(\mathbb{Z})$. In all the following examples, $v$ has the right codimension, so the cover of $\Spec\mathbb{Z}$ corresponding to $v$ is actually just $\Psi_d(v)$ by \Cref{theorem:lvw_gorenstein}(b) below: \begin{itemize}
\item $d = 3$: The binary cubic form $v = (x^2y - xy^2) \otimes (x \wedge y)^{\vee}$ cuts out the finite \'{e}tale $\mathbb{Z}$-scheme $\Psi_3(v) = \{(1:0),(0:1),(1:1)\} \subset \mathbb{P}_{\mathbb{Z}}^2$. We check by direct computation that $\Delta_3(v) = 1$. Because $\Psi_3(v) \to \Spec\mathbb{Z}$ is \'{e}tale, its discriminant is 1, independent of basis\footnote{This just means that to calculate the discriminant, we do not have to check what basis of $\mathcal{O}_{\Psi_3(v)}/\mathbb{Z}$ we get from the frame $x,y$ of $\mathcal{F}_2$. This is because changing the $\mathbb{Z}$-basis multiplies the discriminant by the square of a unit, which must be 1.}. Thus, $a = 1$.

\item $d = 4$: The pair of ternary quadratic forms $v = (x_0x_2 - x_1x_2,x_0x_1 - x_1x_2)$ cuts out the finite \'{e}tale $\mathbb{Z}$-scheme $\{(1:0:0),(0:1:0),(0:0:1),(1:1:1)\} \subset \mathbb{P}_{\mathbb{Z}}^2$. We check by direct compuation that $\Delta_4(v) = 1$. Because $\Psi_4(v) \to \Spec\mathbb{Z}$ is \'{e}tale, its discriminant is 1, independent of basis. Thus, $a = 1$.

\item $d = 5$: The $5 \times 5$ alternating matrix $$v = \begin{pmatrix}0 & x_0 - x_3 & 0 & x_3 & x_2 \\ -x_0 + x_3 & 0 & 0 & x_3 & x_3 \\ 0 & 0 & 0 & x_1 & x_2 \\ -x_3 & -x_3 & -x_1 & 0 & -x_3 \\ -x_2 & -x_3 & -x_2 & x_3 & 0\end{pmatrix} \otimes (x_0 \wedge x_1 \wedge x_2 \wedge x_3)^{\vee}$$ has $4 \times 4$ Pfaffians $x_1x_3 - x_2x_3,-x_1x_2 + x_2x_3,-x_0x_3 + x_2x_3,-x_0x_2 + x_2x_3,x_0x_1 - x_1x_3$, which cut out the finite \'{e}tale $\mathbb{Z}$-scheme $\{(1:0:0:0),(0:1:0:0),(0:0:1:0),(0:0:0:1),(1:1:1:1)\} \subset \mathbb{P}_{\mathbb{Z}}^3$. We check using the description of $\Delta_5$ in \cite[\S4]{b_hcl_iv} that $\Delta_5(v) = 1$. Because $\Psi_5(v) \to \Spec\mathbb{Z}$ is \'{e}tale, its discriminant is 1, independent of basis. Thus, $a = 1$.
\end{itemize}
We conclude that $\phi_d^*\Delta_d^{\Cov} = \Delta_d$.
\end{proof}

\begin{remark}\label{remark:why_prehomog}
One important advantage of $\mathcal{X}_d$ over $\Covers_d$ is that $\mathcal{X}_d$ is smooth. For $d \ge 4$, the stacks $\Covers_d$ are not smooth by \cite[Proposition 8.4]{p_commalg}, so the moduli stacks of maps from curves to $\Covers_d$, as studied in \cite{d_compact}, are no longer nice (as far as we know). If we instead consider stacks of maps from curves to $\mathcal{X}_d$, we get smoothness for our stacks, among other useful properties (\Cref{theorem:rd_nice}).
\end{remark}

\subsection{Parametrization of low degree Gorenstein covers}\label{section:gorenstein_covers}
We say that a finite locally free map of schemes $Y \to S$ is \emph{Gorenstein} if all its fibers are Gorenstein. In particular, if $S$ is Gorenstein, a finite cover $Y \to S$ is Gorenstein if $Y$ is Gorenstein. When we restrict to Gorenstein covers, we do not need to deal with resolvents.

By \cite[Lemma 4.4]{lvw_low_degree_hurwitz}, there is an open substack $\Covers_d^{\Gor} \subset \Covers_d$ parametrizing degree $d$ Gorenstein covers. By the upper semicontinuity of fiber dimension, there is a $G_d$-invariant open subset $U_d \subset V_d$ parametrizing sections with the right codimension. We have the following beautiful parametrization of low degree Gorenstein covers, first proved by Casnati-Ekedahl over a Noetherian integral base and then extended by Landesman-Vakil-Wood to an arbitrary base.

\begin{theorem}[Casnati-Ekedahl, Landesman-Vakil-Wood]\label{theorem:lvw_gorenstein}
\begin{enumerate}[(a)]
\item Over $\Covers_d^{\Gor}$, $\phi_d$ restricts to an isomorphism $[U_d/G_d] \xrightarrow[]{\sim} \Covers_d^{\Gor}$.

\item Suppose we have a morphism $S \to [U_d/G_d]$ classifying a section $\eta$ of the right codimension. Then the map from the complex \eqref{equation:complex_d=3} (if $d = 3$), \eqref{equation:complex_d=4} (if $d = 4$), or \eqref{equation:complex_d=5} (if $d = 5$) to its 0th homology $\mathcal{O}_{\Psi_d(\eta)}$ is a quasi-isomorphism, i.e. the complex is a resolution of $\mathcal{O}_{\Psi_d(\eta)}$.
\end{enumerate}
\end{theorem}

\begin{proof}
Part (a) is \cite[Proposition 4.7]{lvw_low_degree_hurwitz}. Part (b) follows from the proofs of \cite[Theorem 3.13]{lvw_low_degree_hurwitz}, \cite[Theorem 3.14]{lvw_low_degree_hurwitz}, and \cite[Theorem 3.16]{lvw_low_degree_hurwitz} for $d = 3$, $d = 4$, and $d = 5$, respectively.
\end{proof}

\begin{corollary}\label{corollary:vss_bsd}
The morphism $\phi_d: \mathcal{X}_d \to \Covers_d$ restricts to an isomorphism $\mathcal{E}_d = [V_d^{ss}/G_d] \xrightarrow[]{\sim} B\mathbf{S}_d$. In particular, for any scheme $S$ and any vector $v \in V_d^{ss}(S)$, the map $(G_d)_S \to (V_d)_S$ given by $g \mapsto gv$ is \'{e}tale, and the differential $(\mathfrak{g}_d)_S \to (V_d)_S$ given by $x \mapsto xv$ is an isomorphism.
\end{corollary}

\begin{proof}
Because \'{e}tale schemes over a field are Gorenstein, the \'{e}tale locus $B\mathbf{S}_d \subset \Covers_d$ is contained in $\Covers_d^{\Gor}$. By \Cref{proposition:disc_agrees}, the preimage $\phi_d^{-1}(B\mathbf{S}_d)$ is the complement of the zero locus of $\Delta_d$, which is $\mathcal{E}_d$. By \Cref{theorem:lvw_gorenstein}(a), $\phi_d$ restricts to an isomorphism $\mathcal{E}_d \xrightarrow[]{\sim} B\mathbf{S}_d$.

The fact that $(G_d)_S \to (V_d)_S$ is \'{e}tale follows because it is the top map in the following Cartesian square, where the morphism $S \to (B\mathbf{S}_d)_S$ is the composite $S \xrightarrow[]{v} (V_d^{ss})_S \to [(V_d^{ss})_S/(G_d)_S] \xrightarrow[]{\sim} B\mathbf{S}_d$:
\begin{center}
\begin{tikzcd}
(G_d)_S \arrow[dd] \arrow[r] & (V_d^{ss})_S \arrow[d]                     \\
                             & {[(V_d^{ss})_S/(G_d)_S]} \arrow[d, "\sim"] \\
S \arrow[r]                  & (B\mathbf{S}_d)_S                          
\end{tikzcd}
\end{center}
Then the differential $(\mathfrak{g}_d)_S \to (V_d)_S$ is a map of vector bundles on $S$ that is an isomorphism on each fiber by the \'{e}taleness of $(G_d)_S \to (V_d)_S$, so $(\mathfrak{g}_d)_S \to (V_d)_S$ is an isomorphism.
\end{proof}

\begin{corollary}\label{corollary:gorenstein_quartic_unique}
If $Q \to S$ is a quartic Gorenstein cover, then there is a cubic resolvent $(Q,C,\phi,\delta)$, unique up to unique isomorphism inducing the identity on $Q$.
\end{corollary}

\begin{proof}
By construction, the morphism $\phi_d: \mathcal{X}_4 \to \Covers_4$ factors as $\mathcal{X}_4 \xrightarrow[]{\sim} \CovRes_4 \to \Covers_4$, where the first morphism is Wood's isomorphism from \cite{w_quartic_arbitrary} and the second morphism forgets the cubic resolvent. By \Cref{theorem:lvw_gorenstein}(a), the second morphism becomes an isomorphism above $\Covers_4^{\Gor}$, proving the corollary.
\end{proof}

\section{Prehomogeneous vector spaces over local fields}\label{section:prehomog_local}
In this section, we will discuss certain integrals on prehomogeneous vector spaces over local fields related to Igusa zeta functions. In this paper, we will mainly work with local fields of positive characteristic $> d$, but the results below apply equally to $p$-adic fields with $p > d$. The first part of our discussion follows \cite[\S 7.1]{bsw_geometry}.

Fix a nonarchimedean local field $K$ of characteristic not dividing $d!$. Fix the following notation: \begin{itemize}
\item $\mathcal{O}_K$ is the ring of integers of $K$.

\item $k$ is the residue field of $K$.

\item $q$ is the cardinality of $k$.

\item $v_K$ is the valuation on $K$ with image $\mathbb{Z}$.

\item $|\cdot| = q^{-v_K(\cdot)}$ is the normalized absolute value on $K$.
\end{itemize}
We will use $G_d$ and $V_d$ to denote $(G_d)_K$ and $(V_d)_K$, respectively. Pick nonzero left-invariant top forms $\omega_{G_d}$ and $\omega_{V_d}$ on $G_d$ and $V_d$ giving rise to Haar measures $\mu_{G_d}$ and $\mu_{V_d}$, respectively. Because $G_d$ is reductive, $\omega_{G_d}$ is in fact bi-invariant; because $V_d$ is commutative, $\omega_{V_d}$ is bi-invariant as well. For any element $v \in V_d(\mathcal{O}_K)$ with $\Delta_d(v) \ne 0$, consider the map $\pi_v: G_d \to V_d$ given by $g \mapsto gv$, which is \'{e}tale by \Cref{corollary:vss_bsd}. The top form $g \mapsto (\pi_v^*\omega_V)(g)/\chi_d(g)^2$ on $G_d$ is a nonzero left-invariant form because $G_d$ acts on $\det V_d$ by $\chi_d^2$. Thus, $(\pi_v^*\omega_{V_d})(g) = J_d(v)\chi_d(g)^2\omega_{G_d}(g)$ for a rational function $J_d$ on $V_d$. We have $(\pi_{hv}^*\omega_{V_d})(g) = (\pi_v^*\omega_{V_d})(gh)$. Meanwhile, \begin{align*}
(\pi_{hv}^*\omega_{V_d})(g) &= J_d(hv)\chi_d(g)^2\omega_{G_d}(g) \\
(\pi_v^*\omega_{V_d})(gh) &= J_d(v)\chi_d(gh)^2\omega_{G_d}(gh) \\
&= \chi_d(h)^2J_d(v)\chi_d(g)^2\omega_{G_d}(g).
\end{align*}
Thus, $J_d(hv) = \chi_d(h)^2J_d(v)$, and $J_d$ must be equal to $\mathcal{J}_d\Delta_d$ for a nonzero constant $\mathcal{J}_d \in K^{\times}$. Thus, $$\pi_v^*\omega_{V_d} = \mathcal{J}_d\Delta_d(v)\chi_d^2\omega_{G_d}.$$

We pick $\omega_{G_d}$ so that $\omega_{G_d}(1) = e_1 \wedge \cdots \wedge e_r$ for a basis $e_1,\ldots,e_r$ of $\mathfrak{g}_d^*(\mathcal{O}_K)$ and do the same thing for $\omega_{V_d}$. This normalization makes it so that $d\mu_{V_d}(V_d(\mathcal{O}_K)) = 1$. Moreover, if we pick $v \in V_d(\mathcal{O}_K)$ such that $v_K(\Delta_d(v)) = 0$, then the differential $d\pi_v: \mathfrak{g}_d(\mathcal{O}_K) \to V_d(\mathcal{O}_K)$ is an isomorphism because it induces an isomorphism $\mathfrak{g}_d(k) \to V_d(k)$ by \Cref{corollary:vss_bsd}. Thus, with this normalization, $v_K(\mathcal{J}_d) = 0$.

For each $b \in \mathbb{N}$, let $V_d(\mathcal{O}_K)_b$ denote the open subset of $v \in V_d(\mathcal{O}_K)$ such that $v_K(\Delta_d(v)) = b$. Note that each $V_d(\mathcal{O}_K)_b$ is $G_d(\mathcal{O}_K)$-invariant. Moreover, each $G_d(\mathcal{O}_K)$-orbit in $V_d^{ss}(K)$ is open, and $V_d(\mathcal{O}_K)_b$ is compact, so $V_d(\mathcal{O}_K)_b$ is a finite union of $G_d(\mathcal{O}_K)$-orbits.

\begin{proposition}\label{proposition:counting_integral}
For $s \in \mathbb{C}$ with $\mathrm{Re}(s) \ge 1$, we have $$\sum_{b = 0}^{\infty}\#\left(V_d(\mathcal{O}_K)_b // G_d(\mathcal{O}_K)\right)q^{-bs} = \frac{1}{\mu_{G_d}(G_d(\mathcal{O}_K))}\int_{V_d(\mathcal{O}_K)}|\Delta_d(v)|^{s - 1}d\mu_{V_d}(v).$$
\end{proposition}

\begin{proof}
We can write $$\int_{V_d(\mathcal{O}_K)}|\Delta_d(v)|^{s - 1}d\mu_{V_d}(v) = \sum_{b = 0}^{\infty}\int_{V_d(\mathcal{O}_K)_b}q^{-b(s - 1)}d\mu_{V_d}(v).$$ Each $V_d(\mathcal{O}_K)_b$ is a finite union of $G_d(\mathcal{O}_K)$-orbits $V_d(\mathcal{O}_K)_b = \bigsqcup_uG_d(\mathcal{O}_K)u$, so $$\int_{V_d(\mathcal{O}_K)_b}q^{-b(s - 1)}d\mu_{V_d}(v) = q^{-b(s - 1)}\sum_u\int_{G_d(\mathcal{O}_K)u}d\mu_{V_d}.$$ Recall the formula $\pi_v^*\omega_{V_d} = \mathcal{J}_d\Delta_d(v)\chi_d^2\omega_{G_d}$. Since $|\mathcal{J}_d| = 1$ and $|\chi_d(G_d(\mathcal{O}_K))| = 1$, $$\int_{G_d(\mathcal{O}_K)u}d\mu_{V_d} = \frac{1}{\#\stab_{G_d(\mathcal{O}_K)}(u)}\int_{G_d(\mathcal{O}_K)}q^{-b}d\mu_{G_d} = \frac{q^{-b}\mu_{G_d}(G_d(\mathcal{O}_K))}{\#\stab_{G_d(\mathcal{O}_K)}(u)}.$$ Putting this all together, we see that \begin{align*}
\int_{V_d(\mathcal{O}_K)}|\Delta_d(v)|^{s - 1}d\mu_{V_d}(v) &= \mu_{G_d}(G_d(\mathcal{O}_K))\sum_{b = 0}^{\infty}\sum_u\frac{q^{-bs}}{\#\stab_{G_d(\mathcal{O}_K)}(u)} \\
&= \mu_{G_d}(G_d(\mathcal{O}_K))\sum_{b = 0}^{\infty}\#\left(V_d(\mathcal{O}_K)_b // G_d(\mathcal{O}_K)\right)q^{-bs}.
\end{align*}

Note that the right-hand side converges for $\mathrm{Re}(s) \ge 1$ because $V_d(\mathcal{O}_K)$ has finite measure and $|\Delta_d(v)| \le 1$ for all $v \in V_d(\mathcal{O}_K)$. Since all the terms on both sides are nonnegative, we have established convergence for the left-hand side as well for $\mathrm{Re}(s) \ge 1$.
\end{proof}

\begin{proposition}\label{proposition:measure_of_gd}
Let $(a)$ denote $1 - q^{-a}$. With $\mu_{G_d}$ and $\mu_{V_d}$ normalized as above,
\begin{align*}
\mu_{G_3}(G_3(\mathcal{O}_K)) &= (1)(2) \\
\mu_{G_4}(G_4(\mathcal{O}_K)) &= (1)(2)^2(3) \\
\mu_{G_5}(G_5(\mathcal{O}_K)) &= (1)(2)^2(3)^2(4)^2(5)
\end{align*}
\end{proposition}

\begin{proof}
By the proof of \Cref{proposition:counting_integral}, $$\#(V_d(\mathcal{O}_K)_0 // G_d(\mathcal{O}_K)) = \frac{1}{\mu_{G_d}(
G_d(\mathcal{O}_K))}\int_{V_d(\mathcal{O}_K)_0}|\Delta_d(v)|^{-1}d\mu_{V_d}(v) = \frac{\mu_{V_d}(V_d(\mathcal{O}_K)_0)}{\mu_{G_d}(G_d(\mathcal{O}_K))}.$$ Because $V_d(\mathcal{O}_K)_0 = V_d^{ss}(\mathcal{O}_K)$ and $[V_d^{ss}/G_d] \cong B\mathbf{S}_d$ (\Cref{corollary:vss_bsd}), the left-hand side $\#(V_d(\mathcal{O}_K)_0 // G_d(\mathcal{O}_K))$ counts degree $d$ \'{e}tale $\mathcal{O}_K$-algebras inversely weighted by automorphisms, which is equal to the weighted count of degree $d$ \'{e}tale $k$-algebras. For any finite field, this is equal to 1. This can be proven either by direct counting in the cases $d \in \{3,4,5\}$ or by a generating function argument for any $d$. Hence, $\mu_{G_d}(G_d(\mathcal{O}_K)) = \mu_{V_d}(V_d(\mathcal{O}_K)_0)$.

Because of the normalization $\mu_{V_d}(V_d(\mathcal{O}_K)) = 1$ and the condition that $v_K(\Delta_d(v)) = 0$ only depends on the image of $v$ in $V_d(k)$, $\mu_{V_d}(V_d(\mathcal{O}_K)_0)$ is equal to the fraction of points $\overline{v} \in V_d(k)$ for which $\Delta_d(\overline{v}) \ne 0$. This density has been computed for degrees $d \in \{3,4,5\}$ in \cite[Lemma 1]{dh_density_discriminants_cubic}, \cite[Lemma 23]{b_hcl_iii}, and \cite[Lemma 20]{b_hcl_iv}, so we get the proposition. In each paper, the authors work with $\mathbb{F}_p$ because they are interested in counting extensions of $\mathbb{Q}$, but their methods extend to $\mathbb{F}_q$.
\end{proof}

Alternatively, as suggested to us by Will Sawin, \Cref{proposition:measure_of_gd} can be proven by counting the $k$-points of $G_d$ and dividing by $q^{\dim G_d}$.

\subsection{Igusa zeta functions}\label{section:igusa_zeta_functions}
For the rest of the paper, let $$I_d(s) = \frac{1}{\mu_{G_d}(G_d(\mathcal{O}_K))}\int_{V_d(\mathcal{O}_K)}|\Delta_d(v)|^{s - 1}d\mu_{V_d}(v)$$ with the normalizations above. In particular, $\mu_{V_d}(V_d(\mathcal{O}_K)) = 1$, and $\mu_{G_d}(G_d(\mathcal{O}_K))$ is the relevant constant from \Cref{proposition:measure_of_gd}.

We recall Igusa's computations of two local zeta functions relevant to us. The coefficients of these Igusa zeta functions, which count algebras over local fields, will be relevant later when we want to count algebras over global fields. Let $t = q^{-s}$ and $(a) = 1 - q^{-a}$. Let $f(q,t) = 1 + t^2 + t^3 + t^4 - 2t^5 + 2qt^6 + (q - 1)t^7 + qt^8 - qt^9 + (q - 1)qt^{10} - 2qt^{11} + 2q^2t^{12} - q^2t^{13} - q^2t^{14} - q^2t^{15} - q^2t^{17}$. We note that $f(q,t)$ is irreducible.

\begin{theorem}[{\cite{i_stationary_phase}}]\label{theorem:igusa_computation_3_and_4}
\begin{align*}
I_3(s) &= \frac{1 + t + t^2 + t^3 + t^4}{(1 - t^2)(1 - qt^6)} \\
I_4(s) &= \frac{f(q,t)}{(1 - t)(1 - t^2)(1 - qt^6)(1 - q^2t^8)(1 - q^3t^{12})}
\end{align*}
\end{theorem}

Note that for $d \in \{3,4\}$, $I_d(s)$ is dependent only on $q$ and $t$, so we will write $I_d(q,t)$ instead. By some computation, it can be shown that all the coefficients of the power series $I_d(q,t)$ are nonnegative integers.

By \Cref{proposition:counting_integral} and \Cref{proposition:measure_of_gd}, $I_d(q,t)$ is equal to $\sum_{b = 0}^{\infty}\#\left(V_d(\mathcal{O}_K)_b // G_d(\mathcal{O}_K)\right)t^b$, so we know exactly what the counts $\#\left(V_d(\mathcal{O}_K)_b // G_d(\mathcal{O}_K)\right)$ are for $d \in \{3,4\}$. From the discussion in \Cref{section:prehomog_cov_res}, we get the following:

\begin{corollary}\label{corollary:local_cover_count}
\begin{enumerate}[(a)]
\item Let $N_{3,\mathcal{O}_K,b}^{\loc}$ be the number of isomorphism classes of cubic $\mathcal{O}_K$-algebras of discriminant $q^b$, inversely weighted by automorphisms over $\mathcal{O}_K$. Then $$\sum_{b = 0}^{\infty}N_{3,\mathcal{O}_K,b}^{\loc}t^b = I_3(q,t) = \frac{1 + t + t^2 + t^3 + t^4}{(1 - t^2)(1 - qt^6)}.$$

\item Let $N_{4,\mathcal{O}_K,b}^{\loc}$ be the number of isomorphism classes of quartic $\mathcal{O}_K$-algebras with cubic resolvent of discriminant $q^b$, inversely weighted by automorphisms over $\mathcal{O}_K$. Then $$\sum_{b = 0}^{\infty}N_{4,\mathcal{O}_K,b}^{\loc}t^b = I_4(q,t) = \frac{f(q,t)}{(1 - t)(1 - t^2)(1 - qt^6)(1 - q^2t^8)(1 - q^3t^{12})},$$ where $f(q,t) = 1 + t^2 + t^3 + t^4 - 2t^5 + 2qt^6 + (q - 1)t^7 + qt^8 - qt^9 + (q - 1)qt^{10} - 2qt^{11} + 2q^2t^{12} - q^2t^{13} - q^2t^{14} - q^2t^{15} - q^2t^{17}$.
\end{enumerate}
\end{corollary}

Unfortunately, the Igusa zeta function for $(G_5,V_5)$ has not been computed. Thus, many of the results of this paper will be for $d \in \{3,4\}$. However, as we will explain in \Cref{section:conj_igusa5}, computing the Igusa zeta function for the degree 5 prehomogeneous vector space would enable us to extend these results to $d = 5$.

\begin{remark}\label{remark:dimension_matching}
As we hinted towards in \Cref{section:nichols_bd}, there is a mysterious ``coincidence'' where the dimension of $V_d$ is the top degree of $\mathfrak{B}_d$. We will see in \Cref{theorem:inv_coh_bd_34} that for $d \in \{3,4\}$, $I_d(q^{-2},qt)$ is the generating function for the geometric $\mathbf{S}_d$-invariant cohomology of $\mathfrak{B}_d$, providing a connection between $(G_d,V_d)$ and $\mathfrak{B}_d$. The strange thing is that the dimension of $V_d$ and top degree of $\mathfrak{B}_d$ are not apparently related to the Igusa zeta function and cohomology, respectively.
\end{remark}

\section{Orbinodal curves}\label{section:orbinodal_curves}
Following \cite{d_compact}, we want our stacks of covers with resolvents $\mathcal{R}^d$ to include covers of orbinodal curves. In this section, we will recall some information about orbinodal curves and the relevant stacks. Some references for orbinodal curves include \cite[\S 4]{av_compact} and \cite{o_twisted}, where they are referred to as ``twisted curves.''

We first recall the definition of an orbinodal curve. Recall that a Deligne-Mumford stack $\mathcal{Y}$ is \emph{tame} if the automorphism group of any geometric point $x \in \mathcal{Y}(k)$ has order not divisible by $\mathrm{char}(k)$. Note that we require orbinodal curves to be balanced and do not require them to have geometrically connected fibers.

\begin{definition}\label{definition:orbinodal_curve}
Let $S$ be a scheme. An \emph{$n$-pointed orbinodal curve} over $S$ consists of the data $(\mathcal{C} \to C \to S;\sigma_1,\ldots,\sigma_n)$, where \begin{enumerate}[(1)]
\item $C \to S$ is a nodal curve, and the $\sigma_i$ are pairwise disjoint sections $S \to C$.

\item $\mathcal{C} \to S$ is a tame Deligne-Mumford stack with coarse moduli space $\mathcal{C} \to C$. The morphism $\mathcal{C} \to C$ is an isomorphism over the open subset $C^{\gen} \subset C$ defined as the complement of the nodes and the marked points.

\item Let $c$ be a geometric point of $\mathcal{C}$ mapping to a node of $C \to S$, and let $s$ be its image in $S$. Then there exists an open neighborhood $T \subset S$ of $s$, an \'{e}tale neighborhood $U \to \mathcal{C}$ of $c$ over $T$, a function $t \in \mathcal{O}_T$, and an integer $n \ge 1$ for which we have the following Cartesian diagram:
\begin{center}
\begin{tikzcd}
\mathcal{C} \times_C U \arrow[r] \arrow[d, "\text{\'{e}t}"]          & U \arrow[d, "\text{\'{e}t}"]                      \\
{[\mathrm{Spec}\left(\frac{\mathcal{O}_T[u,v]}{uv-t}\right)/\mu_n]} \arrow[r] & {\mathrm{Spec}\left(\frac{\mathcal{O}_T[x,y]}{xy-t^n}\right)}
\end{tikzcd}
\end{center}
Here, the $\mu_n$-action on $\mathrm{Spec}\mathcal{O}_T[u,v]/(uv-t)$ is given by $(\zeta,u) \mapsto \zeta u$ and $(\zeta,v) \mapsto \zeta^{-1}v$, and the bottom morphism is induced by the the map sending $x \mapsto u^n$ and $y \mapsto v^n$.

\item Let $s$ be a geometric point of $S$, and let $c = \sigma_i(s)$. Then there exists an \'{e}tale neighborhood $U \to C$ of $c$ and an integer $n \ge 1$ for which we have the following Cartesian diagram:
\begin{center}
\begin{tikzcd}
\mathcal{C} \times_C U \arrow[d, "\text{\'{e}t}"] \arrow[r] & U \arrow[d, "\text{\'{e}t}"]    \\
{[\mathrm{Spec}\mathcal{O}_S[u]/\mu_n]} \arrow[r]           & {\mathrm{Spec}\mathcal{O}_S[x]}
\end{tikzcd}
\end{center}
Here, the $\mu_n$-action on $\mathrm{Spec}\mathcal{O}_S[u]$ is given by $(\zeta,u) \mapsto \zeta u$, and the bottom morphism is induced by the map sending $x \mapsto u^n$.
\end{enumerate}
We will abbreviate $(\mathcal{C} \to C \to S;\sigma_1,\ldots,\sigma_n)$ as $(\mathcal{C} \to C \to S;\sigma)$ or $(\mathcal{C} \to C;\sigma)$. We will often use the term ``smooth orbicurve'' or ``orbicurve'' when $\mathcal{C} \to S$ is smooth.

A 1-morphism $(\mathcal{C}_1 \to C_1;\{\sigma_{1j}\}) \to (\mathcal{C}_2 \to C_2;\{\sigma_{2j}\})$ is a 1-morphism $\mathcal{C}_1 \to \mathcal{C}_2$ over $S$ that takes $\sigma_{1j}$ to $\sigma_{2j}$. A 2-morphism of two 1-morphisms $(\mathcal{C}_1 \to C_1;\{\sigma_{1j}\}) \rightrightarrows (\mathcal{C}_2 \to C_2;\{\sigma_{2j}\})$ is a 2-morphism of the underlying 1-morphisms $\mathcal{C}_1 \rightrightarrows \mathcal{C}_2$.
\end{definition}

\begin{remark}\label{remark:no_2_morphisms}
Conveniently, \cite[Lemma 4.2.2]{av_compact} states that the 2-category of orbinodal curves is equivalent to a 1-category, i.e. all 2-automorphism groups of a 1-morphism are trivial. Thus, from now on, we will use ``morphism'' to refer to an isomorphism class of 1-morphisms, and we will make no reference to 2-morphisms.
\end{remark}

Because the morphism $\mathcal{C} \to C$ is an isomorphism over $C^{\gen}$, we will often write $C^{\gen}$ to denote its preimage in $\mathcal{C}$. We wil use this convention for subsets of $C^{\gen}$ as well, such as the $U_i$ not containing any nodes or marked points in \Cref{section:rd_smooth}.

\subsection{Stacks of nodal and orbinodal curves}\label{section:stacks_nodal_orbinodal}
Our constructions of stacks of covers with resolvents will lie over the following stacks of nodal and orbinodal curves.

\begin{definition}[{\cite[Definition 3.1]{d_compact}}]\label{definition:node_stack}
Define the stack $\mathcal{M}^{*,n}$ of divisorially marked, $n$-pointed nodal curves as the stack over $\mathbb{Z}$ whose objects over $S$ are tuples $(C \to S;\Sigma;\sigma_1,\ldots,\sigma_n)$, where \begin{enumerate}[(1)]
\item $C$ is an algebraic space and $C \to S$ is a nodal curve with geometrically connected fibers.

\item $\Sigma \subset C$ is a relative effective Cartier divisor over $S$ lying in the smooth locus of $C \to S$.

\item $\sigma_j: S \to C$ are relatively pairwise disjoint sections lying in the smooth locus of $C \to S$ that do not intersect $\Sigma$.
\end{enumerate}
We will abbreviate $(C \to S;\Sigma;\sigma_1,\ldots,\sigma_n)$ as $(C \to S;\Sigma;\sigma)$. Let $\mathcal{M}$ denote the disjoint union of all the $\mathcal{M}^{*,n}$, and let $\mathcal{M}^{\sm} \subset \mathcal{M}$ denote the open substack parametrizing smooth curves. Let $\mathcal{M}^{b,n}$ denote the open and closed substack of $\mathcal{M}$ where the degree of the marked divisor is $b$ and there are $n$ marked points.
\end{definition}

Deopurkar's proof of the following result occurs over $\mathbb{C}$, but it works exactly the same over $\mathbb{Z}$.

\begin{proposition}[{\cite[Proposition 3.2]{d_compact}}]\label{proposition:deopurkar_m_smooth}
$\mathcal{M}$ is a smooth algebraic stack over $\mathbb{Z}$.
\end{proposition}

\begin{definition}\label{definition:orbi_stack}
Define the stack $\mathcal{M}^{\orb;*,n}$ of divisorially marked, $n$-pointed orbinodal curves as the stack over $\mathbb{Z}$ whose objects over $S$ are tuples $(\mathcal{C} \to C \to S;\Sigma;\sigma_1,\ldots,\sigma_n))$, where \begin{enumerate}[(1)]
\item $(\mathcal{C} \to C \to S;\sigma)$ is an $n$-pointed orbinodal curve with geometrically connected fibers.

\item $\Sigma \subset C$ is a relative effective Cartier divisor over $S$ lying in the smooth locus of $C \to S$ and disjoint from the sections $\sigma$.
\end{enumerate}
Let $\mathcal{M}^{\orb}$ denote the disjoint union of all the $\mathcal{M}^{\orb;*,n}$. For $N \in \mathbb{N}_{> 0}$, let $\mathcal{M}^{\orb\le N}$ denote the full subcategory of $\mathcal{M}^{\orb}$ where all the automorphism groups of $\mathcal{C}$ have order bounded above by $N$. Given an $n$-tuple of positive integers $\underline{a} = (a_1,\ldots,a_n)$, let $\mathcal{M}^{\orb;*,n}(\underline{a})$ denote the open and closed substack of $\mathcal{M}^{\orb;*,n}$ where the automorphism groups of the geometric points of $\mathcal{C}$ mapping to $\sigma_i$ have order $a_i$.
\end{definition}

\begin{remark}\label{remark:convention_geo_conn}
For convenience, from now on, whenever we write a curve $C \to S$ or an orbinodal curve $\mathcal{C} \to C \to S$, we will implicitly require the fibers to be geometrically connected unless specified otherwise. This is because many of the stacks we will construct will be over $\mathcal{M}$.
\end{remark}

Given any stack $\mathcal{Y}$ over $\mathcal{M}$, we write $\mathcal{Y}^{b,n}$ to denote the preimage of $\mathcal{M}^{b,n}$ and $\mathcal{Y}^{\sm}$ to denote the preimage of $\mathcal{M}^{\sm}$. If $\mathcal{Y}$ is over $\mathcal{M}^{\orb}$, we write $\mathcal{Y}^{*,n}(\underline{a})$ to denote the preimage of $\mathcal{M}^{\orb;*,n}(\underline{a})$. We have a morphism $\mathcal{M} \to \mathcal{M}^{0,*}$ forgetting the marked divisor. For the stacks of orbinodal curves, we have a morphism $\mathcal{M}^{\orb} \to \mathcal{M}^{\orb;0,*}$ forgetting the marked divisor and a morphism $\mathcal{M}^{\orb} \to \mathcal{M}$ sending $(\mathcal{C} \to C \to S;\Sigma;\sigma) \to (C \to S;\Sigma;\sigma)$. We observe that $\mathcal{M}^{\orb}$ is the open substack of the 2-fiber product $\mathcal{M}^{\orb;0,*} \times_{\mathcal{M}^{0,*}} \mathcal{M}$ where the marked divisor does not intersect the marked points. Similarly, given any stack $\mathcal{Y} \to \mathcal{M}^{\orb}$, we write $\mathcal{Y}^{\le N}$ for the preimage of $\mathcal{M}^{\orb\le N}$.

The following result of Olsson will be very useful when we prove algebraicity and boundedness for our stacks of covers with resolvents in \Cref{section:covers_resolvents}.

\begin{theorem}[{\cite[Theorem 1.9, Corollary 1.11]{o_twisted}}]\label{theorem:olsson_twisted}
\begin{enumerate}[(a)]
\item $\mathcal{M}^{\orb;0,*}$ is a smooth algebraic stack over $\mathbb{Z}$. The morphism $\mathcal{M}^{\orb;0,*} \to \mathcal{M}^{0,*}$ is representable by Deligne-Mumford stacks.

\item The morphism $\mathcal{M}^{\orb\le N;0,*} \to \mathcal{M}^{0,*}$ is of finite type.
\end{enumerate}
\end{theorem}

\begin{corollary}\label{corollary:orb_finite}
\begin{enumerate}[(a)]
\item $\mathcal{M}^{\orb}$ is a smooth algebraic stack over $\mathbb{Z}$. The morphism $\mathcal{M}^{\orb} \to \mathcal{M}$ is representable by Deligne-Mumford stacks.

\item For any $N$, the morphism $\mathcal{M}^{\orb\le N} \to \mathcal{M}$ is of finite type.
\end{enumerate}
\end{corollary}

\begin{proof}
The smoothness in (a) follows from the smoothness of $\mathcal{M}^{\orb;0,*}$ and the fact that the marked divisor must remain in the smooth locus, so that its deformations are unobstructed. Each remaining statement follows from the corresponding part of \Cref{theorem:olsson_twisted} and the fact that $\mathcal{M}^{\orb}$ is an open substack of $\mathcal{M}^{\orb;0,*} \times_{\mathcal{M}^{0,*}} \mathcal{M}$.
\end{proof}

\subsection{Stacky structures on smooth curves}\label{section:stacky_smooth}
Smooth orbicurves admit a simple description: a smooth orbicurve is characterized by its underlying $n$-pointed smooth curve and $n$ positive integers corresponding to the automorphism groups above the marked points. The construction that lets us recover the smooth orbicurve from these data is the root stack, introduced in \cite{c_using_stacks}. We will only discuss root stacks in the context of smooth orbicurves, but the construction works in much more general contexts.

We describe the construction. Fix an $n$-tuple $\underline{a} = (a_1,\ldots,a_n)$ of positive integers, and let $N = a_1\cdots a_n$. Suppose we have a smooth $n$-pointed curve $(C \to S;\sigma)$, where $S$ is a scheme with $N$ invertible. The $n$ relative effective Cartier divisors $D_i = \sigma_i(S)$ induce a morphism $C \to [\mathbb{A}^1/\mathbb{G}_m]^n \cong [\mathbb{A}^n/\mathbb{G}_m^n]$, where the $n$th morphism $C \to [\mathbb{A}^1/\mathbb{G}_m]$ corresponds to the line bundle $\mathcal{O}_C(D_i)$ with section dual to the inclusion $\mathcal{O}_C(-D_i) \to \mathcal{O}_C$. There is an $\underline{a}$th power map $\theta_{\underline{a}}: [\mathbb{A}^n/\mathbb{G}_m^n] \to [\mathbb{A}^n/\mathbb{G}_m^n]$, which is the $a_i$th power map $\theta_{a_i}: [\mathbb{A}^1/\mathbb{G}_m] \to [\mathbb{A}^1/\mathbb{G}_m]$ on the $i$th coordinate. The \emph{$\underline{a}$th root stack associated to $(C \to S;\sigma)$} is the 2-fiber product $\mathcal{C}_{\underline{a}} \coloneqq C \times_{[\mathbb{A}^n/\mathbb{G}_m^n],\theta_{\underline{a}}} [\mathbb{A}^n/\mathbb{G}_m^n]$.

\begin{theorem}[{\cite[Theorem 4.1]{c_using_stacks}}]\label{theorem:cadman_root_equivalence}
The $\underline{a}$th root stack $\mathcal{C}_{\underline{a}}$ is a smooth orbicurve over $S$ with coarse moduli space $\mathcal{C}_{\underline{a}} \to C$. The functor $C \mapsto \underline{C}_{\underline{a}}$ induces an isomorphism $\mathcal{M}_{\mathbb{Z}\left[\frac{1}{N}\right]}^{\sm;0,*} \xrightarrow[]{\sim} \mathcal{M}_{\mathbb{Z}\left[\frac{1}{N}\right]}^{\orb,\sm;0,*}(\underline{a})$ inverse to the map $\mathcal{M}_{\mathbb{Z}\left[\frac{1}{N}\right]}^{\orb,\sm;0,*}(\underline{a}) \to \mathcal{M}_{\mathbb{Z}\left[\frac{1}{N}\right]}^{\sm;0,*}$.
\end{theorem}

\begin{proof}
We note that \cite[Theorem 4.1]{c_using_stacks} is stated only for a Noetherian base $S$. However, we can generalize to arbitrary $S$ by using \cite[Tag 062Y]{stacks_project} in place of \cite[Lemma 5.1]{c_using_stacks} in the proof.
\end{proof}

Because the marked divisor of an orbinodal curve is disjoint from the stacky part, we get the following corollary.

\begin{corollary}\label{corollary:root_equivalence_divisor}
The map $\mathcal{M}_{\mathbb{Z}\left[\frac{1}{N}\right]}^{\orb,\sm}(\underline{a}) \to \mathcal{M}_{\mathbb{Z}\left[\frac{1}{N}\right]}^{\sm}$ is an isomorphism.
\end{corollary}

\begin{remark}\label{remark:smooth_no_autos}
Implicit in \Cref{theorem:cadman_root_equivalence} is the fact that a smooth orbicurve $\mathcal{C} \to S$ has no nontrivial automorphisms that induce the identity on the coarse moduli space $C \to S$. This can also be seen from \cite[Proposition 7.1.1]{acv_twisted_bundles} and the fact that $\mathcal{M}^{\orb;0,*} \to \mathcal{M}^{0,*}$ is representable by Deligne-Mumford stacks. In later sections, we will make use of this fact without explicit reference.
\end{remark}

\subsubsection{Framings for smooth orbicurves}\label{section:framings_smooth}
We finish this section by introducing stacks of smooth orbicurves/curves with framings of the normal bundles to the marked points. These stacks will be useful when we construct stacks of marked covers with resolvents and marked Hurwitz spaces in \Cref{section:related_stacks}.

\begin{definition}\label{definition:fr_orbi_stack}
Define the stack $\mathcal{M}^{\fr;*,n}$ as the stack over $\mathcal{M}^{*,n}$ whose objects over $S$ consist of the following data: \begin{enumerate}[(1)]
\item An object $(C \to S;\Sigma;\sigma) \in \mathcal{M}^{*,n}(S)$.

\item For each $i$, a trivialization $\alpha_i: \sigma_i^{-1}\mathcal{O}_C(D_i) \xrightarrow[]{\sim} \mathcal{O}_S$ on $S$, where $D_i = \sigma_i(S)$.
\end{enumerate}
We will abbreviate these data as $(C \to S;\Sigma;\sigma;\alpha)$. A morphism in $\mathcal{M}^{\fr;*,n}$ is a morphism of objects of $\mathcal{M}^{*,n}$ compatible with the framings $\alpha$. Let $\mathcal{M}^{\fr}$ denote the disjoint union of all the $\mathcal{M}^{\fr;*,n}$.
\end{definition}

Note that $\mathcal{M}^{\fr;*,n} \to \mathcal{M}^{*,n}$ is the principal $\mathbb{G}_m^n$-bundle associated to the direct sum of the $n$ universal line bundles $\sigma_i^{-1}\mathcal{O}_C(D_i)$. Thus, $\mathcal{M}^{\fr;*,n}$ is a smooth algebraic stack over $\mathbb{Z}$ by \Cref{proposition:deopurkar_m_smooth}.

Let $\underline{a} = (a_1,\ldots,a_n)$, and let $S$ be a scheme on which $N = a_1\cdots a_n$ is invertible. The reason we consider $\mathcal{M}^{\fr}$ is that for an object $(C \to S;\Sigma;\sigma;\alpha) \in \mathcal{M}^{\fr,\sm;*,n}(S)$, the framing data gives us $n$ natural sections $S \to \mathcal{C}_{\underline{a}}$, which are defined as follows. Consider the two composites \begin{align*}
S \xrightarrow[]{\sigma_i} C &\to [\mathbb{A}^n/\mathbb{G}_m^n] \\
S \to [\mathbb{A}^n/\mathbb{G}_m^n] &\xrightarrow[]{\theta_{\underline{a}}} [\mathbb{A}^n/\mathbb{G}_m^n]
\end{align*}
Here, the map $C \to [\mathbb{A}^n/\mathbb{G}_m^n]$ is the one considered above and the map $S \to [\mathbb{A}^n/\mathbb{G}_m^n]$ in the second composite classifies $n$ trivial line bundle with the zero section in the $i$th bundle and the identity sections in the other bundles. The first composite classifies the $n$-tuple of line bundles $(\mathcal{O}_S,\ldots,\mathcal{O}_S,\sigma_i^{-1}\mathcal{O}_C(D_i),\mathcal{O}_S,\ldots,\mathcal{O}_S)$ where $\sigma_i^{-1}\mathcal{O}_C(D_i)$ is in the $i$th position, along with the zero section in the $i$th bundle and the identity sections in the other bundles. The second composite classifies $(\mathcal{O}_S,\ldots,\mathcal{O}_S)$ with the zero section in the $i$th bundle and the identity sections in the other bundles. Thus, $\alpha$ provides a 2-isomorphism of the two composites and gives rise to a section $S \to \mathcal{C}_{\underline{a}_i} = C \times_{[\mathbb{A}^n/\mathbb{G}_m^n]} [\mathbb{A}^n/\mathbb{G}_m^n]$. This construction yields a canonical morphism $\psi_{i,\underline{a}}^{\univ}: \mathcal{M}_{\mathbb{Z}\left[\frac{1}{N}\right]}^{\fr,\sm;*,n} \to \mathcal{C}_{\underline{a},\mathbb{Z}\left[\frac{1}{N}\right]}^{\univ}$, where $\mathcal{C}_{\underline{a}}^{\univ}$ is the universal orbicurve over $\mathcal{M}^{\orb,\sm;*,n}(\underline{a})$.

\section{The stacks of covers with resolvents $\mathcal{R}^d$}\label{section:covers_resolvents}
In this section, we construct the main geometric objects of this paper $\mathcal{R}^d$ and prove some important geometric properties about them. Unless specified otherwise, for the rest of the paper, we will work over $\mathbb{S}_d \coloneqq \Spec\mathbb{Z}\left[\frac{1}{d!}\right]$, as the invertibility of $d!$ will ensure that we only need to deal with tame ramification and tame stacks. The following definition is nearly identical to \cite[Definition 3.3]{d_compact}, except that we replace $\Covers_d$ with $\mathcal{X}_d$.

\begin{definition}\label{definition:cov_res}
Define the \emph{big stack of degree $d$ covers with resolvents} $\mathcal{R}^d$ as the disjoint union of the stacks $\mathcal{R}^{d;*,n}$ over $\mathbb{S}_d$ whose objects over $S$ are $$\mathcal{R}^{d;*,n}(S) \coloneqq \{(\mathcal{C} \to C \to S;\sigma_1,\ldots,\sigma_n;\chi: \mathcal{C} \to \mathcal{X}_d)\},$$ where \begin{enumerate}[(1)]
\item $(\mathcal{C} \to C \to S;\sigma)$ is an $n$-pointed orbinodal curve.

\item $\chi: \mathcal{C} \to \mathcal{X}_d$ is a representable morphism that maps the following to $\mathcal{E}_d$ for each fiber $\mathcal{C}_s$ of $\mathcal{C} \to S$: the generic points of the components of $\mathcal{C}_s$, the nodes of $\mathcal{C}_s$, and the preimages of the marked points in $\mathcal{C}_s$.
\end{enumerate}

A morphism between $(\mathcal{C}_1 \to P_1 \to S_1;\{\sigma_{1j}\};\chi_1: \mathcal{C}_1 \to \mathcal{X}_d)$ and $(\mathcal{C}_2 \to P_2 \to S_2;\{\sigma_{2j}\};\chi_2: \mathcal{C}_2 \to \mathcal{X}_d)$ over a morphism $S_1 \to S_2$ consists of a pair $(F,\alpha)$, where \begin{enumerate}[(1)]
\item $F$ is a morphism of pointed orbinodal curves: $F: \mathcal{C}_1 \to \mathcal{C}_2$ such that we get a Cartesian diagram

\begin{center}
\begin{tikzcd}
\mathcal{C}_1 \arrow[r, "F"] \arrow[d] & \mathcal{C}_2 \arrow[d] \\
S_1 \arrow[r]                          & S_2                    
\end{tikzcd}
\end{center}

\item $\alpha: \chi_1 \to \chi_2 \circ F$ is a 2-morphism.
\end{enumerate}
We abbreviate $(\mathcal{C} \to C \to S;\sigma_1,\ldots,\sigma_n;\chi: \mathcal{C} \to \mathcal{X}_d)$ by $(\mathcal{C} \to C;\sigma;\chi)$.
\end{definition}

Given an object $(\mathcal{C} \to C \to S;\sigma;\chi)$ of $\mathcal{R}^{d;*,n}(S)$, the branch locus $\chi^{-1}(\Sigma_d)$ is disjoint from the stacky points in $\mathcal{C}$, so we can identify it with its image in $C$, which we will denote $\br\chi$. Because $\br\chi$ is the zero locus of a section of a line bundle and its intersection with each fiber $P_s$ is an effective Cartier divisor, $\br\chi$ is a relative effective Cartier divisor in $C \to S$ by \cite[Tag 062Y]{stacks_project}. Thus, we have a \emph{branch morphism} $\br: \mathcal{R}^d \to \mathcal{M}$ defined by $$(\mathcal{C} \to C \to S;\sigma;\chi) \mapsto (C \to S;\br\chi;\sigma).$$ We will also refer to the map $\mathcal{R}^d \to \mathcal{M}^{\orb}$ as $\br$ as an abuse of notation.

\begin{remark}\label{remark:deopurkar_d=3}
The stack $\mathcal{X}_3$ is equivalent to the stack of degree 3 covers $\mathcal{A}_3$ considered in \cite{d_compact} and \cite{p_commalg}, so our construction is nothing new when $d = 3$ (other than extending Deopurkar's construction from $\mathbb{Q}$ to $\mathbb{Z}\left[\frac{1}{6}\right]$). However, we do get a new proof of smoothness for Deopurkar's degree 3 big Hurwitz stack $\mathcal{H}^3$ that also works for $\mathcal{R}^4$ and $\mathcal{R}^5$.
\end{remark}

\begin{remark}\label{remark:marked_rep}
Later in this paper, we will consider points $p \in \mathcal{M}(k)$ corresponding to curves of the form $(\mathbb{A}_k^1;\Sigma;\infty)$, i.e. $\infty$ will be the unique marked point and $\Sigma \in \Sym^n(\mathbb{A}^1)(k)$ will be some marked divisor. The fiber $\br^{-1}(p)$ will parametrize covers $D \to \mathbb{A}^1$ (if $d = 3$) or covers with resolvents $(D,R) \to \mathbb{A}^1$ (if $d \in \{4,5\}$) such that $\pi: D \to \mathbb{A}^1$ is \'{e}tale outside $\pi^{-1}(\Sigma)$.
\end{remark}

In the rest of \Cref{section:covers_resolvents}, we will prove the following theorem concerning $\mathcal{R}^d$ and the branch morphism $\br: \mathcal{R}^d \to \mathcal{M}$. The proof is very similar to that of \cite[Theorem 3.8]{d_compact}. However, Deopurkar works over a field of characteristic 0, so we will rewrite everything over $\mathbb{S}_d$.

\begin{theorem}\label{theorem:rd_nice}
\begin{enumerate}[(a)]
\item $\mathcal{R}^d$ is a smooth algebraic stack over $\mathbb{S}_d$.

\item The morphism $\br: \mathcal{R}^d \to \mathcal{M}$ is proper and representable by Deligne-Mumford stacks.

\item The composite $\mathcal{R}^d \xrightarrow[]{\br} \mathcal{M}^{\orb} \to \mathcal{M}^{\orb;0,*}$ is smooth. The restriction $\mathcal{R}^{d;b,*} \to \mathcal{M}^{\orb;0,*}$ has relative dimension $b$.
\end{enumerate}
\end{theorem}

From \Cref{theorem:cadman_root_equivalence} and \Cref{theorem:rd_nice}(c), we get the following.

\begin{corollary}\label{corollary:sm_smooth}
The composite $\mathcal{R}^{d,\sm} \xrightarrow[]{\br} \mathcal{M}^{\sm} \to \mathcal{M}^{\sm;0,*}$ is smooth. The restriction $\mathcal{R}^{d,\sm;b,*} \to \mathcal{M}^{\sm;0,*}$ has relative dimension $b$.
\end{corollary}

The proof of \Cref{theorem:rd_nice} is broken down into parts.

\subsection{Generating sheaves for tame Deligne-Mumford stacks}\label{section:gen_sheaf}
A large part of our proof of \Cref{theorem:rd_nice}(b) is based on Deopurkar's proof of \cite[Theorem 3.8]{d_compact} for his stacks of degree $d$ covers. Deopurkar works over a field of characteristic 0, whereas we will work over $\mathbb{S}_d$. We will show that much of his proof works the same way for our situation.

We will need some results from \cite{n_moduli} concerning moduli of sheaves on tame Deligne-Mumford stacks. We recall the notion of a generating sheaf, as generating sheaves are a necessary part of Deopurkar's proof of \cite[Proposition 4.12]{d_compact}, which appears in our paper as \Cref{proposition:deopurkar_4.12}. For more information, see \cite[\S 5]{os_quot} or \cite[\S 5.2]{k_geometry}.

\begin{definition}\label{definition:generating_sheaf}
Let $\mathcal{X}$ be a tame Deligne-Mumford stack with coarse moduli space $\rho: \mathcal{X} \to X$. A locally free sheaf $\mathcal{E}$ on $\mathcal{X}$ is a \emph{generating sheaf} for $\mathcal{X}$ if for any quasi-coherent sheaf $\mathcal{F}$ on $\mathcal{X}$, the morphism $$\rho^*\rho_*(\mathcal{H}om_{\mathcal{O}_{\mathcal{X}}}(\mathcal{E},\mathcal{F})) \otimes_{\mathcal{O}_{\mathcal{X}}} \mathcal{E} \to \mathcal{F}$$ is surjective.
\end{definition}

By \cite[Theorem 5.2]{os_quot}, a locally free sheaf $\mathcal{E}$ is a generating sheaf for $\mathcal{X}$ if and only if for every geometric point $x \in \mathcal{X}(k)$, the representation $\mathcal{E}_x$ of the automorphism group $G_x$ contains every irreducible representation of $G$.

We recall some helpful facts about families of nodal and orbinodal curves.

\begin{lemma}[{\cite[Proposition 2.1]{h_moduli}}]\label{lemma:hall_proj}
Let $S$ be a Noetherian scheme, and let $C \to S$ be a nodal curve. Then there exists an \'{e}tale cover $T \to S$ such that $C \times_S T \to T$ is projective.
\end{lemma}

\begin{lemma}[{\cite[Proposition 4.3]{d_compact}}]\label{lemma:deopurkar_proj}
Let $S$ be a Noetherian scheme and $(\mathcal{C} \to C \to S;\sigma)$ a tame pointed orbinodal curve. Then there exists an \'{e}tale cover $T \to S$ such that \begin{enumerate}[(1)]
\item $\mathcal{C}_T \coloneqq \mathcal{C} \times_S T$ admits a finite flat morphism from a projective $T$-scheme $Z$.

\item $\mathcal{C}_T$ is the quotient of a quasi-projective scheme by $\GL_{n,T}$.

\item $\mathcal{C}_T$ admits a generating sheaf.
\end{enumerate}
\end{lemma}

\begin{proof}
Deopurkar's proof of \Cref{lemma:deopurkar_proj} is for $S$ over a field of characteristic 0, so we comment on the modifications needed to prove the statement for arbitrary $S$. The facts needed for Deopurkar's proofs of (1) and (2) are \cite[Theorem 1.13]{o_twisted}, \cite[Theorem 2.14]{ehkv_brauer}, and \cite[Remark 4.3]{k_geometry}. The first two statements take place over an arbitrary Noetherian base. For the last statement, we note that Kresch works over a field, sometimes of characteristic 0. However, the reasoning in \cite[Remark 4.3]{k_geometry} works verbatim over an arbitrary Noetherian base. This is enough to prove (1) and (2). Now (3) follows from (2) and \cite[Theorem 5.5]{os_quot}.
\end{proof}

\subsection{$\mathcal{R}^d$ is algebraic and locally of finite type}\label{section:alg_locfin}
For each $d$, a morphism $\mathcal{C} \to \mathcal{X}_d = [V_d/G_d]$ from an orbinodal curve $\mathcal{C} \to S$ is given by a $G_d$-bundle over $\mathcal{C}$ along with a section of the associated vector bundle with fiber $V_d$. Moreover, for each $d$, a $G_d$-bundle is the same thing as the following data over $\mathcal{C}$: \begin{itemize}
\item $d = 3$: A rank 2 vector bundle $\mathcal{F}_2$.

\item $d = 4$: A rank 3 vector bundle $\mathcal{F}_3$, a rank 2 vector bundle $\mathcal{F}_2$, and a nonvanishing section $s \in H^0(\mathcal{C},\det\mathcal{F}_3 \otimes \det\mathcal{F}_2^{\vee})$.

\item $d = 5$: A rank 4 vector bundle $\mathcal{F}_4$, a rank 5 vector bundle $\mathcal{F}_5$, and a nonvanishing section $s \in H^0(\mathcal{C},(\det\mathcal{F}_4)^2 \otimes \det\mathcal{F}_5^{\vee})$.
\end{itemize}
Hence, in each case, a morphism $\mathcal{C} \to \mathcal{X}_d$ is given by some vector bundles on $\mathcal{C}$ and some sections of associated vector bundles.

Following \cite[\S 4.1]{d_compact}, for every $r \in \mathbb{N}_{> 0}$, we define the stack $\mathcal{V}ect^r$ over $\mathcal{M}^{\orb}$ whose objects over $S$ are $$\mathcal{V}ect^r(S) = \{(\mathcal{C} \to C \to S;\Sigma;\sigma;\mathcal{V})\},$$ where $\mathcal{V}$ is a rank $r$ vector bundle on $\mathcal{C}$ and the stack $\mathcal{S}ect^r$ over $\mathcal{V}ect^r$ whose objects over $S$ are $$\mathcal{S}ect^r(S) = \{(\mathcal{C} \to C \to S;\Sigma;\sigma;\mathcal{V};v)\},$$ where $v \in H^0(\mathcal{C},\mathcal{V})$ is a global section. We then have maps $$\mathcal{S}ect^r \to \mathcal{V}ect^r \to \mathcal{M}^{\orb} \to \mathcal{M}$$ along with the following fully faithful embeddings of the $\mathcal{R}^d$: \begin{itemize}
\item $d = 3$: $\mathcal{R}^3 \xhookrightarrow{} \mathcal{W}^3 \coloneqq \mathcal{V}ect^2 \times_{\mathcal{V}ect^4} \mathcal{S}ect^4$, where the morphism $\mathcal{V}ect^2 \to \mathcal{V}ect^4$ sends $\mathcal{F}_2 \mapsto \Sym^3\mathcal{F}_2 \otimes \det\mathcal{F}_2^{\vee}$.

\item $d = 4$: $\mathcal{R}^4 \xhookrightarrow{} \mathcal{W}^4 \coloneqq (\mathcal{V}ect^3 \times_{\mathcal{M}^{\orb}} \mathcal{V}ect^2) \times_{\mathcal{V}ect^1 \times_{\mathcal{M}^{\orb}} \mathcal{V}ect^{12}} (\mathcal{S}ect^1 \times \mathcal{S}ect^{12})$, where the morphism $\mathcal{V}ect^3 \times_{\mathcal{M}^{\orb}} \mathcal{V}ect^2 \to \mathcal{V}ect^1 \times_{\mathcal{M}^{\orb}} \mathcal{V}ect^{12}$ sends $(\mathcal{F}_3,\mathcal{F}_2) \mapsto (\det\mathcal{F}_3 \otimes \det\mathcal{F}_2^{\vee},\Sym^2\mathcal{F}_3 \otimes \mathcal{F}_2^{\vee})$.

\item $d = 5$: $\mathcal{R}^5 \xhookrightarrow{} \mathcal{W}^5 \coloneqq (\mathcal{V}ect^4 \times_{\mathcal{M}^{\orb}} \mathcal{V}ect^5) \times_{\mathcal{V}ect^1 \times_{\mathcal{M}^{\orb}} \mathcal{V}ect^{40}} (\mathcal{S}ect^1 \times \mathcal{S}ect^{40})$, where the morphism $\mathcal{V}ect^4 \times_{\mathcal{M}^{\orb}} \mathcal{V}ect^5 \to \mathcal{V}ect^1 \times_{\mathcal{M}^{\orb}} \mathcal{V}ect^{40}$ sends $(\mathcal{F}_4,\mathcal{F}_5) \mapsto ((\det\mathcal{F}_4)^2 \otimes \det\mathcal{F}_5^{\vee},\mathcal{F}_4 \otimes \det\mathcal{F}_4^{\vee} \otimes \wedge^2\mathcal{F}_5)$.
\end{itemize}

\begin{proposition}[{\cite[Proposition 4.4]
{d_compact}}]\label{proposition:vect_morb}
$\mathcal{V}ect^{r;\le d!}$ is an algebraic stack, locally of finite type over $\mathcal{M}^{\orb\le d!}$.
\end{proposition}

\begin{proof}
The proof is basically the same as that of \cite[Proposition 4.4]{d_compact}, but we will write it here for convenience.

By \Cref{corollary:orb_finite}, $\mathcal{M}^{\orb\le d!}$ is locally of finite type over $\mathbb{S}_d$. Hence, we just need to show that for a finite type affine $\mathbb{S}_d$-scheme $S$ and a morphism $S \to \mathcal{M}^{\orb\le d!}$, the 2-fiber product $\mathcal{V}ect^{r;\le d!} \times_{\mathcal{M}^{\orb\le d!}} S \to S$ is algebraic and locally of finite type. In other words, we want to prove that the stack $\mathcal{V}ect_{\mathcal{C}/S}^r$ whose objects over $T \to S$ are rank $d$ vector bundles on $\mathcal{C} \times_S T$ is algebraic and locally of finite type. For this purpose, we are allowed to replace $S$ by any \'{e}tale cover.

We have an open immersion $\mathcal{V}ect_{\mathcal{C}/S}^r \xhookrightarrow{} \mathcal{C}oh_{\mathcal{C}/S}$, where $\mathcal{C}oh_{\mathcal{C}/S}$ is the stack of coherent sheaves on $\mathcal{C}/S$. By \Cref{lemma:deopurkar_4.6} and \Cref{lemma:deopurkar_4.7}, after passing to an \'{e}tale cover of $S$, we can ensure that $C \to S$ is projective and $\mathcal{C} \to S$ is a global quotient. Then we can conclude using \cite[Corollary 2.27]{n_moduli} because $\mathcal{C}$ is tame.
\end{proof}

\begin{proposition}\label{proposition:sect_vect}
The morphism $\mathcal{S}ect^r \to \mathcal{V}ect^r$ is representable by algebraic spaces of finite type.
\end{proposition}

\begin{proof}
By \Cref{corollary:orb_finite} and \Cref{proposition:vect_morb}, $\mathcal{V}ect^r$ is locally of finite type over $\mathbb{S}_d$. Hence, we just need to check that for a finite type affine $\mathbb{S}_d$-scheme $S$ and a morphism $S \to \mathcal{V}ect^r$, the 2-fiber product $\mathcal{S}ect^r \times_{\mathcal{V}ect^r} S$ is an algebraic space of finite type over $S$.

Fix a smooth chart $S \to \mathcal{V}ect^r$ with $S$ of finite type over $\mathbb{S}_d$, and let $(\mathcal{C} \to C \to S;\Sigma;\sigma;\mathcal{V})$ be the corresponding object. Now the proposition follows from the following two lemmas. \Cref{lemma:deopurkar_4.6} is a generalization of \cite[Lemma 4.6]{d_compact} to arbitrary characteristic. \Cref{lemma:deopurkar_4.7} is \cite[Lemma 4.7]{d_compact}.

\begin{lemma}[{\cite[Lemma 4.6]{d_compact}} in arbitrary characteristic]\label{lemma:deopurkar_4.6}
Let $S$ be a Noetherian affine scheme and $\mathcal{X} \to S$ be a proper tame Deligne-Mumford stack with coarse moduli space $\rho: \mathcal{X} \to X$, where $X$ is a scheme. Let $\mathcal{F}$ be a coherent sheaf on $\mathcal{X}$, flat over $S$. Then there is a finite complex $M_{\bullet}$ of locally free sheaves on $S$ $$M_0 \to M_1 \to \cdots \to M_n$$ such that for every $f: T \to S$, there are natural isomorphisms $$H^i(T,f^*M_{\bullet}) \xrightarrow[]{\sim} H^i(\mathcal{X}_T,\mathcal{F}_T).$$
\end{lemma}

\begin{proof}
Let $F = \rho_*\mathcal{F}$. Then $F$ is a coherent sheaf on $X$. Once we show that $F$ is flat over $S$, the proof goes exactly the same way as the proof of \cite[Lemma 4.6]{d_compact}.

We will use the tameness to show that $F$ is flat. Because we can work locally, we may assume that $S = \Spec R$ and $X = \Spec A$ are affine. By \cite[Theorem 3.2]{aov_tame_stacks}, because $\mathcal{X}$ is tame, there exists an fppf cover $X^{\prime} \to X$ such $\mathcal{X} \times_X X^{\prime} \cong [U/G]$ for a finite flat linearly reductive group scheme $G \to X^{\prime}$ acting on a finite and finitely presented scheme $U \to X^{\prime}$. Hence, by replacing $X$ with $X^{\prime}$, we may assume that $\mathcal{X}$ is of the form $[U/G]$ as above. Because $U \to U/G \cong X$ is finite, $U$ is affine, and we can write $U = \Spec B$.

The sheaf $\mathcal{F}$ corresponds to a $G$-equivariant $B$-module $M$ flat over $R$. Then $F = \rho_*\mathcal{F}$ corresponds to the $A$-module $M^G$. For any $R$-module $N$, we have $M^G \otimes_R N \cong (M \otimes_R N)^G$, where the $G$-action on $M \otimes_R N$ is trivial on the $N$ factor. Because $M$ is flat over $R$, tensoring with $M$ is exact. Because $G$ is linearly reductive, taking $G$-invariants is exact. Hence, tensoring with $M^G$ is exact, and we conclude that $F$ is flat.
\end{proof}

\begin{lemma}[{\cite[Lemma 4.7]{d_compact}}]\label{lemma:deopurkar_4.7}
Let $\mathcal{X} \to S$ and $\mathcal{F}$ be as in \Cref{lemma:deopurkar_4.6}. Then the contravariant functor $\mathbf{Schemes}_S \to \mathbf{Sets}$ defined by $$(f: T \to S) \mapsto H^0(\mathcal{X}_T,\mathcal{F}_T)$$ is representable by an affine scheme of finite type over $S$.
\end{lemma}
\end{proof}

\begin{proposition}\label{proposition:open_imm}
The morphism $\mathcal{R}^d \to \mathcal{W}^d$ is an open immersion.
\end{proposition}

\begin{proof}
The conditions that certain points map to $\mathcal{E}_d$ is open because $\mathcal{E}_d$ is open in $\mathcal{X}_d$, and if $d \in \{4,5\}$, the condition that the section of the line bundle is nowhere vanishing is open. It remains to show that representability is an open condition.

For a morphism $S \to \mathcal{W}^d$ landing in the open substack where the section of the line bundle is nonvanishing, we get a morphism $\chi: \mathcal{C} \to \mathcal{X}_d$, where $\mathcal{C} \to S$ be the corresponding orbicurve. Let $\mathcal{I}_{\chi} \to \mathcal{C}$ be the relative inertia stack of $\chi$. Then $\mathcal{I}_{\chi} \to \mathcal{C}$ is representable and finite. The representable locus of $\chi$ is precisely the open subset of $\mathcal{C}$ above which $\mathcal{I}_{\chi} \to \mathcal{C}$ has fiber of degree 1. Because $\mathcal{C} \to S$ is proper, we have proven that representability is an open condition.
\end{proof}

\begin{proposition}\label{proposition:rd_alg_locfin}
$\mathcal{R}^d$ is algebraic and locally of finite type over $\mathbb{S}_d$.
\end{proposition}

\begin{proof}
By \Cref{proposition:vect_morb} and \Cref{proposition:sect_vect}, $\mathcal{W}^d$ is algebraic and locally of finite type over $\mathbb{S}_d$. We then conclude by \Cref{proposition:open_imm}.
\end{proof}

\subsection{$\br$ is of finite type}\label{section:br_fin}
The morphism $\br$ factors as \begin{itemize}
\item $d = 3$: $\mathcal{R}^3 \to \mathcal{V}ect^{2;\le 3} \to \mathcal{M}^{\orb\le 3} \to \mathcal{M}$

\item $d = 4$: $\mathcal{R}^4 \to \mathcal{V}ect^{3;\le 4} \times_{\mathcal{M}^{\orb\le 4}} \mathcal{V}ect^{2;\le 4} \to \mathcal{M}^{\orb\le 4} \to \mathcal{M}$

\item $d = 5$: $\mathcal{R}^5 \to \mathcal{V}ect^{4;\le 5} \times_{\mathcal{M}^{\orb\le 5}} \mathcal{V}ect^{5;\le 5} \to \mathcal{M}^{\orb\le 5} \to \mathcal{M}$
\end{itemize}
In each case, the first morphism is of finite type by \Cref{proposition:sect_vect} and \Cref{proposition:open_imm}, and the third morphism is of finite type by \Cref{corollary:orb_finite}(b). The second morphism is not of finite type in any of these cases, so we will show that $\br$ factors through a finite type substack of $\mathcal{V}ect$ or $\mathcal{V}ect \times \mathcal{V}ect$. As in \cite[\S 4.2]{d_compact}, the idea will be to find a uniform bound for the degree and $h^0$ of vector bundles that can occur. The degree of a vector bundle on an orbinodal curve is defined using the pullback to a smooth curve mapping to the normalization as in \cite[\S 7.2]{agv_gw_theory}, so degrees of vector bundles are in $\frac{1}{d!}\mathbb{Z}$. Let $\mathcal{V}ect^r_{l,N}$ denote the open substack of $\mathcal{V}ect^r$ parametrizing vector bundles with fiberwise degree $l$ and $h^0 \le N$. The key fact is the following proposition:

\begin{proposition}[{\cite[Proposition 4.12]{d_compact}}]\label{proposition:deopurkar_4.12}
The morphism $\mathcal{V}ect^{r;\le d!}_{l,N} \to \mathcal{M}^{\orb\le d!}$ is of finite type.
\end{proposition}

The proof of \cite[Proposition 4.12]{d_compact} uses the existence of generating sheaves\footnote{This was the entire reason for our discussion of generating sheaves in \Cref{section:gen_sheaf}.}, along with \cite[Lemma 4.13]{d_compact} and \cite[Remark 4.14]{d_compact}. We write a proof of \cite[Remark 4.14]{d_compact} for the reader's convenience.

\begin{lemma}[{\cite[Remark 4.14]{d_compact}}]\label{lemma:bounded_rank_deg_h0}
For a nodal curve $C \to S$ with $S$ Noetherian, a family of sheaves $(P_U,\mathcal{F})$ over an $S$-scheme $U$ is bounded if the degree, rank, and $h^0$ of each fiber $(P_u,\mathcal{F}_u)$ is bounded for all $u$ (by rank, we mean the maximum rank on a component.) More precisely, there exists a finite type $S$-scheme $T$ and a family of sheaves $(P_T,\mathcal{G})$ such that every geometric fiber $(P_{\overline{u}},\mathcal{F}_{\overline{u}})$ for $\overline{u} \to U$ occurs as a geometric fiber of $(P_T,\mathcal{G})$.
\end{lemma}

\begin{proof}
By \Cref{lemma:deopurkar_4.6}, we may replace $S$ with an \'{e}tale cover such that $C$ has an $S$-relatively very ample line bundle $\mathcal{O}_P(1)$. By taking $S$ connected, we may also assume that all the curves have the same arithmetic genus $g$.

After we fix $\mathcal{O}_P(1)$, bounding the degree, rank, and genus bounds the Hilbert polynomial. By the boundedness of Quot schemes of fixed Hilbert polynomial, it suffices to show that there exists a uniform $m$ such that for each $u \in U$, $\mathcal{F}_u(m)$ is globally generated and $h^1(P_u,\mathcal{F}_u(m)) = 0$. Indeed, every $\mathcal{F}_u$ is then be a quotient of $\mathcal{O}_{P_u}(-m) \otimes H^0(P_u,\mathcal{F}_u(m))$, and $h^0(P_u,\mathcal{F}_u(m))$ is bounded because $h^1(P_u,\mathcal{F}_u(m)) = 0$. We can then bound $(P_U,\mathcal{F})$ by the universal family on $\bigsqcup_{\Phi,h}C \times_S \Quot_{\mathcal{O}_P(-m)^{\oplus h}/C/S}^{\Phi,\mathcal{O}_P(1)}$, where $\Phi$ ranges over a bounded set of Hilbert polynomials and $h$ ranges over a bounded set of positive integers.

By \cite[Lemma 5.19(c)]{n_construction}, the vanishing of $H^1(P_u,\mathcal{F}_u(m))$ implies the global generation of $\mathcal{F}_u(m)$, so we actually just need to find an $m$ for which $h^1(P_u,\mathcal{F}_u(m)) = 0$. As explained in the proof of \cite[Theorem 5.3]{n_construction}, $h^1(P_u,\mathcal{F}_u(m))$ is strictly decreasing in $m$ until it reaches 0. Because we have a bound on the Hilbert polynomial of $\mathcal{F}_u$ and on $h^0(\mathcal{F}_u)$, we have a bound on $h^1(P_u,\mathcal{F}_u)$. Thus, there exists a uniform $m$ for which $h^1(P_u,\mathcal{F}_u(m)) = 0$, proving boundedness for $(P_U,\mathcal{F})$.
\end{proof}

\begin{proposition}\label{proposition:bound_deg_h0}
Fix a geometric point $\Spec k \to \mathcal{R}^{d;b,n}$ corresponding to $(\mathcal{C} \to C;\sigma;\chi)$, where $C$ has arithmetic genus $g$ and the normalization of $C$ has $c$ components. Then there are bounds on the degrees and $h^0$ of the associated bundles $\mathcal{F}_2$ (if $d = 3$), $\mathcal{F}_3$ and $\mathcal{F}_2$ (if $d = 4$), or $\mathcal{F}_4$ and $\mathcal{F}_5$ (if $d = 5$) depending only on $b$, $n$, $g$, and $c$.
\end{proposition}

\begin{proof}
We start by bounding the degree. We consider the cases separately: \begin{itemize}
\item $d = 3$: The discriminant is a section of $(\det\mathcal{F}_2)^2$, so $\deg\mathcal{F}_2 = \frac{b}{2}$.

\item $d = 4$: The discriminant is a section of $(\det\mathcal{F}_3)^2$, so $\deg\mathcal{F}_3 = \frac{b}{2}$ and $\deg\mathcal{F}_2 = \frac{b}{2}$.

\item $d = 5$: The discriminant is a section of $(\det\mathcal{F}_4)^2$, so $\deg\mathcal{F}_4 = \frac{b}{2}$ and $\deg\mathcal{F}_5 = b$.
\end{itemize}
Thus, the degrees are fixed by $b$.

To bound $h^0$, we first reduce to the case where $\mathcal{C}$ is a smooth projective curve. Let $\mathcal{C}^{\circ} \subset \mathcal{C}$ be the preimage $\chi^{-1}(\mathcal{E}_d)$. The map $\mathcal{C}^{\circ} \to \mathcal{E}_d$ corresponds to an $\mathbf{S}_d$-torsor $D^{\circ} \to \mathcal{C}^{\circ}$, where $D^{\circ}$ is a nodal curve because $\chi$ is representable. In a unique way, we can complete $D^{\circ}$ to a proper nodal curve $D$ and extend $D^{\circ} \to \mathcal{C}^{\circ}$ to a morphism $D \to \mathcal{C}$. Letting $\overline{D}$ be the normalization of $D$, we can consider the composite $\chi^{\prime}: \overline{D} \to D \to \mathcal{C} \xrightarrow[]{\chi} \mathcal{X}_d$. Denote the map $\overline{D} \to \mathcal{C}$ by $\nu$. For any of the vector bundles $\mathcal{F}$ on $\mathcal{C}$ whose $h^0$ we want to bound, we have an injection $\mathcal{F} \xhookrightarrow{} \nu_*\nu^*\mathcal{F}$, so $$h^0(\mathcal{C},\mathcal{F}) \le h^0(\overline{D},\nu^*\mathcal{F}) = \sum_ih^0(D_i,\nu^*\mathcal{F}|_{D_i}),$$ where $D_i$ ranges over the components of $\overline{D}$. Each $D_i$ has a generically \'{e}tale degree $\le d!$ map to a component $C_i$ of the normalization $\overline{C}$ whose genus is bounded above by $g$. This map is ramified at most above the preimages of the marked points, nodes, and branch divisor. The ramification above the marked points and nodes is bounded by the size of the automorphism group ($\le d!$), and the ramification above the branch divisor is bounded by $b$. Thus, by the Riemann-Hurwitz formula, the genera of the $D_i$ are uniformly bounded in terms of $b$, $n$, and $g$. The branch divisor for $\chi^{\prime}$ is just the pullback of the branch divisor for $\chi$, so it has degree $\le bd!$ on each $C_i$. Finally, there are at most $cd!$ components of $\overline{D}$, so $h^0(\mathcal{C},\mathcal{F}) \le cd!\max_ih^0(D_i,\nu^*\mathcal{F}|_{D_i})$. Putting this together, we see that bounds for the smooth non-stacky case imply bounds for the general case.

Now assume that $\mathcal{C} = C$ is a smooth projective curve. For this proof, we will use $G_d$ and $V_d$ to refer to $(G_d)_k$ and $(V_d)_k$, respectively. Let $P \to C$ be the principal $G_d$-bundle given by $\chi$. Let $B_d$ and $T_d$ be the standard Borel of (pairs of) upper triangular matrices and the standard maximal torus of (pairs of) diagonal matrices, respectively. Let $B_{d,u}$ denote the unipotent radical of $B_d$, so that there is a natural isomorphism $B_d/B_{d,u} \cong T_d$. Let $X^*(T_d)$ denote the character group of $T_d$, and let $X^*(G_d) \subset X^*(T_d)$ denote the character group of $G_d$. Note that $X^*(T_d)$ and $X^*(G_d)$ are independent of $k$, in the sense that there are natural identifications between them for different $k$.

By a standard result \cite[\S 2, Remark b)]{ds_b_structures} about the Galois cohomology of a function field, $P$ has a $B_d$-reduction $P^{\prime} \to C$, from which we get a principal $T_d$-bundle $P^{\prime}/B_{d,u} \to C$. By \cite[Lemma 3.2.8]{ckm_git}, there exists a uniform (i.e. independent of $P^{\prime}$, $P$, $C$, and $k$) $\mathbb{Q}$-basis $\{\theta_i\}$ of $X^*(T_d) \otimes \mathbb{Q}$ with $\{\theta_i\} \subset X^*(T_d)$ such that \begin{enumerate}[(1)]
\item For each $i$, the line bundle $$P^{\prime} \times_{B_d} \mathbb{A}_{\theta_i}^1 = P^{\prime}/B_{d,u} \times_{T_d} \mathbb{A}_{\theta_i}^1$$ on $C$ has nonnegative degree.

\item $\chi_d|_{T_d} = \sum_ia_i\theta_i$ for some positive rational numbers $a_i$.
\end{enumerate}
By (2), $$\frac{b}{2} = \deg(P \times_{G_d} \mathbb{A}_{\chi_d}^1) = \deg(P^{\prime} \times_{B_d} \mathbb{A}_{\sum_ia_i\theta_i}^1) = \sum_ia_i\deg(P^{\prime} \times_{B_d} \mathbb{A}_{\theta_i}^1).$$ Because all the terms are nonnegative, we have $\deg(P^{\prime} \times_{B_d} \mathbb{A}_{\theta_i}^1) \le \frac{b}{2a_i}$. Because the $a_i$ are uniform, we see that the $\deg(P^{\prime} \times_{B_d} \mathbb{A}_{\theta_i}^1)$ are uniformly bounded in terms of $b$. Because the $\theta_i$ are a $\mathbb{Q}$-basis of $X^*(T_d) \otimes \mathbb{Q}$, we have a uniform bound on $\deg(P^{\prime} \times_{B_d} \mathbb{A}_{\theta}^1)$ for any fixed character $\theta \in X^*(T_d)$ depending only on $b$. If we write $\mathcal{L}_{\theta}$ for the invertible sheaf corresponding to the line bundle $P^{\prime} \times_{B_d} \mathbb{A}_{\theta}^1$, we have a uniform bound on $h^0(C,\mathcal{L}_{\theta})$ in terms of $b$ and $g$.

For a fixed representation $W$ of $G_d$ defined over $\mathbb{S}_d$, since $P \times_{G_d} W \cong P^{\prime} \times_{B_d} W$, we can write the locally free sheaf $\mathcal{F}$ on $C$ corresponding to $P \times_{G_d} W$ as a successive extension of invertible sheaves $\mathcal{L}_{\theta}$, where the $\theta$ are the $T_d$-weights appearing in the representation $W$. We have uniform bounds on $h^0(C,\mathcal{L}_{\theta})$, so we have a uniform bound on $h^0(C,\mathcal{F})$ in terms of $b$ and $g$. All the associated bundles in the proposition arise in this way for standard $G_d$-representations $W$, so we have bounded their $h^0$.
\end{proof}

\begin{remark}\label{remark:ckm_3.2.8}
We comment on the use of \cite[Lemma 3.2.8]{ckm_git} in the proof of \Cref{proposition:bound_deg_h0}. The proof of \cite[Lemma 3.2.8]{ckm_git} assumes we are in the setup of \cite[Theorem 3.2.5]{ckm_git}, which is exactly the situation we are in. In \cite{ckm_git}, the authors work over $\mathbb{C}$. However, the proof of \cite[Lemma 3.2.8]{ckm_git} works over an arbitrary algebraically closed field $k$ and produces $\theta_i$ independent of $k$, assuming we start with split $G$ and $V$ defined over a open subscheme $\mathbb{S} \subset \Spec\mathbb{Z}$. The key point of the proof of \cite[Lemma 3.2.8]{ckm_git} is that for a given character $\theta \in X^*(T)$, there exists an open cone $A \subset X^*(T) \otimes \mathbb{Q}$ containing $\theta$ such that for all $\xi \in A$, $V^s(T,\theta) \subset V^{ss}(T,\xi)$. By a spreading out argument, the cone constructed by the proof in the characteristic 0 case works for all but finitely many characteristics. However, for each of these characteristics, we can find an open cone, and then we can pick $\theta_i$ in the intersection of these finitely many cones that work for any characteristic.

There is actually a typo in the statement of \cite[Lemma 3.2.8]{ckm_git}: in condition (2), the $a_i$ should be positive, not just nonnegative. This change is necessary for its application in the proof of \cite[Theorem 3.2.5]{ckm_git}. Thankfully, assuming we have $\theta_i$ with nonnegative coefficients $a_i$, we can use the following argument to get new $\theta_i$ such that the $a_i$ become positive. If some $a_i$ are 0, pick some positive $a_j$ and replace $\theta_j$ with $\theta_j - \epsilon\sum_i\theta_i$ for some small rational $\epsilon > 0$. Then scale the new $\theta_i$ so that they all become integral.
\end{remark}

\begin{proposition}\label{proposition:br_fin_type}
The morphism $\br: \mathcal{R}^d \to \mathcal{M}$ is of finite type.
\end{proposition}

\begin{proof}
Because $\mathcal{M}$ is locally of finite type over $\mathbb{S}_d$, it suffices to show that for $S$ a Noetherian scheme with a morphism $S \to \mathcal{M}$ classifying a family of nodal curves $C \to S$, the morphism $S \times_{\mathcal{M}} \mathcal{R}^d \to S$ is of finite type. Because $S$ is Noetherian, we have bounds on the quantities $b$, $n$, $g$, and $c$ appearing in \Cref{proposition:bound_deg_h0}, so we have bounds on the degree and $h^0$ of the bundles appearing in $S \times_{\mathcal{M}} \mathcal{R}^d$. Thus, we conclude by \Cref{proposition:deopurkar_4.12} and the discussion at the beginning of this section.
\end{proof}

\subsection{$\br$ is Deligne-Mumford}\label{section:br_dm}
\begin{proposition}\label{proposition:br_dm}
The morphism $\br: \mathcal{R}^d \to \mathcal{M}$ is of Deligne-Mumford type.
\end{proposition}

\begin{proof}
By \Cref{theorem:olsson_twisted}(a), it suffices to show that $\mathcal{R}^d \to \mathcal{M}^{\orb\le d!}$ is of Deligne-Mumford type. This is equivalent to saying that the morphism has unramified inertia, which can be checked on the level of points. Let $(\mathcal{C} \to C \to \Spec k;\sigma;\chi)$ be a geometric point of $\mathcal{R}^d$. In this proof, we will write $\mathcal{X}_d$ to denote $(\mathcal{X}_d)_k$ and do the same thing for $G_d$, $V_d$, etc.

We will show that the morphism $\chi: \mathcal{C} \to \mathcal{X}_d$ has no infinitesimal automorphisms by showing that it has no first-order infinitesimal automorphisms (these correspond to the Lie algebra of the automorphism group scheme). The morphism $\chi: \mathcal{C} \to \mathcal{X}_d$ classifies a principal $G_d$-bundle $P \to \mathcal{C}$ and a section $s: \mathcal{C} \to P \times_{G_d} V_d$ corresponding to a $G_d$-equivariant morphism $f: P \to V_d$, and a second-order infinitesimal automorphism of $\chi$ is a first-order infinitesimal automorphism of $P$ that preserves $s$. First-order infinitesimal automorphisms of $P$ are classified by $H^0(\mathcal{C},P \times_{G_d} \mathfrak{g}_d)$, and the ones that preserve $s$ are given by $\ker(H^0(\mathcal{C},P \times_{G_d} \mathfrak{g}_d) \to H^0(\mathcal{C},P \times_{G_d} V_d))$, where the map $P \times_{G_d} \mathfrak{g}_d \to P \times_{G_d} V_d$ sends $(p,x) \mapsto (p,xf(p))$.

Recall that $\chi$ maps an open dense subset of $\mathcal{C}$ to $\mathcal{E}_d = [V_d^{ss}/G_d]$ and that for any point $v \in V_d^{ss}(k)$, the map $\mathfrak{g}_d \to V_d$ given by $x \mapsto xv$ is an isomorphism (\Cref{corollary:vss_bsd}). This means that the morphism of vector bundles $P \times_{G_d} \mathfrak{g}_d \to P \times_{G_d} V_d$ is injective, so the map on $H^0$ is injective. We conclude that there are no infinitesimal automorphisms of $\chi$ and that $\br$ is of Deligne-Mumford type.
\end{proof}

\subsection{$\br$ is proper}\label{section:br_proper}
Because $\br$ is of finite type, we can use the valuative criterion for properness \cite[Tag 0CLZ]{stacks_project}. Let $R$ be a strictly Henselian DVR over $\mathcal{M}$ with fraction field $K$ and residue field $k$. Denote the special and generic points of $\Spec R$ by 0 and $\eta$, respectively.

\begin{proposition}\label{proposition:br_uniqueness_part}
$\br$ satisfies the uniqueness part of the valuative criterion.
\end{proposition}

\begin{proof}
Let $(\mathcal{C}_i \to C \to \Spec R;\sigma;\chi_i)$ (for $i \in \{1,2\}$) be two objects of $\mathcal{R}^d(R)$ over an object $(C \to \Spec R;\Sigma;\sigma)$ of $\mathcal{M}(R)$. We have an isomorphism $(\mathcal{C}_1)_{\eta} 
 \xrightarrow[]{\sim} (\mathcal{C}_2)_{\eta}$ over $C$ such that the composite $(\mathcal{C}_1)_{\eta} \xrightarrow[]{\sim} (\mathcal{C}_2)_{\eta} \xrightarrow[]{(\chi_2)_{\eta}} \mathcal{X}_d$ is isomorphic to $(\chi_1)_{\eta}$, and we want to show that this isomorphism extends to an isomorphism $\mathcal{C}_1 \xrightarrow[]{\sim} \mathcal{C}_2$ such that the composite $\mathcal{C}_1 \xrightarrow[]{\sim} \mathcal{C}_2 \xrightarrow[]{\chi_2} \mathcal{X}_d$ is isomorphic to $\chi_1$.
 
Once we find such an extension of the isomorphism over $C^{\gen}$, we can find an extension over all of $C$ by copying Steps 2-5 of the proof of \cite[Proposition 4.18]{d_compact}. Indeed, the complements $\mathcal{C}_i - C^{\gen}$ map to $\mathcal{E}_d \cong B\mathbf{S}_d$ (\Cref{corollary:vss_bsd}), which puts us in the same situation as \cite[Proposition 4.18]{d_compact}.

We will find the extension to $C^{\gen}$ in two steps. The morphisms $\mathcal{C}_i \to C$ restrict to isomorphisms over $C^{\gen}$. Hence, we just want to show that two objects $\chi_i: C^{\gen} \to \mathcal{X}_d$ that restrict to isomorphic objects $C_{\eta}^{\gen} \to \mathcal{X}_d$ are isomorphic. An object $\chi_i: C^{\gen} \to \mathcal{X}_d$ consists of a $G_d$-bundle $P_i \to C^{\gen}$ and a section $s_i \in H^0(C^{\gen},P_i \times_{G_d} V_d)$. The restrictions $\chi_i: C^{\gen\circ} \coloneqq C^{\gen} - \chi_1^{-1}(\Sigma_d) - \chi_2^{-1}(\Sigma_d) \to \mathcal{X}_d$ land in $\mathcal{E}_d \cong B\mathbf{S}_d$, so they classify $\mathbf{S}_d$-torsors $D_i \to C^{\gen\circ}$. When we pass to the generic point, the $D_i \to C^{\gen\circ}$ become isomorphic as $\mathbf{S}_d$-torsors $(D_i)_{\eta} \to C_{\eta}^{\gen\circ}$. Because $C^{\gen\circ}$ is regular, $D_i$ is the normalization of $C^{\gen\circ}$ in $(D_i)_{\eta}$; in particular, $D_1$ and $D_2$ are isomorphic over $C^{\gen\circ}$. Thus, we have an extension of the isomorphism $(\chi_1)_{\eta} \xrightarrow[]{\sim} (\chi_2)_{\eta}$ to $C_{\eta}^{\gen} \cup C^{\gen\circ}$, and it remains to extend the isomorphism to the finite set $C^{\gen} - (C_{\eta}^{\gen} \cup C^{\gen\circ})$.

For each $d$, an object $C^{\gen} \to \mathcal{X}_d$ is the data of one or two vector bundles over $C^{\gen}$ and one or two sections of some associated bundles. Because $C^{\gen}$ is normal and $C^{\gen} - (C_{\eta}^{\gen} \cup C^{\gen\circ})$ has codimension 2, Hartogs's theorem implies that any isomorphism of such data over $C_{\eta}^{\gen} \cup C^{\gen\circ}$ extends to an isomorphism over $C^{\gen}$. Thus, we can find an extension from $C_{\eta}^{\gen}$ to $C^{\gen}$.
\end{proof}

For finishing the proof of properness, we need the following theorem of Horrocks.

\begin{theorem}[{\cite[Corollary 4.1.1]{h_vector_bundles}}]\label{theorem:horrocks_reg_dim2}
Let $S = \Spec R$ for a regular local ring $R$ of dimension 2, and let $S^{\circ}$ be the punctured spectrum ($S$ minus its closed point). Then every vector bundle on $S^{\circ}$ is trivial.
\end{theorem}

\begin{corollary}\label{corollary:dim2_gd}
With the same notation, every principal $G_d$-bundle on $S^{\circ}$ is trivial.
\end{corollary}

\begin{proof}
For $d = 3$, there is nothing to do because $G_3 = \GL_2$. For $d \in \{4,5\}$, a principal $G_d$-bundle on $S^{\circ}$ is the same as the data of two vector bundles $\mathcal{F},\mathcal{F}^{\prime}$ on $S^{\circ}$ and a nonvanishing section $s \in H^0(S^{\circ},(\det\mathcal{F})^i \otimes \det\mathcal{F}^{\prime\vee})$ for some $i \in \mathbb{Z}$. By \Cref{theorem:horrocks_reg_dim2}, $\mathcal{F}$ and $\mathcal{F}^{\prime}$ are trivial, and we can pick an automorphism of $\mathcal{F}^{\prime}$ that scales $s$ so that we get a trivial $G_d$-bundle.
\end{proof}

\begin{proposition}\label{proposition:br_existence_part}
$\br$ satisfies the existence part of the valuative criterion.
\end{proposition}

\begin{proof}
The proof is identical to that of \cite[Proposition 4.20]{d_compact}, except that we use \Cref{corollary:dim2_gd} in place of \Cref{theorem:horrocks_reg_dim2} in Step 4 when $d \in \{4,5\}$. We can use Abhyankar's lemma as in Step 1 of Deopurkar's proof because we are assuming that $d!$ is invertible.
\end{proof}

\begin{corollary}\label{corollary:br_proper}
The morphism $\br: \mathcal{R}^d \to \mathcal{M}$ is proper.
\end{corollary}

This concludes the proof of \Cref{theorem:rd_nice}(b).

\subsection{$\mathcal{R}^d \to \mathcal{M}^{\orb;0,*}$ is smooth}\label{section:rd_smooth}
\subsubsection{Deformation theory}\label{section:rd_defo_theory}
Because $\mathcal{R}^d$ and $\mathcal{M}^{\orb;0,*}$ are locally of finite type over $\mathbb{S}_d$, it suffices to show that $\mathcal{R}^d \to \mathcal{M}^{\orb;0,*}$ satisfies the infinitesimal lifting criterion with respect to an extension of Artinian local rings $A^{\prime} \to A$ with algebraically closed residue field $k$. We may assume that this is a small extension, i.e. $J = \ker(A^{\prime} \to A)$ is killed by the maximal ideal $\mathfrak{m}^{\prime}$ of $A^{\prime}$. Given a pointed orbinodal curve $(\mathcal{C}_{A^{\prime}} \to C_{A^{\prime}} \to \Spec A^{\prime};\sigma_{A^{\prime}})$ and a morphism $\chi_A: \mathcal{C}_A \coloneqq (\mathcal{C}_{A^{\prime}})_A \to \mathcal{X}_d$, we want to extend $\chi_A$ to a morphism $\chi_{A^{\prime}}: \mathcal{C}_{A^{\prime}} \to \mathcal{X}_d$.

We will simplify the task of studying deformations over a nodal curve $C$ by picking an adapted affine open cover $\{U_i\}$ of $C$ in the sense of \cite[\S 3.3]{f_moduli_hyperelliptic}. This means that each $U_i$ contains at most one point that is a node, marked point, or branch point and each such point is contained in at most one $U_i$. We will also assume that each $U_i$ is connected. We will denote the preimage of $U_i$ in $\mathcal{C}$ by $\mathcal{U}_i$ and the preimage of $U_i$ in some thickening $C_A \to C$ as $U_{i,A}$. We will denote the intersections of the $U_i$ by $U_{ij}$, $U_{ijk}$, etc. We require the indices $i,j,k$ to be pairwise distinct when we do this. The key point is that each $U_{ij,A}$ is a scheme mapping to $\mathcal{E}_d \cong B\mathbf{S}_d$ under $\chi_A$. Because deformations of schemes lift uniquely to deformations of their finite \'{e}tale covers, we can deform each $\chi_A|_{U_{i,A}}$ independently. We will make this precise below.

\begin{proposition}\label{proposition:rd_smooth}
The morphism $\mathcal{R}^d \to \mathcal{M}^{\orb;0,*}$ is smooth.
\end{proposition}

\begin{proof}
Given a pointed orbinodal curve $(\mathcal{C}_{A^{\prime}} \to C_{A^{\prime}} \to \Spec A^{\prime};\sigma_{A^{\prime}})$ and a morphism $\chi_A: \mathcal{C}_A \coloneqq (\mathcal{C}_{A^{\prime}})_A \to \mathcal{X}_d$, we want to extend $\chi_A$ to a morphism $\chi_{A^{\prime}}: \mathcal{C}_{A^{\prime}} \to \mathcal{X}_d$. Pick an adapted affine cover $\{U_i\}$ of $C$. We will first find an extension $\chi_{i,A^{\prime}}: \mathcal{U}_{i,A^{\prime}} \to \mathcal{X}_d$ for each $i$. For a given $i$, we are in one of the following cases: \begin{itemize}
\item $U_i$ contains a branch point: In this case, $\mathcal{U}_i$ contains no stacky points, so $\mathcal{U}_i = U_i$ is an affine curve. From $\chi_A$, we have a $G_d$-bundle $P_A$ on $U_{i,A}$ and a section $s_A \in H^0(U_{i,A},P_A \times_{G_d} V_d)$ that we want to extend to $U_{i,A^{\prime}}$. Recall that a $G_d$-bundle has a description in terms of vector bundles and sections of associated bundles. We can extend a vector bundle $\mathcal{F}_A$ on $U_{i,A}$ to a vector bundle $\mathcal{F}_A^{\prime}$ on $U_{i,A^{\prime}}$ because obstructions to extending $\mathcal{F}_A$ lie in $H^2(U_i,\mathcal{E}nd(\mathcal{F})) \otimes_k J \cong 0$. We can then extend any section $s_A \in H^0(U_{i,A},\mathcal{F}_A) \cong H^0(U_{i,A^{\prime}},\mathcal{F}_A)$ to a section $s_A^{\prime} \in H^0(U_{i,A^{\prime}},\mathcal{F}_{A^{\prime}})$ because $U_{i,A^{\prime}}$ is affine and $\mathcal{F}_{A^{\prime}} \to \mathcal{F}_A$ is surjective.

\item $U_i$ does not contain a branch point: Because $\chi_A$ maps $U_{i,A}$ to $\mathcal{E}_d$, we just want to extend an $\mathbf{S}_d$-torsor $V_{i,A} \to \mathcal{U}_{i,A}$ to an $\mathbf{S}_d$-torsor $V_{i,A^{\prime}} \to \mathcal{U}_{i,A^{\prime}}$. Because $\mathcal{U}_{i,A}$ is cut out of $\mathcal{U}_{i,A^{\prime}}$ by a nilpotent sheaf of ideals, there exists a unique such extension. This fact for schemes is \cite[IX, Proposition 1.7]{gr_sga1}. For the statement for orbinodal curves, see the proof of \cite[Theorem 3.0.2]{acv_twisted_bundles}.
\end{itemize}
We now have extensions $\chi_{i,A^{\prime}}: \mathcal{U}_{i,A^{\prime}} \to \mathcal{X}_d$ of $\chi_A$ that we want to glue together. We have isomorphisms $\psi_{ij,A}: \chi_{i,A^{\prime}}|_{U_{ij,A}} \xrightarrow[]{\sim} \chi_A|_{U_{ij,A}} \xleftarrow[]{\sim} \chi_{j,A^{\prime}}|_{U_{ij,A}}$. Because each branch point is contained in at most one $U_i$, $\chi_A$ maps every $U_{ij,A}$ to $\mathcal{E}_d$. By the above discussion of extending \'{e}tale covers along nilpotent ideal sheaves, there exists a unique isomorphism $\psi_{ij,A^{\prime}}: \chi_{i,A^{\prime}}|_{U_{ij,A^{\prime}}} \xrightarrow[]{\sim} \chi_{j,A^{\prime}}|_{U_{ij,A^{\prime}}}$ extending $\psi_{ij,A}$. Because the $\psi_{ij,A}$ satisfy the cocycle condition, the $\psi_{ij,A^{\prime}}$ satisfy the cocycle condition by their uniqueness. Thus, we can glue the $\chi_{i,A^{\prime}}$ to an extension $\chi_{A^{\prime}}: \mathcal{C}_{A^{\prime}} \to \mathcal{X}_d$ of $\chi_A$.
\end{proof}

We have proven the first part of \Cref{theorem:rd_nice}(c). We also deduce \Cref{theorem:rd_nice}(a) from \Cref{theorem:olsson_twisted}(a) and \Cref{proposition:rd_smooth}.

\subsubsection{Dimension}\label{section:rd_dimension}
To prove \Cref{theorem:rd_nice}, it remains to prove that the morphism $\mathcal{R}^{d;b,*} \to \mathcal{M}^{\orb;0,*}$ is smooth of relative dimension $b$. Fix a geometric point of $\mathcal{R}^d$ corresponding to $(\mathcal{C} \to C \to \Spec k;\sigma;\chi)$. We want to study the deformations of $\chi$ over the trivial deformation $(\mathcal{C}_{k[\epsilon]} \to C_{k[\epsilon]} \to \Spec k[\epsilon];\sigma_{k[\epsilon]})$. Pick an adapted affine cover $\{U_i\}$ of $C$ as in the proof of \Cref{proposition:rd_smooth}. For all the $U_i$ not containing a branch point, there exists a unique extension $\chi_{i,k[\epsilon]}: \mathcal{U}_{i,k[\epsilon]} \to \mathcal{X}_d$ of $\chi|_{\mathcal{U}_i}$ because $\chi$ maps $\mathcal{U}_i$ to $\mathcal{E}_d$. Once we pick extensions $\chi_{i,k[\epsilon]}$ for the $U_i$ containing branch points, because the intersections $U_{ij}$ contain no branch points, there are unique isomorphisms $\chi_{i,k[\epsilon]}|_{U_{ij,k[\epsilon]}} \xrightarrow[]{\sim} \chi_{j,k[\epsilon]}|_{U_{ij,k[\epsilon]}}$ such that the $\chi_{i,k[\epsilon]}$ glue to an extension $\chi_{k[\epsilon]}$ of $\chi$. Hence, the dimension of the deformation space of $\chi$ is the sum of the dimensions of the deformation spaces of the $\chi_i$ for the $U_i$ containing branch points. We have reduced the second part of \Cref{theorem:rd_nice}(c) to the following proposition.

\begin{proposition}\label{proposition:dim_def_chi_i}
Suppose $U_i$ contains a branch point $p_i$. Then the dimension of the deformation space of $\chi_i$ is the degree of the divisor $\chi_i^{-1}(\Sigma_d)$ supported at $p_i$.
\end{proposition}

\begin{proof}
By the above discussion, the deformation space of $\chi_i$ does not change if we shrink $U_i = \Spec R$ around the branch point $p_i$. We can therefore assume that there exists a function $f$ on $U_i$ cutting out $p_i$ and that the $G_d$-bundle classified by $\chi_i$ is trivial on $U_i$ (a $G_d$-bundle is Zariski locally trivial because of the description of $G_d$-bundles in terms of vector bundles and sections). Since we can lift sections and frames to $k[\epsilon]$, any deformation of this $G_d$-bundle to $U_{i,k[\epsilon]}$ is trivial.

Let $v \in H^0(U_i,U_i \times V_d) \cong V_d(R)$ be the section classified by $\chi_i$. The first order deformations of $\chi_i$ are identified with the set of $G_d(R[\epsilon])$-equivalence classes of sections of $U_{i,k[\epsilon]} \times V_d$ whose restriction to $U_i \times V_d$ is equivalent to $v \in V_d(R)$ under the $G_d(R)$-action. Because $G_d(R[\epsilon]) \to G_d(R)$ is surjective, this is the same as the set of $\ker(G_d(R[\epsilon]) \to G_d(R))$-equivalence classes of sections $v_{k[\epsilon]} \in H^0(U_{i,k[\epsilon]},U_{i,k[\epsilon]} \times V_d) \cong V_d(R[\epsilon])$ such that $v_{k[\epsilon]}|_{U_i} = v \in V_d(R)$. We can write $v_{k[\epsilon]} = v + \epsilon u$ for some $u \in V_d(R)$. Under the identification $\mathfrak{g}_d(R) \cong \ker(G_d(R[\epsilon]) \to G_d(R))$, $x \in \mathfrak{g}_d(R)$ acts on $v + \epsilon u$ by $x \cdot (v + \epsilon u) = v + \epsilon(u + xv)$. Thus, the first order deformations of $\chi_i$ are identified with $V_d(R)/\mathfrak{g}_d(R)v$.

The natural map $V_d(R)/\mathfrak{g}_d(R)v \to V_d(\widehat{\mathcal{O}}_{U_i,p_i})/\mathfrak{g}_d(\widehat{\mathcal{O}}_{U_i,p_i})v$ is an isomorphism because as a coherent sheaf on $U_i$, $V_d(R)/\mathfrak{g}_d(R)v$ is supported at $p_i$ since the deformation space of $\chi_i$ does not change when we shrink $U_i$ around $p_i$. This reduces the proposition to the following lemma.

\begin{lemma}\label{lemma:lie_alg_fact}
Suppose $v \in V_d(k[[t]])$ such that $\Delta_d(v) \ne 0$. Then $$\dim_k(V_d(k[[t]])/\mathfrak{g}_d(k[[t]])v) = v_t(\Delta_d(v)).$$
\end{lemma}

\begin{proof}
Everything in \Cref{section:prehomog_local} makes sense for arbitrary $k[[t]]$ except the discussion of Haar measures. As in \Cref{section:prehomog_local}, pick left-invariant top forms $\omega_{G_d}$ and $\omega_{V_d}$ on $(G_d)_{k((t))}$ and $(V_d)_{k((t))}$ normalized so that their values at the identities are wedges of bases of $\mathfrak{g}_d^*(k[[t]])$ and $V_d^*(k[[t]])$, respectively. We know that $\pi_v^*\omega_{V_d} = \mathcal{J}_d\Delta_d(v)\chi_d^2\omega_{G_d}$, where $v_t(\mathcal{J}_d) = 0$. Evaluating both sides at $1 \in G_d(k[[t]])$ shows that the image of $d\pi_v: \det\mathfrak{g}_d(k[[t]]) \to \det V_d(k[[t]])$ is $\Delta_d(v)\det V_d(k[[t]])$. This lets us conclude that $\dim_k(V_d(k[[t]])/\mathfrak{g}_d(k[[t]])v) = v_t(\Delta_d(v))$.
\end{proof}

\end{proof}

We have completed the proof of \Cref{theorem:rd_nice}.

\begin{remark}\label{remark:quotient_impossible}
At first glance, the proofs of properness and smoothness above do not seem specific to our situation. As long as we have a quotient stack $[X/G]$, where \begin{itemize}
\item $X$ is smooth and affine,

\item $G$ is reductive,

\item $G$ acts on $X$ with semistable locus $X^{ss}$ such that $X^{ss}/G \cong BH$ for an \'{e}tale group scheme $H$,
\end{itemize}
the space of quasimaps from a smooth non-stacky curve to $BH \subset [X/G]$ should be smooth and proper. Based on this, it seems like we can try to get finite groups $H$ as generic stabilizers besides the ones we can get from prehomogeneous vector spaces so that we can compactify Hurwitz spaces of $H$-covers.

However, a short argument the author learned from Andres Fernandez Herrero that uses the Luna \'{e}tale slice theorem \cite{ahr_a_luna} shows that any such $[X/G]$ is actually isomorphic to a quotient of a prehomogeneous vector space by its group action. Therefore, we cannot get any new Hurwitz space compactifications besides the ones we can get from prehomogeneous vector spaces with reductive groups (which are the ones we consider in this paper). We recall the argument here.

We'll work over an algebraically closed field $k$ of characteristic 0. Suppose the pair $(G,X)$ has the properties above. Then the affine quotient $X//G$ is $\Spec k$, and by standard facts from GIT, there is a unique closed orbit $Gx \subset X$ corresponding to a unique closed $k$-point of $[X/G]$. Let $G_x \subset G$ be the stabilizer of $p$. Because $G/G_x \cong Gx$ is affine, Matsushima's criterion implies that $G_x$ is reductive. By \cite[Theorem 4.12, Remark 4.15]{ahr_a_luna}, because $[X/G]$ is smooth, we have an isomorphism $[X/G] \cong [N_x/G_x]$, where $N_x \coloneqq T_xX/T_x(Gx)$ is the tangent space to $[X/G]$ at $x$. Because there is an open dense $k$-point of $[X/G]$, there is an open dense $k$-point of $[N_x/G_x]$, and $(G_x,N_x)$ is a prehomogeneous vector space.
\end{remark}

\section{The fibers of $\br$}\label{section:fibers_of_br}
Fix a (not necessarily algebraically closed) base field $k$ over $\mathbb{S}_d$, and consider a point $p \in \mathcal{M}^{b,n}(k)$ classifying an object $(C;\Sigma;\sigma)$. For convenience, we will assume that $C$ is smooth (see \Cref{section:stacky_smooth}), but a similar analysis should be possible for nodal $C$. The main result of this section is \Cref{proposition:fiber_bound_34}, which bounds the dimension of the fiber $p \times_{\mathcal{M}} \mathcal{R}^d$ for $d \in \{3,4\}$ using the Igusa zeta functions from \Cref{section:igusa_zeta_functions} and the Lang-Weil bounds. We will work over a base field $k$ throughout this section, unless specified otherwise.

\begin{lemma}\label{lemma:etale_orbi}
Specifying an orbicurve $\mathcal{C} \to C$ and an \'{e}tale degree $d$ cover $\overline{\mathcal{D}} \to \mathcal{C} - \Sigma$ such that the corresponding morphism $\mathcal{C} - \Sigma \to B\mathbf{S}_d$ is representable is equivalent to specifying an \'{e}tale degree $d$ cover $D \to C^{\gen} - \Sigma$.
\end{lemma}

\begin{proof}
It is clear how to go from the former to the latter. To go the other way, consider the associated $\mathbf{S}_d$-torsor $E^{\prime} \to C^{\gen} - \Sigma$, and complete it to a finite map $\overline{E}^{\prime} \to C - \Sigma$ with $\overline{E}$ a smooth curve. Then take $\mathcal{C}^{\prime} - \Sigma = [\overline{E}^{\prime}/\mathbf{S}_d]$ and $\overline{\mathcal{D}}^{\prime} = [\overline{E}^{\prime}/\mathbf{S}_{d - 1}]$; we can glue $\mathcal{C}^{\prime} - \Sigma$ with $C^{\gen}$ to get an orbicurve $\mathcal{C}^{\prime}$. We want to check that $\mathcal{C} \cong \mathcal{C}^{\prime}$ and $\mathcal{D} \cong \mathcal{D}^{\prime}$. Note that the representable morphism $\mathcal{C} - \Sigma \to B\mathbf{S}_d$ classifies an \'{e}tale $\mathbf{S}_d$-torsor $\overline{E} \to \mathcal{C} - \Sigma$, where $\overline{E}$ is an algebraic space. Because $\overline{E} \to C - \Sigma$ is finite, $\overline{E}$ must actually be a scheme. In fact, since $\overline{E} \to \mathcal{C} - \Sigma$ is \'{e}tale, $\overline{E}$ is a smooth curve. There is only one smooth curve finite over $C - \Sigma$ that restricts to $E \to C^{\gen} - \Sigma$, namely $\overline{E}^{\prime} \to C - \Sigma$, so we have $\overline{E} \cong \overline{E}^{\prime}$, and we conclude that $\mathcal{C} \cong \mathcal{C}^{\prime}$ and $\mathcal{D} \cong \mathcal{D}^{\prime}$, proving the lemma.
\end{proof}

Now fix an \'{e}tale degree $d$ cover $D \to C^{\gen} - \Sigma$. Suppose $q \in (p \times_{\mathcal{M}} \mathcal{R}^d)(k)$ classifies a morphism $\chi: \mathcal{C} \to \mathcal{X}_d$ such that the restriction $C^{\gen} - \Sigma \to \mathcal{E}_d \simeq B\mathbf{S}_d$ classifies $D \to C^{\gen} - \Sigma$. To proceed, we need to fix some notation. Let $c_1,\ldots,c_r \in C$ be the points in the support of $\Sigma$, and let $b_1,\ldots,b_r$ and $d_1,\ldots,d_r$ be their multiplicities in $\Sigma$ and degrees over $k$, respectively, so that $b = \sum_ib_id_i$. Let $G_D$ be the automorphism group of $D$ over $C^{\gen} - \Sigma$. For each $1 \le i \le r$, let $\widehat{C}_i$ denote the formal neighborhood $\Spec\widehat{\mathcal{O}}_{C,c_i}$ of $c_i$, and let $\widehat{C}_i^{\circ} = \Spec K(\widehat{\mathcal{O}}_{C,c_i})$. Let $\widehat{D}_i = D \times_C \widehat{C}_i^{\circ}$, and let $G_{\widehat{D}_i}$ denote the automorphism group of $\widehat{D}_i$ over $\widehat{C}_i^{\circ}$. For each $i$, let $V_d(\widehat{C}_i)_{\widehat{D}_i,b_i} \subset V_d(\widehat{C}_i)_{b_i}$ denote the subset of $v$ such that the resulting object of $V_d^{ss}(\widehat{C}_i^{\circ}) // G_d(\widehat{C}_i^{\circ}) \simeq \mathcal{E}_d(\widehat{C}_i^{\circ})$ is isomorphic to $\widehat{D}_i$.

\begin{proposition}\label{proposition:local_description_fibers}
There is an equivalence of groupoids \begin{align*}
&(p \times_{\mathcal{M}} \mathcal{R}^d)(k) \\ &\simeq \bigsqcup_D\left(BG_D \times_{BG_{\widehat{D}_1}} V_d(\widehat{C}_1)_{\widehat{D}_1,b_1} // G_d(\widehat{C}_1) \times_{BG_{\widehat{D}_2}} \cdots \times_{BG_{\widehat{D}_r}} V_d(\widehat{C}_r)_{\widehat{D}_r,b_r} // G_d(\widehat{C}_r)\right),
\end{align*}
where $D$ ranges over the isomorphism classes of \'{e}tale degree $d$ covers of $C^{\gen} - \Sigma$ and the 2-fiber products come from the morphisms of groupoids $BG_D \to BG_{\widehat{D}_i}$ and $V_d(\widehat{C}_i)_{\widehat{D}_i,b_i} // G_d(\widehat{C}_i) \to BG_{\widehat{D}_i} \subset V_d(\widehat{C}_i^{\circ}) // G_d(\widehat{C}_i^{\circ})$.
\end{proposition}

\begin{proof}
A point of $(p \times_{\mathcal{M}} \mathcal{R}^d)(k)$ corresponds to a $G_d$-bundle on an orbicurve $\mathcal{C} \to C$ with stackiness only above $\sigma$ and a global section of some associated bundle such that the branch divisor is $\Sigma$. As noted above, a finite \'{e}tale cover of $C^{\gen} - \Sigma$ extends to a unique finite \'{e}tale cover of $\mathcal{C} - \Sigma$ for a unique orbicurve $\mathcal{C} \to C$ such that the corresponding morphism $\mathcal{C} - \Sigma \to [V_d^{ss}/G_d] \simeq B\mathbf{S}_d$ is representable. Thus, the data of an orbicurve $\mathcal{C} \to C$ and a bundle and section on $\mathcal{C}$ are equivalent to the data of a bundle and section on $C^{\gen}$.

Recall that the data of a $G_d$-bundle is equivalent to the data of a vector bundle ($d = 3$) or the data of two vector bundles and a section of an associated bundle ($d \in \{4,5\}$). Because of this vector bundle description of $G_d$-bundles, we may apply the Beauville-Laszlo theorem \cite{bl_un_lemme} to deduce the following: the data of a bundle and section on $C^{\gen}$ are equivalent to the data of bundles and sections on $C^{\gen} - \Sigma,\widehat{C}_1,\ldots,\widehat{C}_r$ equipped with isomorphisms on $\widehat{C}_1^{\circ},\ldots,\widehat{C}_r^{\circ}$. Because $C^{\gen} - \Sigma$ is required to map to $\mathcal{E}_d$, the data of a bundle and section on $C^{\gen}$ are equivalent to the datum of an \'{e}tale degree $d$ cover $D$ of $C^{\gen} - \Sigma$. Similarly, the data of a bundle and section on $\widehat{C}_i^{\circ}$ are equivalent to the datum of an \'{e}tale degree $d$ cover $\widehat{D}_i$ of $\widehat{C}^{\gen} - \Sigma$. Because $G_d$-bundles are Zariski locally trivial, they are automatically trivial on any $\widehat{C}_i$. Thus, the groupoid of bundles and section on $\widehat{C}_i$ with the appropriate branch divisor is equivalent to the groupoid quotient $V_d(\widehat{C}_i)_{\widehat{D}_i,b_i} // G_d(\widehat{C}_i)$. Combining these observations, we get the desired equivalence of groupoids.
\end{proof}

\begin{corollary}\label{corollary:fiber_finite_field}
If $k$ is finite, then as $k^{\prime}$ ranges over finite extensions of $k$,  we have a bound on the cardinality of the fiber $$\#(p \times_{\mathcal{M}} \mathcal{R}^d)(k^{\prime}) \le \alpha\prod_i\#\left(V_d((\widehat{C}_i)_{k^{\prime}})_{b_i} // G_d((\widehat{C}_i)_{k^{\prime}})\right),$$ where $\alpha$ is a positive constant depending only on $p$.
\end{corollary}

\begin{proof}
Because $\mathrm{char}(k)$ does not divide $d!$, the number of \'{e}tale degree $d$ covers $D_{k^{\prime}} \to (C^{\gen} - \Sigma)_{k^{\prime}}$ has an upper bound depending only on $C^{\gen} - \Sigma$ and thus only on $p$. This can be seen from the fact that the corresponding Hurwitz spaces are finite \'{e}tale over the corresponding configuration spaces of points in $C$ (e.g. \cite[Theorem 4.11]{rw_hur} or \Cref{proposition:hurwitz_etale} below). Let this upper bound be $\beta$. All the $G_D,G_{\widehat{D}_i}$ have order bounded above by $d!$, and each $V_d(\widehat{C}_i)_{\widehat{D}_i} // G_d(\widehat{C}_i)$ is a subgroupoid of $V_d(\widehat{C}_i)_{b_i} // G_d(\widehat{C}_i)$. We thus have an upper bound \begin{align*}
&\#(p \times_{\mathcal{M}} \mathcal{R}^d)(k) \\
&\le \#\left(\bigsqcup_D\left(BG_D \times_{BG_{\widehat{D}_1}} V_d(\widehat{C}_1)_{\widehat{D}_1,b_1} // G_d(\widehat{C}_1) \times_{BG_{\widehat{D}_2}} \cdots \times_{BG_{\widehat{D}_r}} V_d(\widehat{C}_r)_{\widehat{D}_r,b_r} // G_d(\widehat{C}_r)\right)\right) \\
&\le \beta\max_D\#\left(BG_D \times_{BG_{\widehat{D}_1}} V_d(\widehat{C}_1)_{\widehat{D}_1,b_1} // G_d(\widehat{C}_1) \times_{BG_{\widehat{D}_2}} \cdots \times_{BG_{\widehat{D}_r}} V_d(\widehat{C}_r)_{\widehat{D}_r,b_r} // G_d(\widehat{C}_r)\right) \\
&\le \beta(d!)^b\prod_i\#\left(V_d(\widehat{C}_i)_{b_i} // G_d(\widehat{C}_i)\right).
\end{align*}
Note that if we replace $\widehat{C}_i$ with $(\widehat{C}_i)_{k^{\prime}} = \bigsqcup_j\widehat{C}_{ij}^{\prime}$, then $$\prod_jV_d(\widehat{C}_{ij})_{b_i} // G_d(\widehat{C}_{ij}) \simeq V_d((\widehat{C}_i)_{k^{\prime}})_{b_i} // G_d((\widehat{C}_i)_{k^{\prime}}),$$ and so \begin{align*}
\#(p \times_{\mathcal{M}} \mathcal{R}^d)(k^{\prime}) &\le \beta(d!)^b\prod_{i,j}\#\left(V_d(\widehat{C}_{ij})_{b_i} // G_d(\widehat{C}_{ij})\right) \\
&= \beta(d!)^b\prod_i\#\left(V_d((\widehat{C}_i)_{k^{\prime}})_{b_i} // G_d((\widehat{C}_i)_{k^{\prime}})\right)
\end{align*}
Taking $\alpha = \beta(d!)^b$, we are done.
\end{proof}

We would like to apply the Lang-Weil bounds \cite{lw_number_of_points} to bound the dimension of the fiber $p \times_{\mathcal{M}} \mathcal{R}^d$. However, we could not find the statement for Deligne-Mumford stacks like $p \times_{\mathcal{M}} \mathcal{R}^d$ in the literature. Note that all the automorphism groups appearing in $p \times_{\mathcal{M}} \mathcal{R}^d$ have order dividing $d!$, as they embed into $\mathbf{S}_d$; this means that $p \times_{\mathcal{M}} \mathcal{R}^d$ is tame. Hence, for our purposes, it suffices to prove the Lang-Weil bounds for tame Deligne-Mumford stacks.

\begin{lemma}\label{lemma:lang_weil_dm}
Let $\mathcal{X}$ be a tame separated Deligne-Mumford stack of finite type over $\mathbb{F}_q$. In all the asymptotics below, we take $r \to \infty$. \begin{enumerate}[(a)]
\item If $\#\mathcal{X}(\mathbb{F}_{q^r}) = O(q^{rn})$ for some $n \in \mathbb{N}$, then $\dim\mathcal{X} \le n$.

\item If there exist some $s,C \in \mathbb{N}_{> 0}$ and $n \in \mathbb{N}$ such that $\#\mathcal{X}(\mathbb{F}_{q^{rs}}) = Cq^{rsn} + O\left(q^{rs\left(n - \frac{1}{2}\right)}\right)$, then $\dim\mathcal{X} = n$.
\end{enumerate}
\end{lemma}

\begin{proof}
Consider the morphism to the coarse moduli space $\rho: \mathcal{X} \to X$. By \cite[Proposition A.0.1]{acv_twisted_bundles}, the Frobenius-equivariant maps $H_c^i(\mathcal{X}_{\overline{\mathbb{F}}_q},\mathbb{Q}_{\ell}) \to H_c^i(X_{\overline{\mathbb{F}}_q},\mathbb{Q}_{\ell})$ are isomorphisms because $\mathcal{X}$ is tame; these maps exist because $\rho$ is proper. By Sun's fixed point formula for stacks \cite[Theorem 1.1]{s_l_series}, the point counts of $\mathcal{X}$ and $X$ agree. Because $\dim\mathcal{X} = \dim X$, we can deduce both (a) and (b) using the Lang-Weil bounds for the algebraic space $X$ \cite[Lemma 5.1]{l_geometric_average}, which can be proven using the fact that a quasi-separated algebraic space has a dense open subset represented by a scheme.
\end{proof}

\begin{proposition}\label{proposition:fiber_bound_34}
Assume $k$ is algebraically closed. Suppose $p \in \mathcal{M}^{b,n}(k)$ classifies an object $(C;\Sigma;\sigma)$ such that $C$ is smooth and $\Sigma = \sum_ib_i[q_i]$ for pairwise distinct points $q_i \in C(k)$. Then we have the following bounds on fiber dimensions:
\begin{enumerate}[(a)]
\item $\dim p \times_{\mathcal{M}} \mathcal{R}^3 \le \sum_i\lfloor\frac{b_i}{6}\rfloor$.

\item $\dim p \times_{\mathcal{M}} \mathcal{R}^4 \le \sum_i\lfloor\frac{b_i}{4}\rfloor$.
\end{enumerate}
\end{proposition}

\begin{proof}
\begin{enumerate}[(a)]
\item We first assume that $k = \overline{\mathbb{F}}_q$ and that all the points in the support of $p$ are defined over $\mathbb{F}_q$. We observe that in \Cref{theorem:igusa_computation_3_and_4}, the coefficient $$\#(V_3(\mathbb{F}_{q^r}[[t]])_{b_i} // G_3(\mathbb{F}_{q^r}[[t]]))$$ of $t^{b_i}$ is $O\left(q^{r\lfloor\frac{b_i}{6}\rfloor}\right)$ as $r \to \infty$. Thus, \Cref{corollary:fiber_finite_field} and \Cref{lemma:lang_weil_dm}(a) imply that $\dim p \times_{\mathcal{M}} \mathcal{R}^3 \le \sum_i\lfloor\frac{b_i}{6}\rfloor$ in the case that $p$ is defined over a finite field.

For the general case, we can use the following spreading out argument. Because $\mathcal{M}$ is locally of finite type over $\mathbb{S}_3$, we can find some finitely generated $\mathbb{Z}$-algebra $A \subset k$ and a map $\Spec A \to \mathcal{M}^{b,n}$ such that $p$ is the composite $\Spec k \to \Spec A \to \mathcal{M}^{b,n}$. All the fibers of $\Spec A \times_{\mathcal{M}} \mathcal{R}^3 \to \Spec A$ over closed points have dimension $\le \sum_i\lfloor\frac{b_i}{6}\rfloor$ by the case $k = \overline{\mathbb{F}}_q$. Because $\br$ is proper, we conclude that the generic fiber $p \times_{\mathcal{M}} \mathcal{R}^3$ has dimension $\le \sum_i\lfloor\frac{b_i}{6}\rfloor$.

\item The exact same argument as in part (a) works, except that we use the bound $$\#(V_4(\mathbb{F}_{q^r}[[t]])_{b_i} // G_4(\mathbb{F}_{q^r}[[t]])) = O\left(q^{r\lfloor\frac{b_i}{4}\rfloor}\right)$$ from \Cref{theorem:igusa_computation_3_and_4} instead.
\end{enumerate}
\end{proof}

\section{Some stacks related to $\mathcal{R}^d$}\label{section:related_stacks}
In this section, we show how our construction of $\mathcal{R}^d$ leads to constructions of other related stacks, such as Hurwitz stacks and marked variants of $\mathcal{R}^d$.

\subsection{Hurwitz stacks}\label{section:hurwitz_stacks}
In this section, we construct Hurwitz stacks as open substacks of our $\mathcal{R}^d$.

Recall that for a smooth curve $C$ over a field, a finite separable cover $D \to C$ is \emph{simply branched} if the branch divisor $\Sigma_D \subset C$ is geometrically reduced. Then $D$ must be smooth by the following argument. Consider the normalization $D^{\nu} \to D \to C$, and let $c_1,\ldots,c_r \in C$ be the branch points of $D \to C$. Then by \cite[Proposition 7.3]{d_compact}, we have the following equation involving the branch divisors of $D$ and $D^{\nu}$: $$\Sigma_D = \Sigma_{D^{\nu}} + 2\sum_{i = 1}^r\length_{\mathcal{O}_{C,c_i}}(\mathcal{O}_{D^{\nu},c_i}/\mathcal{O}_{D,c_i}),$$ where $\mathcal{O}_{D,c_i} = \mathcal{O}_D \otimes_{\mathcal{O}_C} \mathcal{O}_{C,c_i}$ and $\mathcal{O}_{D^{\nu},c_i} = \mathcal{O}_{D^{\nu}} \otimes_{\mathcal{O}_C} \mathcal{O}_{C,c_i}$. Because $\Sigma_D$ is geometrically reduced, all the coefficients in $\Sigma_D$ are at most 1, so all the terms $\length_{\mathcal{O}_{C,c_i}}(\mathcal{O}_{D^{\nu},c_i}/\mathcal{O}_{D,c_i})$ are 0. Thus, $D^{\nu} = D$, and $D$ is smooth.

Let $\mathcal{M}^{\circ} \subset \mathcal{M}$ be the open substack parametrizing $(C \to S;\Sigma;\sigma)$ such that for every geometric point $\overline{s} \to S$, the geometric fiber $\Sigma_{\overline{s}}$ is reduced. Define the \emph{big Hurwitz stack of simply branched degree $d$ covers} $\mathcal{H}^d \coloneqq \br^{-1}(\mathcal{M}^{\circ}) \subset \mathcal{R}^d$. By \Cref{proposition:disc_agrees}, requiring the branch divisor of $\mathcal{C} \to \mathcal{X}_d$ to have geometrically reduced fibers is equivalent to requiring the cover of $C^{\gen}$ classified by $\mathcal{C} \to \mathcal{X}_d \to \Covers_d$ to be simply branched.

\begin{proposition}\label{proposition:hurwitz_moduli}
To give an object $(\mathcal{C} \to C \to S;\sigma;\chi) \in \mathcal{H}^d(S)$, it is equivalent to give data $(\mathcal{C} \to C \to S;\sigma;\varphi: \mathcal{D} \to \mathcal{C})$, where \begin{enumerate}[(1)]
\item $(\mathcal{C} \to C \to S;\sigma)$ is a pointed orbinodal curve.

\item $\varphi: \mathcal{D} \to \mathcal{C}$ is a degree $d$ cover such that for each geometric point $\overline{s} \to S$, $\varphi_{\overline{s}}: \mathcal{D}_{\overline{s}} \to \mathcal{C}_{\overline{s}}$ is \begin{itemize}
\item \'{E}tale over the generic points, orbinodes, and marked points of $\mathcal{C}_{\overline{s}}$.

\item Simply branched over $C_{\overline{s}}^{\gen}$.
\end{itemize}
\end{enumerate}
\end{proposition}

\begin{proof}
Because $\phi_d: \mathcal{X}_d \to \Covers_d$ restricts to an isomorphism over the Gorenstein locus $\Covers_d^{\Gor} \subset \Covers_d$ (\Cref{theorem:lvw_gorenstein}), it suffices to show that for any morphism $\mathcal{C} \to \mathcal{X}_d$ classified by a geometric point $\overline{s} \in \mathcal{H}^d(k)$, the composite $\mathcal{C} \to \mathcal{X}_d \to \Covers_d$ lands in $\Covers_d^{\Gor}$. The orbinodes and marked points of $\mathcal{C}$ are sent to the \'{e}tale locus $B\mathbf{S}_d \subset \Covers_d^{\Gor}$, so we just have to check that $C^{\gen}$ is sent to $\Covers_d^{\Gor}$. However, for any simply branched cover $D \to C^{\gen}$, $D$ must be smooth. Thus, $D \to C^{\gen}$ is Gorenstein, and the composite $\mathcal{C} \to \mathcal{X}_d \to \Covers_d$ lands in $\Covers_d^{\Gor}$, as desired.
\end{proof}

\begin{proposition}\label{proposition:hurwitz_etale}
The morphism $\br: \mathcal{H}^{d,\sm} \to \mathcal{M}^{\sm}$ is \'{e}tale.
\end{proposition}

\begin{proof}
Recall that $\mathcal{H}^{d,\sm}$ and $\mathcal{M}^{\sm}$ are locally of finite type over $\mathbb{S}_d$ by \Cref{theorem:rd_nice}(a) and \Cref{proposition:deopurkar_m_smooth}, respectively. Thus, we can use the infinitesimal lifting criterion. We have the following setup: \begin{itemize}
\item $A^{\prime} \to A$ is a small extension of Artinian local rings with algebraically closed residue field $k$.

\item $(C_{A^{\prime}} \to \Spec A^{\prime};\Sigma_{A^{\prime}};\sigma_{A^{\prime}})$ is an object of $\mathcal{M}^{\sm}(A^{\prime})$.

\item $(\mathcal{C}_A \to C_A \to \Spec A;\sigma_A;\varphi: \mathcal{D}_A \to \mathcal{C}_A)$ is an object of $\mathcal{H}^{d,\sm}(A)$, where $C_A \coloneqq (C_{A^{\prime}})_A$ (here, we use the description of $\mathcal{H}^d$ from \Cref{proposition:hurwitz_moduli}).
\end{itemize}
To check the infinitesimal lifting criterion, we need to show that there is a unique object $(\mathcal{C}_{A^{\prime}} \to C_{A^{\prime}} \to \Spec A^{\prime};\sigma_{A^{\prime}};\varphi_{A^{\prime}}: \mathcal{D}_{A^{\prime}} \to \mathcal{C}_{A^{\prime}}) \in \mathcal{H}^{d,\sm}(A^{\prime})$ up to unique isomorphism that lifts $(\mathcal{C}_A \to C_A \to \Spec A;\sigma_A;\varphi: \mathcal{D}_A \to \mathcal{C}_A)$.

We first observe that by \Cref{theorem:cadman_root_equivalence}, there is a unique choice of $\mathcal{C}_{A^{\prime}} \to C_{A^{\prime}}$ up to unique isomorphism, namely the one with the same automorphism groups at the marked points as $\mathcal{C}_A$. We will let $\mathcal{C}_{A^{\prime}}$ be this unique choice. Hence, the only thing left to lift is the cover $\varphi_A: \mathcal{D}_A \to \mathcal{C}_A$.

Now as in the proof of \Cref{proposition:rd_smooth}, we can approach the problem locally. Let $\{U_i\}$ be an adapted affine open cover of $C \coloneqq (C_A)_k$. Recall that this means each marked point or branch point (of $D \to C$) on $C$ is contained in at most one $U_i$ and each such point is contained in at most one $U_i$. As in \Cref{section:rd_smooth}, let $\mathcal{U}_i$ denote the preimage of $U_i$ in $\mathcal{C} \coloneqq (\mathcal{C}_A)_k$, let $U_{i,A}$ denote the preimage of $U_i$ in $C_A$, etc. Let $V_i$ denote the preimage of $U_i$ in $D$; we will use the same conventions $\mathcal{V}_i,V_{i,A}$, etc.

As in the proof of \Cref{proposition:rd_smooth}, for each $\mathcal{U}_{i,A^{\prime}}$ not containing a branch point, there is a unique \'{e}tale cover $\mathcal{V}_{i,A^{\prime}} \to \mathcal{U}_{i,A^{\prime}}$ up to unique isomorphism that lifts the \'{e}tale cover $\mathcal{V}_{i,A} \to \mathcal{U}_{i,A}$. Similarly, for every $U_{ij,A^{\prime}}$, there is a unique \'{e}tale cover $V_{ij,A^{\prime}}$ up to unique isomorphism that lifts the \'{e}tale cover $V_{ij,A} \to U_{ij,A}$. By this uniqueness on overlaps, the groupoid of lifts of $\varphi_A: \mathcal{D}_A \to \mathcal{C}_A$ to $A^{\prime}$ is equivalent to the product over the $U_i$ containing branch points of the groupoids of lifts of $V_{i,A} \to U_{i,A}$ to $A^{\prime}$ with branch divisor $\Sigma_{A^{\prime}} \subset U_{i,A^{\prime}}$.

It remains to show that there is a unique lift of $V_{i,A} \to U_{i,A}$ to $A^{\prime}$ up to unique isomorphism with branch divisor $\Sigma_{A^{\prime}} \subset U_{i,A^{\prime}}$. Because 2 is invertible, the cover $V_i \to U_i$ is tamely ramified. Thus, \cite[Theorem 4.8]{f_hurwitz} implies that we get a unique deformation $V_{i,A^{\prime}} \to U_{i,A^{\prime}}$ up to unique isomorphism. This concludes the proof of the proposition.
\end{proof}

\subsection{Stacks of marked covers with resolvents}\label{section:stacks_marked}
Recall from \Cref{section:top_hurwitz} that from the topological and quantum algebraic point of view, it is more natural to study marked Hurwitz spaces, which parametrize $G$-covers of the closed unit disc with a marked point in the fiber above 1, than unmarked Hurwitz spaces, which are obtained by taking the quotient of the corresponding marked Hurwitz space by $G$. This leads to some awkwardness, as the Hurwitz spaces that algebraic geometers typically work with are unmarked.

In this section, we will provide an algebraic construction of marked Hurwitz spaces of degree $d$ covers as open substacks of stacks of marked covers with resolvents. The construction is a slight modification of Abramovich-Corti-Vistoli's construction of the stack of Teichm\"{u}ller structures in \cite[\S7.6]{acv_twisted_bundles}: for a family of $G$-covers $D \to C \to S$ with a marked point $\sigma: S \to C$ above which $D \to C$ is \'{e}tale (this corresponds to there being no stackiness on $\sigma$), Abramovich-Corti-Vistoli add the datum of a trivialization of the $G$-torsor on $S$ pulled back along $\sigma$, which serves as a marking. In our situation, we may not have any marked points $\sigma$ without stackiness, so instead, we trivialize a fiber ``infinitesimally close'' to $\sigma$. To do this, we need to pick a direction in which we travel an infinitesimal distance away from $\sigma$, so we work over the stack $\mathcal{M}^{\fr}$ from \Cref{section:framings_smooth}, which parametrizes pointed curves with a framing of the normal bundle of $\sigma \subset C$.

Recall from \Cref{section:framings_smooth} that for $\alpha = (a_1,\ldots,a_n)$ where all the $a_i$ divide $d!$, for each $i$, there is a canonical morphism\footnote{All stacks here are assumed to be over $\mathbb{S}_d$, so we omit the subscript $\mathbb{Z}\left[\frac{1}{d!}\right]$ from \Cref{section:framings_smooth}.} $\psi_{i,\underline{a}}^{\univ}: \mathcal{M}^{\fr,\sm;*,n} \to \mathcal{C}_{\underline{a}}^{\univ}$, where $\mathcal{C}_{\underline{a}}^{\univ} \to \mathcal{M}^{\orb,\sm;*,n}(\underline{a})$ is the universal orbicurve. Set $\mathcal{R}^{d,\fr} \coloneqq \mathcal{R}^d \times_{\mathcal{M}} \mathcal{M}^{\fr}$. For any morphism $S \to \mathcal{R}^{d,\fr,\sm;*,n}(\underline{a})$ corresponding to an object $(\mathcal{C} \to C \to S;\sigma;\chi)$ of $\mathcal{R}^d(S)$ and framings $\alpha_i: \sigma_i^{-1}\mathcal{O}_C(D_i) \xrightarrow[]{\sim} \mathcal{O}_S$, we get sections $\psi_i: S \to \mathcal{C}$ by pulling back $\psi_{i,\underline{a}}^{\univ}$. Because the sections $\psi_i$ are over the marked points $\sigma_i$, $\chi: \mathcal{C} \to \mathcal{X}_d$ sends the images of the $\psi_i$ to the \'{e}tale locus $\mathcal{E}_d \cong B\mathbf{S}_d$. Thus, $\chi \circ \psi_1,\ldots,\chi \circ \psi_n$ classify $n$ degree $d$ \'{e}tale covers $Y_i \to S$. Everything here is canonical, so we get a morphism $\mathcal{R}^{d,\fr,\sm;*,n}(\underline{a}) \to B\mathbf{S}_d^n$. We take the disjoint union over all $\underline{a}$ to get a morphism $\mathcal{R}^{d,\fr,\sm;*,n} \to B\mathbf{S}_d^n$.

\begin{definition}\label{definition:marked_stacks}
Define the \emph{big stack of marked degree $d$ covers with resolvents} $\mathcal{R}_M^{d,\fr,\sm}$ as the disjoint union of the stacks $\mathcal{R}_M^{d,\fr,\sm;*,n} \coloneqq \mathcal{R}^{d,\fr,\sm;*,n} \times_{B\mathbf{S}_d^n} \mathbb{S}_d$. Define the \emph{big Hurwitz stack of marked simply branched degree $d$ covers} $\mathcal{H}_M^{d,\fr,\sm} \coloneqq \mathcal{H}^{d,\fr,\sm} \times_{\mathcal{R}^{d,\fr,\sm}} \mathcal{R}_M^{d,\fr,\sm}$.
\end{definition}

In other words, $\mathcal{R}_M^{d,\fr,\sm}$ parametrizes objects of $\mathcal{R}^{d,\fr,\sm}$ with trivializations of the associated $\mathbf{S}_d$-torsors over the sections $\psi_i: S \to \mathcal{C}$. Equivalently, we can view trivializations of $\mathbf{S}_d$-torsors over $S$ as trivializations of the associated degree $d$ \'{e}tale covers, i.e. identifications of the covers with $\{1,\ldots,d\} \times S$. We will refer to such a trivialization over $\psi_i$ as a \emph{marking}.

By construction, $\mathcal{R}_M^{d,\fr,\sm;*,n} \to \mathcal{R}^{d,\fr,\sm;*,n}$ is a $\mathbf{S}_d^n$-torsor, so it inherits many of the nice geometric properties of $\mathcal{R}^d$. One additional useful geometric property of $\mathcal{R}_M^{d,\fr,\sm;*,n}$ is the following.

\begin{proposition}\label{proposition:marked_rep}
For $n \ge 1$, the composite $\br_M: \mathcal{R}_M^{d,\fr,\sm;*,n} \to \mathcal{R}^{d,\fr,\sm;*,n} \xrightarrow[]{\br} \mathcal{M}^{\fr,\sm;*,n}$ is representable.
\end{proposition}

\begin{proof}
Because $\mathcal{R}_M^{d,\fr,\sm;*,n} \to \mathcal{R}^{d,\fr,\sm;*,n}$ is an $\mathbf{S}_d^n$-torsor (in particular, representable) and $\mathcal{R}^{d,\fr,\sm;*,n} \xrightarrow[]{\br} \mathcal{M}^{\fr,\sm;*,n}$ is representable by Deligne-Mumford stacks, the given morphism is representable by Deligne-Mumford stacks. Thus, it suffices to check representability on $k$-points, where $k$ is an algebraically closed field.

Let $(\mathcal{C} \to C \to \Spec k;\sigma;\chi: \mathcal{C} \to \mathcal{X}_d)$ be an object of $\mathcal{R}^{d,\sm;*,n}(k)$. Let $\mathcal{U} = \chi^{-1}(\mathcal{E}_d)$ with coarse moduli space $U$; recall that $\mathcal{U}$ is required to contain all the stacky points of $\mathcal{C}$. Let $W \to \mathcal{U}$ be the $\mathbf{S}_d$-torsor induced by $\chi$. Note that $W$ is an algebraic space because $\chi$ is representable. Because $W \to U$ is finite, $W$ is actually a scheme and thus a smooth curve.

Picking a framing of the normal bundle to each $p_i = \sigma_i(k) \subset C$ amounts to picking a tangent vector to $p_i$. Let $\overline{p}_i \in \mathcal{C}(k)$ be the geometric point mapping to $p_i \in C(k)$. Independent of what tangent vector we pick to $p_i$ (all these framings are uniquely isomorphic by scaling), picking a marking over $\overline{p}_i$ is equivalent to picking an identification of the geometric fiber $\beta_i: (\overline{p}_i \times_{\mathcal{U}} W)(k) \xrightarrow[]{\sim} \mathbf{S}_d$ as $\mathbf{S}_d$-sets.

The proposition will follow if we prove that any automorphism of the data $(\mathcal{C} \to C \to \Spec k;\sigma;\chi: \mathcal{C} \to \mathcal{X}_d)$ over $C$ that preserves the markings $\beta_i$ is trivial. First, recall that $\mathcal{C}$ has no nontrivial automorphisms over $C$. Thus, an automorphism of the above data is an automorphism of $\chi$. Next, note that any automorphism of $\chi$ is determined by its restriction $\chi|_{\mathcal{U}}$: an automorphism of $\chi$ is the same as an automorphism of the associated $G_d$-bundle fixing some section of the associated $V_d$-bundle, and such an automorphism is determined by what it does over the open dense substack $\mathcal{U}$ because $\mathcal{C}$ is separated. Thus, we just need to show that an $\mathbf{S}_d$-equivariant automorphism of $W$ over $\mathcal{U}$ that respects the markings $\beta_i$ is trivial.

To show that such an automorphism is trivial, we will only need one marking, say $\beta_1$. Let $\phi: W \xrightarrow[]{\sim} W$ be an $\mathbf{S}_d$-equivariant automorphism over $\mathcal{U}$. Then $\phi$ pulls back to an $\mathbf{S}_d$-equivariant automorphism of the trivial $\mathbf{S}_d$-torsor $W \times_{\mathcal{U}} W \cong \mathbf{S}_d \times W \to W$, so $\phi$ corresponds to some element of $\mathbf{S}_d$, which we'll denote $\widetilde{\phi}$. If we let $w \in W(k)$ be a preimage of $\overline{p}_1 \in \mathcal{U}(k)$, then because $\phi$ is required to act trivially on $(\overline{p}_1 \times_{\mathcal{U}} W)(k)$, $\widetilde{\phi}$ acts trivially on $(w \times_W (\mathbf{S}_d \times W))(k) \cong \mathbf{S}_d$. Thus, $\widetilde{\phi} = 1$, and $\phi$ is trivial. We conclude that there are no automorphisms of $(\mathcal{C} \to C \to \Spec k;\sigma;\chi: \mathcal{C} \to \mathcal{X}_d)$ preserving the marking $\beta_1$, so we have proven the proposition.
\end{proof}

\begin{remark}\label{remark:marked_arb_g}
This paper is only about $\mathbf{S}_d$-covers, so for our purposes, marked Hurwitz spaces of simply branched $\mathbf{S}_d$-covers are enough. However, we expect our construction of marked Hurwitz spaces to work for constructing marked Hurwitz spaces of $G$-covers for arbitrary $G$.
\end{remark}

\subsection{Fixing a curve}\label{section:fixing_a_curve}
Up to now, we have worked over the entire stack $\mathcal{M}$. Now let us fix a base field $k$ over $\mathbb{S}_d$ and a smooth $n$-pointed curve $(C \to \Spec k;\sigma)$. These data correspond to a point $m \in \mathcal{M}^{0,n}(k)$. Let $b$ be a nonnegative integer, and consider the pullback of the sequence of maps $\mathcal{R}^{d;b,n} \xrightarrow[]{\br} \mathcal{M}^{b,n} \to \mathcal{M}^{0,n}$ along $m$. The fiber of $\mathcal{M}^{b,n}$ parametrizes degree $b$ effective Cartier divisors on $C^{\gen}$, so it is identified with $\Sym^b(C^{\gen})$. We denote the fiber of $\mathcal{R}^{d;b,n} \to \mathcal{M}^{0,n}$ over $m$ by $\mathcal{R}_{(C;\sigma),b}^d$ and continue to use $\br: \mathcal{R}_{(C;\sigma),b}^d \to \Sym^b(C^{\gen})$ to denote the branch morphism. If $C$ has no marked points, we will simply write $\mathcal{R}_{C,b}^d$. By \Cref{theorem:rd_nice} and \Cref{corollary:sm_smooth}, $\mathcal{R}_{(C;\sigma),b}^d$ is a smooth Deligne-Mumford stack of dimension $b$ over $k$, and $\br$ is a proper morphism.

We also define $\mathcal{H}_{(C;\sigma),b}^d \subset \mathcal{R}_{(C;\sigma),b}^d$ as the preimage of $\mathcal{H}^d \subset \mathcal{R}^d$. Similarly, we can define the marked variants $\mathcal{H}_{M,(C;\sigma),b}^d \subset \mathcal{R}_{M,(C;\sigma),b}^d$. Here, we do not need to specify the framings on $(C;\sigma)$ because a choice of framing at $\sigma_i$ is equivalent to a choice of tangent vector and thus all framings are uniquely isomorphic to each other via scaling. These isomorphisms between framings induce isomorphisms between the different $\mathcal{R}_{M,(C;\sigma),b}^d$ we get from different framings.

\begin{proposition}\label{proposition:rm_hm_rep}
\begin{enumerate}[(a)]
\item $\mathcal{R}_{M,(C;\sigma),b}^d$ is an algebraic space.

\item $\mathcal{H}_{M,(C;\sigma),b}^d \to \Conf^b(C^{\gen})$ is finite \'{e}tale. In particular, $\mathcal{H}_{M,(C;\sigma),b}^d$ is represented by a scheme.
\end{enumerate}
\end{proposition}

\begin{proof}
Part (a) follows from \Cref{proposition:marked_rep}. Part (b) then follows because the map $\mathcal{H}_{M,(C;\sigma),b}^d \to \Conf^b(C^{\gen})$ is proper (\Cref{theorem:rd_nice}(b)) and \'{e}tale (\Cref{proposition:hurwitz_etale}, along with the fact that $\mathcal{H}_{M,(C;\sigma),b}^d \to \mathcal{H}_{(C;\sigma),b}^d$ is an $\mathbf{S}_d^n$-torsor).
\end{proof}

In addition to smoothness, the key reason why our construction $\mathcal{R}^d$ is useful for studying intersection cohomology is the following result. Recall that a morphism of equidimensional finite type Deligne-Mumford stacks $\pi: X \to Y$ of the same dimension is \emph{ (stratified) small} if $\pi$ is proper and $Y$ admits a (finite) stratification by locally closed substacks $Y = \bigcup_iY_i$ such that \begin{itemize}
\item For $Y_i \subset Y$ open: The restriction $\pi^{-1}(Y_i) \to Y_i$ is quasi-finite.

\item For $Y_i \subset Y$ not open: For all geometric points $y \to Y_i$, we have $\dim Y_i + 2\dim\pi^{-1}(y) < \dim Y$.
\end{itemize}

\begin{theorem}\label{theorem:dense_small34}
For $d \in \{3,4\}$, \begin{enumerate}[(a)]
\item $\mathcal{H}_{(C;\sigma),b}^d$ is dense in $\mathcal{R}_{(C;\sigma),b}^d$.

\item $\br$ is small.
\end{enumerate}
\end{theorem}

\begin{proof}
We may assume that $k$ is algebraically closed. Consider the stratification $\Sym^b(C^{\gen}) = \bigcup_{\mu}\Sym^{\mu}(C^{\gen})$ by locally closed subvarieties indexed by partitions $\mu = (b_1,\ldots,b_{\ell})$ of $b$, where $\Sym^{\mu}(C^{\gen})$ parametrizes divisors of the form $\sum_ib_i[c_i]$, where $c_1,\ldots,c_{\ell} \in C^{\gen}(k)$ are pairwise distinct. Then for $p \in (\Sym^{\mu}(C^{\gen}))(k)$, we have the following bounds by \Cref{proposition:fiber_bound_34}: \begin{align*}
\dim p \times_{\Sym^b(C^{\gen})} \mathcal{R}_{(C;\sigma),b}^3 &\le \sum_i\left\lfloor\frac{b_i}{6}\right\rfloor \\
\dim p \times_{\Sym^b(C^{\gen})} \mathcal{R}_{(C;\sigma),b}^4 &\le \sum_i\left\lfloor\frac{b_i}{4}\right\rfloor
\end{align*}
Note that $\Conf^b(C^{\gen}) = \Sym^{1,\ldots,1}(C^{\gen})$ and that $\mathcal{H}_{(C;\sigma),b}^d = \br^{-1}(\Conf^b(C^{\gen}))$.

We now prove the theorem for both $d = 3$ and $d = 4$: \begin{enumerate}[(a)]
\item Let $Z$ be a connected component of $\mathcal{R}_{(C;\sigma),b}^d$. Then by \Cref{corollary:sm_smooth}, $Z$ is smooth of dimension $b$ and thus irreducible. For the sake of contradiction, suppose $\mathcal{H}_{(C;\sigma),b}^d \cap Z$ is empty. Then $Z$ is contained in the union of the $\br^{-1}(\Sym^{\mu}(C^{\gen}))$, where $\mu \ne (1,\ldots,1)$. However, by the bounds on the fiber dimension of $\br$, for $\mu = (b_1,\ldots,b_{\ell}) \ne (1,\ldots,1)$, we have \begin{align*}
\dim\br^{-1}(\Sym^{\mu}(C^{\gen})) &\le \dim\Sym^{\mu}(C^{\gen}) + \sum_i\left\lfloor\frac{b_i}{4}\right\rfloor = \ell + \sum_i\left\lfloor\frac{b_i}{4}\right\rfloor = \sum_i\left(\left\lfloor\frac{b_i}{4}\right\rfloor + 1\right) \\
&< \sum_ib_i = b,
\end{align*}
where we use the fact that $\mu \ne (1,\ldots,1)$ for the last inequality. Since $b = \dim Z \le \max_{\mu \ne (1,\ldots,1)}\dim\br^{-1}(\Sym^{\mu}(C^{\gen})) < b$, we have a contradiction. Thus, $\mathcal{H}_{(C;\sigma),b}^d \cap Z$ is nonempty and thus dense because $Z$ is irreducible.

\item By \Cref{proposition:hurwitz_etale} $\br$ is \'{e}tale over the open stratum $\Conf^n(C^{\gen}) = \Sym^{1,\ldots,1}(C^{\gen})$. In particular, because $\br$ is proper, $\br$ is also quasi-finite over this stratum. Now consider $\mu = (b_1,\ldots,b_{\ell}) \ne (1,\ldots,1)$. For any $p \in \Sym^{\mu}(C^{\gen})(k)$, we have \begin{align*}
\dim\Sym^{\mu}(C^{\gen}) + 2\dim\br^{-1}(p) &\le \ell + 2\sum_i\left\lfloor\frac{b_i}{4}\right\rfloor = \sum_i\left(2\left\lfloor\frac{b_i}{4}\right\rfloor + 1\right) \\
&< \sum_ib_i = b,
\end{align*}
where we again use the fact that $\mu \ne (1,\ldots,1)$ for the last inequality. We conclude that $\br$ is small.
\end{enumerate}
\end{proof}

\begin{corollary}\label{corollary:marked_dense_small34}
For $d \in \{3,4\}$, \begin{enumerate}[(a)]
\item $\mathcal{H}_{M,b}^d$ is dense in $\mathcal{R}_{M,(C;\sigma),b}^d$.

\item $\br_M: \mathcal{R}_{M,(C;\sigma),b}^d \to \Sym^b(C^{\gen})$ is small.

\item Let $j$ denote the open immersion $\Conf^b(C^{\gen}) \xhookrightarrow{} \Sym^b(C^{\gen})$, and let $\pi: \mathcal{H}_{M,(C;\sigma),b}^d \to \Conf^b(C^{\gen})$ be the restriction of $\br_M$. Then $$j_{!*}\pi_*\underline{\mathbb{Q}_{\ell}}_{\mathcal{H}_{M,(C;\sigma),b}^d} \simeq R\br_{M*}\underline{\mathbb{Q}_{\ell}}_{\mathcal{R}_{M,(C;\sigma),b}^d}.$$
\end{enumerate}
\end{corollary}

\begin{proof}
Parts (a) and (b) follow immediately from \Cref{theorem:dense_small34} because $\mathcal{R}_{M,(C;\sigma),b}^d \to \mathcal{R}_{(C;\sigma),b}^d$ is an $\mathbf{S}_d$-torsor.

Part (c) follows from parts (a) and (b), the fact that $\mathcal{H}_{M,(C;\sigma),b)}^d$ is smooth, and the following claim: if $f: X \to Y$ is a small morphism of algebraic spaces with $f^{-1}(V) \to V$ finite over a dense open subset $V \subset Y$, then for a local system $\mathcal{L}$ on $f^{-1}(V)$, we have $$j_{Y!*}f_*\mathcal{L} \simeq Rf_*j_{X!*}\mathcal{L},$$ where $j_X: f^{-1}(V) \xhookrightarrow{} X$ and $j_Y: V \xhookrightarrow{} Y$ are the open immersions. This claim can be proven in exactly the same way as the corresponding claim for schemes (e.g. \cite[Lemma 4.3]{mv_geometric_langlands}) by using the fact\footnote{This fact can be proven using the corresponding fact for schemes and the fact that quasi-separated algebraic spaces have open dense subsets represented by schemes.} that proper algebraic spaces of dimension $r$ over an algebraically closed field have cohomology in degrees at most $2r$.
\end{proof}

\begin{remark}\label{remark:analytic_ic}
Note that if $k = \mathbb{C}$, then we can use the same proof to prove the analogue of \Cref{corollary:marked_dense_small34}(c) in the analytic category: for $d \in \{3,4\}$, $j_{!*}\pi_*\underline{\mathbb{Q}}_{\mathcal{H}_{M,(C;\sigma),b}^d(\mathbb{C})} \simeq R\br_{M*}\underline{\mathbb{Q}}_{\mathcal{R}_{M,(C;\sigma),b}^d(\mathbb{C})}$.
\end{remark}

\subsubsection{Application to smoothability of curves}\label{section:app_smoothability}
In addition to the above properties relevant to our study of intersection cohomology and arithmetic statistics, we get the following application of our construction to the theory of smoothing curve singularities. Recall that a finite type scheme $X_0$ over an algebraically closed field $k$ is \emph{smoothable} if there exists an integral variety $T$ over $k$ and a flat family $X \to T$ with smooth generic fiber such that $X_0$ appears as a fiber above a closed point of $T$.

\begin{proposition}\label{proposition:smoothable_34}
Fix $d \in \{3,4\}$. Over an algebraically closed field $k$ over $\mathbb{S}_d$, any reduced proper curve admitting a degree $d$ map to a smooth curve is smoothable.
\end{proposition}

\begin{proof}
Let $C$ be a smooth proper curve over $k$, and let $\pi: D \to C$ be a degree $d$ cover with $D$ reduced. Because $D$ is reduced and the characteristic of $k$ does not divide $d$, $\pi$ is generically \'{e}tale. Let $b$ be the degree of the branch divisor of $\pi$. If $d = 3$, we can find a point $q \in \mathcal{R}_{C,b}^3(k)$ corresponding to the degree 3 cover $\pi$. If $d = 4$, then because every quartic algebra over a Dedekind domain has a cubic resolvent \cite[Theorem 1.4]{o_rings_small}, we can find a cubic resolvent for $D \to C$ by deleting a point $c \in C(k)$ above which $\pi$ is \'{e}tale, finding a cubic resolvent for $\pi^{-1}(C - c) \to C - c$, and then adding $c$ back in (the cubic resolvent will be the unique extension \'{e}tale over $c$). This means that we can find some point $q \in \mathcal{R}_{C,b}^4(k)$ whose associated quartic cover is $D \to C$.

Let $U \to \mathcal{R}_{C,b}^d$ be some \'{e}tale chart around $q$. Then we can pick $U$ to be integral because $\mathcal{R}_{C,b}^d$ is smooth. Let $\overline{D} \to C \times U \to U$ be the associated degree $d$ cover. By \Cref{theorem:dense_small34}(a), the generic point of $U$ maps to $\mathcal{H}_{C,b}^d$, so the generic fiber of $\overline{D} \to U$ is a simply branched cover of $C$ and hence a smooth curve. We conclude that $D$ is smoothable.
\end{proof}

\begin{remark}\label{remark:smoothable_curves}
For $d = 3$, \Cref{proposition:smoothable_34} can be obtained using Schaps's theorem on the smoothability of codimension 2 Cohen-Macaulay subschemes in low-dimensional smooth affine varieties \cite[Theorem 2]{s_deformations_cm}, as in the proof of \cite[Theorem 5.5]{d_compact}. However, for an arbitrary finite cover $D \to C$ (as long as the degree is at least 3), $D$ can have non-Gorenstein singularities (e.g. spatial triple points). Thus, it does not seem possible to apply any special theorems about codimension 3 Gorenstein subschemes to the $d = 4$ case.

In the other direction, we are unsure what the minimal degree of a map from an unsmoothable reduced curve to a smooth curve is. The author is aware of two sources of unsmoothable curves. The first family of examples is due to Mumford \cite{m_pathologies_iv}, and it is more complicated than we would like to describe here. The second family of examples arises by taking $r$ general lines through the origin in $\mathbb{A}^n$ for certain pairs $(n,r)$, as studied in \cite{g_deformation,p_deformations,s_number_points}. By \cite[Proposition 10]{s_number_points}, the singularity given by 10 general lines in $\mathbb{A}^6$ is not smoothable. This means that there are degree 10 reduced curve singularities that are not smoothable.

We are very optimistic that we can extend our methods to show that degree 5 reduced curve singularities are smoothable (see below), but we do not know what happens between 6 and 9.
\end{remark}

\subsubsection{Conjectures about the $d = 5$ Igusa zeta function}\label{section:conj_igusa5}
We conclude this section by discussing the $d = 5$ situation. There are two things missing from our knowledge when $d = 5$ that would enable us to extend the results of this paper to this case: (1) the computation of the Igusa zeta function for $(G_5,V_5)$ (2) the isomorphism between $\mathcal{X}_5$ and a stack of quintic covers with sextic resolvents.

The more important one for all the geometry in this paper is (1). The results about density and smallness follow from the computations of the Igusa zeta functions for $(G_3,V_3)$ and $(G_4,V_4)$. Knowing certain properties of the Igusa zeta function for $(G_5,V_5)$ would allow us to prove versions of \Cref{proposition:fiber_bound_34} and thus \Cref{theorem:dense_small34} for $d = 5$.

Let $K$ be a local field of residue characteristic not dividing $d!$. Let $q$ be the size of the residue field, and let $t = q^{-s}$ and $(a) = 1 - q^{-a}$. Let $$I_5(s) = \frac{1}{(1)(2)^2(3)^2(4)^2(5)}\int_{V_5(\mathcal{O}_K)}|\Delta_5(v)|^{s - 1}d\mu_{V_5}(v).$$ We make two conjectures, the second stronger than the first.

\begin{conjecture}\label{conjecture:igusa5_value}
\begin{enumerate}[(a)]
\item $I_5(s)$ is a power series in $q$ and $t$ with positive integer coefficients independent of $K$ such that all nonzero terms $q^at^b$ with $b \ge 2$ have $a + 1 < b$.

\item $I_5(s)$ is a power series in $q$ and $t$ with positive integer coefficients independent of $K$ such that all nonzero terms $q^at^b$ with $b \ge 2$ have $2a + 1 < b$.
\end{enumerate}
\end{conjecture}

We can prove the following theorem in exactly the same way as \Cref{theorem:dense_small34}.

\begin{theorem}\label{theorem:dense_small5}
\begin{enumerate}[(a)]
\item If \Cref{conjecture:igusa5_value}(a) is true, then $\mathcal{H}_{M,b}^5$ is dense in $\mathcal{R}_{M,(C;\sigma),b}^5$.

\item If \Cref{conjecture:igusa5_value}(b) is true, then the composite $$\mathcal{R}_{M,(C;\sigma),b}^5 \to \mathcal{R}_{(C;\sigma),b}^5 \xrightarrow[]{\br} \Sym^b(C^{\gen})$$ is small.

\item Let $j$ denote the open immersion $\Conf^b(C^{\gen}) \xhookrightarrow{} \Sym^b(C^{\gen})$, and let $\pi: \mathcal{H}_{M,(C;\sigma),b}^d \to \Conf^b(C^{\gen})$ be the restriction of $\br_M$. If \Cref{conjecture:igusa5_value} is true, then $$j_{!*}\pi_*\underline{\mathbb{Q}_{\ell}}_{\mathcal{H}_{M,(C;\sigma),b}^d} \simeq R\br_{M*}\underline{\mathbb{Q}_{\ell}}_{\mathcal{R}_{M,(C;\sigma),b}^d}.$$
\end{enumerate}
\end{theorem}

\begin{corollary}\label{proposition:smoothable_5}
Assume \Cref{conjecture:igusa5_value}(a) holds. Then over an algebraically closed field $k$ over $\mathbb{S}_5$, any reduced proper curve admitting a degree 5 map to a smooth curve is smoothable.
\end{corollary}

\begin{proof}
We want to copy the proof of \Cref{proposition:smoothable_34} for $d = 4$. We do not have an isomorphism between $\mathcal{X}_d$ and a stack of quintic covers with sextic resolvents as in the $d = 4$ case. However, we do know that for a Dedekind domain $R$, the isomorphism classes $\mathcal{X}_d(R)/\sim$ are in natural bijection with the isomorphism classes of quintic $R$-algebras with sextic resolvent and that each quintic $R$-algebra has a sextic resolvent \cite[Theorem 1.2]{o_quintic_dedekind}. By the same reasoning as in the proof of \Cref{proposition:smoothable_34}, if we are given a generically \'{e}tale quintic cover of curves $D \to C$ over $k$ with branch divisor of degree $b$, where $C$ is smooth, then we can find some point of $\mathcal{R}_{C,b}^d(k)$ whose corresponding quintic cover is $D \to C$. We can then proceed as in the proof of \Cref{proposition:smoothable_34}.
\end{proof}

\begin{remark}\label{remark:igusa5_believable}
We present some evidence for \Cref{conjecture:igusa5_value}. It has been proven that Igusa zeta functions over $p$-adic fields are rational functions \cite{d_rationality,i_complex_powers}, but the question remains open in positive characteristic. If $I_5(s)$ is a rational function in $t$, it seems rather likely that \Cref{conjecture:igusa5_value}(a) holds. $I_5(s)$ converges for all $\mathrm{Re}(s) \ge 1$, so all the poles have $\mathrm{Re}(s) < 1$. If we assume rationality, all the terms $q^at^b$ (other than 1) coming from the denominator would have ``slope'' $\frac{a}{b} < 1$, which is almost what we want in \Cref{conjecture:igusa5_value}(a).

There is another conjecture \cite[\S2]{m_survey_igusa} that the poles of an Igusa zeta function, up to subtracting a natural number, are contained in the set of roots of the associated $b$-function. This has been established by Bernstein \cite{b_analytic_continuation} in the archimedean case. In our situation, this conjecture would imply that the poles of $I_5(s)$ have real part at most $\frac{1}{3}$, as we can check using the corresponding $b$-function\footnote{Here, we add 1 to all the zeros because we are integrating $|\Delta_5|^{s - 1}$ rather than $|\Delta_5|^s$.} \cite[Appendix, (11)]{k_prehomog}. This would translate to the terms $q^at^b$ from the denominator having slope $\frac{a}{b} \le \frac{1}{3}$, almost giving us \Cref{conjecture:igusa5_value}(b).

In the cases $d \in \{3,4\}$, adding 1 to all the zeros of the corresponding $b$-functions that are at least $-1$ give us the slopes appearing in the denominators: $0,\frac{1}{6}$ for $d = 3$ and $0,\frac{1}{6},\frac{1}{4}$ for $d = 4$. Assuming this pattern continues to hold for $d = 5$, we get the slopes $0,\frac{1}{10},\frac{1}{6},\frac{1}{5},\frac{1}{4},\frac{3}{10},\frac{1}{3}$ for $d = 5$. As we will see in \Cref{section:covers_of_a1}, these slopes correspond to the exponents of secondary terms in counts of $\mathbb{F}_q[t]$-algebras with resolvents and the bidegrees of generators of the cohomology of the corresponding Nichols algebra.
\end{remark}

\section{Counting covers of $\mathbb{A}^1$ with resolvents}\label{section:covers_of_a1}
In this section, we study the morphisms $\mathcal{R}_{M,(\mathbb{P}^1;\infty),b}^d \to \mathcal{R}_{(\mathbb{P}^1;\infty),b}^d \xrightarrow[]{\br} \Sym^b(\mathbb{A}^1)$.

We consider these objects over $\mathbb{S}_d$. To make sense of this, we need to provide a framing of the normal bundle to $\infty_{\mathbb{S}_d} \subset \mathbb{P}_{\mathbb{S}_d}^1$. For this, we can just take the tangent vector $\frac{\partial}{\partial x^{-1}}$. For each $r | d!$, the smooth orbicurve $\mathbb{P}_r^1$ over $\mathbb{S}_d$ with automorphism groups of order $r$ at $\infty$ can be constructed by gluing the stacky neighborhood $[\mathbb{A}^1/\mu_r]$ above $\infty$ with the neighborhood $\mathbb{A}^1$ of 0 using the isomorphism $[\mathbb{G}_m/\mu_r] \cong \mathbb{G}_m$ induced by the $(-r)$th power map $\mathbb{G}_m \to \mathbb{G}_m$. Our framing gives us the sections $$\psi_r: \mathbb{S}_d \xrightarrow[]{0} \mathbb{A}^1 \to [\mathbb{A}^1/\mu_r] \to \mathbb{P}_r^1$$ over $\infty$.

For us, there are two reasons why $(\mathbb{P}^1;\infty)$ is special. The first is that the objects $\mathcal{R}_{M,(\mathbb{P}^1;\infty),b}^d \to \mathcal{R}_{(\mathbb{P}^1;\infty),b}^d \to \Sym^b(\mathbb{A}^1)$ are all $\mathbb{G}_m$-equivariant, with the $\mathbb{G}_m$-action on $\Sym^b(\mathbb{A}^1)$ contracting everything to 0. The second is the following lemma, which connects our algebro-geometric constructions to the topological constructions of \Cref{section:top_hurwitz} and thus to the quantum algebra of \Cref{section:nichols_algs}.

\begin{lemma}\label{lemma:ag_top_comp}
There are identifications between the maps $\mathcal{H}_{M,(\mathbb{P}^1;\infty),b}^d(\mathbb{C}) \to \mathcal{H}_{(\mathbb{P}^1;\infty),b}^d(\mathbb{C})/\sim$ and $\Hur_{\mathbf{S}_d,b}^{\tau_d} \to \Hur_{\mathbf{S}_d,b}^{\tau_d}/\mathbf{S}_d$ over $\Conf^b(\mathbb{A}^1)(\mathbb{C})$.
\end{lemma}

\begin{proof}
In this proof, all our algebro-geometric objects will be over $\mathbb{C}$ rather than $\mathbb{S}_d$. This way, we will not have to write that we are taking the base change to $\mathbb{C}$ every time.

First, we note that the maps $\mathcal{H}_{M,(\mathbb{P}^1;\infty),b}^d(\mathbb{C}) \to \mathcal{H}_{(\mathbb{P}^1;\infty),b}^d(\mathbb{C})/\sim$ and $\Hur_{\mathbf{S}_d,b}^{\tau_d} \to \Hur_{\mathbf{S}_d,b}^{\tau_d}/\mathbf{S}_d$ are both quotients by $\mathbf{S}_d$-actions. Thus, it suffices to identify $\mathcal{H}_{M,(\mathbb{P}^1;\infty),b}^d(\mathbb{C})$ and $\Hur_{\mathbf{S}_d,b}^{\tau_d}$ in an $\mathbf{S}_d$-equivariant way over $\Conf^b(\mathbb{A}^1)(\mathbb{C})$.

By \Cref{proposition:rm_hm_rep}, $\mathcal{H}_{M,(\mathbb{P}^1;\infty),b}^d(\mathbb{C}) \to \Conf^b(\mathbb{A}^1)(\mathbb{C})$ is a covering map. The inclusion $\Conf^b(D) \xhookrightarrow{} \Conf^b(\mathbb{A}^1)(\mathbb{C})$ is a homotopy equivalence with homotopy inverse $\Conf^b(\mathbb{A}^1)(\mathbb{C}) \xrightarrow[]{\sim} \Conf^b(D)$ induced by the radial homeomorphism $\mathbb{C} \xrightarrow[]{\sim} D$ we picked back in \Cref{section:hurwitz_a1}. Thus, it suffices to identify the restriction $\mathcal{H}_{M,(\mathbb{P}^1;\infty),b}^d(\mathbb{C})|_{\Conf^b(D)} \to \Conf^b(D)$ with the covering space $\Hur_{\mathbf{S}_d,b}^{\tau_d} \to \Conf^b(D)$ in an $\mathbf{S}_d$-equivariant way.

By \Cref{proposition:hurwitz_moduli}, a point of $\mathcal{H}_{M,(\mathbb{P}^1;\infty),b}^d(\mathbb{C})$ above $S \in \Conf^b(D)$ corresponds to an orbicurve $\mathbb{P}_r^1 \to \mathbb{P}^1$, a degree $d$ cover $\pi: \mathcal{E} \to \mathbb{P}_r^1$ simply ramified above $S$ and \'{e}tale elsewhere, and an identification $\alpha: F(\mathbb{C}) \xrightarrow[]{\sim} \{1,\ldots,d\}$, where $F = \Spec\mathbb{C} \times_{\psi_r,\mathbb{P}_r^1,\pi} \mathcal{E}$. We get a cover of the disc simply branched above $S$ by taking the restriction $\pi^{-1}(\overline{D}) \to \overline{D}$. To get a marking of $\pi^{-1}(1)$, we first define the section $$\psi_r^{\prime}: \Spec\mathbb{C} \xrightarrow[]{1} \mathbb{A}^1 \to [\mathbb{A}^1/\mu_r] \to \mathbb{P}_r^1.$$ By travelling along the line from 0 to 1 in $\mathbb{A}^1(\mathbb{C})$, we can parallel transport our identification $\alpha: F(\mathbb{C}) \xrightarrow[]{\sim} \{1,\ldots,d\}$ to an identification $\alpha^{\prime}: F^{\prime}(\mathbb{C}) \xrightarrow[]{\sim} \{1,\ldots,d\}$, where $F^{\prime} = \Spec\mathbb{C} \times_{\psi_r^{\prime},\mathbb{P}_r^1,\pi} \mathcal{E}$. However, the image of $\psi_r^{\prime}$ is 1, so $F^{\prime} = \pi^{-1}(1)$, and we have our desired marking $\pi^{-1}(\mathbb{C}) \xrightarrow[]{\sim} \{1,\ldots,d\}$. This gives a continuous map $\mathcal{H}_{M,(\mathbb{P}^1;\infty),b}^d(\mathbb{C})|_{\Conf^b(D)} \to \Hur_{\mathbf{S}_d,b}^{\tau_d}$ over $\Conf^b(D)$ by \Cref{corollary:tophur_degree_d}.

If we have a point of $\Hur_{\mathbf{S}_d,b}^{\tau_d}$ corresponding to a marked simply $n$-branched degree $d$ cover of the disc $(X,p,\beta,S)$, then since the inclusion $\overline{D} - S \xhookrightarrow{} \mathbb{A}^1(\mathbb{C}) - S$ is a homotopy equivalence, the covering map $p: X \to \overline{D} - S$ uniquely extends to a covering map $\pi: E^{\circ} \to \mathbb{A}^1(\mathbb{C}) - S$. By Riemann's existence theorem, $E^{\circ}$ is the space of $\mathbb{C}$-points of some smooth curve with a finite \'{e}tale map $\pi: E^{\circ} \to \mathbb{A}^1 - S$. By \Cref{lemma:etale_orbi}, we can complete $E^{\circ}$ to an orbicurve with a degree $d$ cover $\mathcal{E} \to \mathbb{P}_r^1$ (for some $r$) simply branched over $S$. We can use the same parallel transport procedure as above to produce a marking of $F$, where $F$ is defined as above. This provides a continuous inverse $\Hur_{\mathbf{S}_d,b}^{\tau_d} \to \mathcal{H}_{M,(\mathbb{P}^1;\infty),b}^d(\mathbb{C})|_{\Conf^b(D)}$ to the map we defined above.

Both these maps are $\mathbf{S}_d$-equivariant because $\mathbf{S}_d$ acts on both spaces by acting on the markings. We have thus proven the lemma.
\end{proof}

\subsection{$\mathbb{G}_m$-localization}\label{section:gm_localization}
The main way we will use the $\mathbb{G}_m$-action on $$\br_M: \mathcal{R}_{M,(\mathbb{P}^1;\infty),b}^d \to \Sym^b(\mathbb{A}^1)$$ is by applying the following lemma of Sawin. This lemma is the reason for the local-global duality we will discuss in \Cref{section:counting_fqt_algs}.

\begin{lemma}[{\cite[Lemma 2.16]{s_multiple_dirichlet}}]\label{lemma:sawin_gm_local}
Let $X$ be a affine scheme of finite type over a field $k$ with a \emph{contracting} $\mathbb{G}_m$-action, i.e. a $\mathbb{G}_m$-action such that all nonconstant $\mathbb{G}_m$-homogeneous functions on $X$ have positive degree. Let $K$ be a $\mathbb{G}_m$-invariant complex of $\ell$-adic sheaves on $X$. Let 0 be the unique $\mathbb{G}_m$-fixed $\overline{k}$-point of $X$. Then for each $i$, the pullback $$H^i(X_{\overline{k}},K) \to \mathcal{H}^i(K_0)$$ along the closed embedding $0 \xhookrightarrow{} X_{\overline{k}}$ is an isomorphism. Here, $H^i(X_{\overline{k}},K)$ refers to the $i$th cohomology of $X_{\overline{k}}$ with coefficients in $K$, while $\mathcal{H}^i(K_0)$ refers to the $i$th cohomology of the stalk of $K$ at 0.
\end{lemma}

Because the map $\br_M: \mathcal{R}_{M,(\mathbb{P}^1;\infty),b}^d \to \Sym^b(\mathbb{A}^1)$ is $\mathbb{G}_m$-equivariant, the pushforward $R\br_{M*}\underline{\mathbb{Q}_{\ell}}$ is a $\mathbb{G}_m$-invariant complex on $\Sym^b(\mathbb{A}^1)$. Moreover, for any $k$, $\Sym^b(\mathbb{A}_k^1)$ is an affine scheme whose $\mathbb{G}_m$-action contracts everything to 0.

Let $\mathcal{F}_{M,b}^d = \br_M^{-1}(0)$ be the central fiber. It is a proper algebraic space over $\mathbb{S}_d$. We can apply the $\mathbb{G}_m$-localization lemma to prove the following facts about the cohomology of $\mathcal{R}_{M,(\mathbb{P}^1;\infty),b}^d$ and $\mathcal{F}_{M,b}^d$.

\begin{lemma}\label{lemma:contract_pure}
Let $\pi: Y \to X$ be a $\mathbb{G}_m$-equivariant proper morphism, where $X$ is an affine scheme of finite type over a field $k$ with a contracting $\mathbb{G}_m$-action with fixed point 0 and $Y$ is a smooth algebraic space over $k$. Then \begin{enumerate}[(a)]
\item $H^i(Y_{\overline{k}},\mathbb{Q}_{\ell}) \cong H^i((Y_0)_{\overline{k}},\mathbb{Q}_{\ell})$

\item Suppose $k = \mathbb{F}_q$. Then both cohomology groups from (a) are pure of weight $i$ in the sense of Deligne: their Frobenius eigenvalues are algebraic integers whose Galois conjugates all have absolute value $q^{i/2}$ in $\mathbb{C}$.
\end{enumerate}
\end{lemma}

\begin{proof}
\begin{enumerate}[(a)]
\item Apply \Cref{lemma:sawin_gm_local} to the $\mathbb{G}_m$-invariant complex $R\pi_*\underline{\mathbb{Q}_{\ell}}_Y$, and conclude by proper base change \cite[\S5]{lo_six_operations_i}.

\item First, observe that the isomorphism $H^i(Y_{\overline{\mathbb{F}}_q},\mathbb{Q}_{\ell}) \cong H^i((Y_0)_{\overline{\mathbb{F}}_q},\mathbb{Q}_{\ell})$ is Frobenius-equivariant because it is induced by the pullback along $0 \xhookrightarrow{} X$.

Fix $\iota: \overline{\mathbb{Q}}_{\ell} \xhookrightarrow{} \mathbb{C}$. It suffices to show that for all $\iota$, the images under $\iota$ of the Frobenius eigenvalues have absolute value $q^{i/2}$, i.e. the $\iota$-weights are $i$. By Sun's generalization of Deligne's theorem to stacks \cite[Theorem 1.4]{s_l_series}, $H_c^i(Y_{\overline{\mathbb{F}}_q},\mathbb{Q}_{\ell})$ and $H^i((Y_0)_{\overline{\mathbb{F}}_q},\mathbb{Q}_{\ell})$ ($Y_0$ is proper) have $\iota$-weights $\le i$.

Because $Y$ is smooth, its dualizing complex in the sense of \cite[\S3]{lo_six_operations_i} is $\underline{\mathbb{Q}_{\ell}}[-\dim Y]$. Thus, Laszlo-Olsson's Verdier duality for stacks \cite[Proposition 4.4.2]{lo_six_operations_i} implies that we have Poincar\'{e} duality $$H^i(Y_{\overline{\mathbb{F}}_q},\mathbb{Q}_{\ell}) \cong H_c^{2\dim Y - i}(Y_{\overline{\mathbb{F}}_q},\mathbb{Q}_{\ell})^{\vee}(-\dim Y).$$ Thus, $H^i(Y_{\overline{\mathbb{F}}_q},\mathbb{Q}_{\ell})$ has $\iota$-weights $\ge 2\dim Y - (2\dim Y - i) = i$.

We know that $H^i(Y_{\overline{\mathbb{F}}_q},\mathbb{Q}_{\ell}) \cong H^i((Y_0)_{\overline{\mathbb{F}}_q},\mathbb{Q}_{\ell})$ has $\iota$-weights $\le i$ and $\ge i$. We conclude that the $\iota$-weights must all be $i$ for any $\iota$, proving that the cohomology is pure of weight $i$.
\end{enumerate}
\end{proof}

\begin{corollary}\label{corollary:rd_fbd_pure}
\begin{enumerate}[(a)]
\item For any algebraically closed field $k$ over $\mathbb{S}_d$, the map $$H^*((\mathcal{R}_{M,(\mathbb{P}^1;\infty),b}^d)_{\overline{k}},\mathbb{Q}_{\ell}) \to H^*((\mathcal{F}_{M,b}^d)_{\overline{k}},\mathbb{Q}_{\ell})$$ is an isomorphism.

\item For any $q$ coprime to $d!$, $H^*((\mathcal{R}_{M,(\mathbb{P}^1;\infty),b}^d)_{\overline{\mathbb{F}}_q},\mathbb{Q}_{\ell}) \cong H^*((\mathcal{F}_{M,b}^d)_{\overline{\mathbb{F}}_q},\mathbb{Q}_{\ell})$ is pure of weight $i$.
\end{enumerate}
\end{corollary}

\begin{proof}
Both parts follow from the corresponding parts of \Cref{lemma:contract_pure}, applied to $\br_M$. Here, we are using that $\br_M$ is proper (\Cref{theorem:rd_nice}(b)) and $\mathbb{G}_m$-equivariant.
\end{proof}

Let $\mathcal{F}_b^d = \br^{-1}(0) \subset \mathcal{R}_{(\mathbb{P}^1;\infty),b}^d$. Then because $\mathcal{R}_{(\mathbb{P}^1;\infty),b}^d \cong [\mathcal{R}_{M,(\mathbb{P}^1;\infty),b}^d/\mathbf{S}_d]$, we have $\mathcal{F}_b^d \cong [\mathcal{F}_{M,b}^d/\mathbf{S}_d]$.

\begin{corollary}\label{corollary:rd_fbd_quotient_pure}
\begin{enumerate}[(a)]
\item For any algebraically closed field $k$ over $\mathbb{S}_d$, the map $$H^*((\mathcal{R}_{(\mathbb{P}^1;\infty),b}^d)_{\overline{k}},\mathbb{Q}_{\ell}) \to H^*((\mathcal{F}_b^d)_{\overline{k}},\mathbb{Q}_{\ell})$$ is an isomorphism.

\item For any $q$ coprime to $d!$, $H^*((\mathcal{R}_{(\mathbb{P}^1;\infty),b}^d)_{\overline{\mathbb{F}}_q},\mathbb{Q}_{\ell}) \cong H^*((\mathcal{F}_b^d)_{\overline{\mathbb{F}}_q},\mathbb{Q}_{\ell})$ is pure of weight $i$.
\end{enumerate}
\end{corollary}

\begin{proof}
The corollary follows from \Cref{corollary:rd_fbd_pure} and the fact that for a finite group $G$ acting on a finite type algebraic space $X$, $H^*([X/G],\mathbb{Q}_{\ell}) \to H^*(X,\mathbb{Q}_{\ell})^G$ is an isomorphism. This fact can be proven using the Leray spectral sequence for $X \to [X/G]$, as all contributions to $H^*([X/G],\mathbb{Q}_{\ell})$ coming from the higher group cohomology of $G$ vanish because $\mathbb{Q}_{\ell}$ has characteristic 0.
\end{proof}

\subsection{Counting low degree $\mathbb{F}_q[t]$-algebras}\label{section:counting_fqt_algs}
In this section, we will compute the point counts $\#\mathcal{R}_{(\mathbb{P}^1;\infty),b}^d(\mathbb{F}_q)$ for $d \in \{3,4\}$ using the corresponding Igusa zeta functions and interpret these point counts as counts of $\mathbb{F}_q[t]$-algebras with discriminant $q^b$. We begin by computing $\#\mathcal{F}_b^d(\mathbb{F}_q)$.

\begin{proposition}\label{proposition:fbd_k_pts}
For a field $k$ over $\mathbb{S}_d$, we have an equivalence of groupoids $$\mathcal{F}_b^d(k) \simeq V_d(k[[t]])_b // G_d(k[[t]]).$$
\end{proposition}

\begin{proof}
By \Cref{proposition:local_description_fibers}, $$\mathcal{F}_b^d(k) \simeq \bigsqcup_E\left(BG_E \times_{BG_{\widehat{E}_0}} V_d(k[[t]])_{\widehat{E}_0,b} // G_d(k[[t]])\right),$$ where \begin{itemize}
\item $E$ ranges over degree $d$ \'{e}tale covers $E \to \mathbb{A}_k^1 - 0$.

\item $\widehat{E}_0 = \Spec k[[t]] \times_{\mathbb{A}_k^1} E$, where $\Spec k[[t]] \to \mathbb{A}_k^1$ is the formal disc around 0.

\item $V_d(k[[t]])_{\widehat{E}_0,b}$ is the subset of $v \in V_d(k[[t]])_b$ whose image in $V_d^{ss}(k((t)))$ parametrizes $\widehat{E}_0 \to \Spec k((t))$.

\item $G_E$ (resp. $G_{\widehat{E}_0}$) is the group of automorphisms of the \'{e}tale cover $E \to \mathbb{A}_k^1 - 0$ (resp. $\widehat{E}_0 \to \Spec k((t))$).
\end{itemize}

The key point of this proof (and much of the reason we can get the point counts we want) is this: the functor $E \mapsto \widehat{E}_0$ is an equivalence from the groupoid of degree $d$ \'{e}tale covers of $\mathbb{A}_k^1 - 0$ to the groupoid of degree $d$ \'{e}tale covers of $\Spec k((t))$. Another way to say this is that the map of \'{e}tale fundamental groups $\pi_1(\Spec k((t)),\overline{\eta}) \to \pi_1(\mathbb{A}_k^1 - 0,\overline{\eta})$ induces an isomorphism on pro-$d!$ completions.

This isomorphism can be seen for $k = \mathbb{C}$ by the comparison with topological $\pi_1$, for algebraically closed $k$ over $\mathbb{S}_d$ by Grothendieck's specialization theorem for the tame fundamental group \cite[XIII, Corollaire 2.12]{gr_sga1}, and then for any $k$ over $\mathbb{S}_d$ by the homotopy exact sequence. Alternatively, we can use the Riemann-Hurwitz formula to see that for $d^{\prime} \le d$, any (necessarily separable) degree $d^{\prime}$ cover of $\mathbb{P}_k^1$ by a smooth curve ramified at most above 0 and $\infty$ must be $\mathbb{P}_{k^{\prime}}^1$ for some $k^{\prime}/k$ and that the pullback of the coordinate $x$ to $\mathbb{P}_{k^{\prime}}^1$ must be a scalar times $x^{d/[k^{\prime}:k]}$, exactly matching the description of extensions of $k((t))$.

With this equivalence of groupoids, we see the map $G_E \to G_{\widehat{E}_0}$ is an isomorphism and that $$\mathcal{F}_b^d(k) \simeq \bigsqcup_E\left(V_d(k[[t]])_{\widehat{E}_0,b} // G_d(k[[t]])\right) \simeq V_d(k[[t]])_b // G_d(k[[t]]).$$
\end{proof}

\begin{corollary}\label{corollary:pt_count_fbd_34}
For $d \in \{3,4\}$, $\sum_{b = 0}^{\infty}\#\mathcal{F}_b^d(\mathbb{F}_q)t^b = I_d(q,t)$.
\end{corollary}

\begin{proof}
Combine \Cref{proposition:fbd_k_pts} for $k = \mathbb{F}_q$ and \Cref{proposition:counting_integral}.
\end{proof}

\begin{proposition}\label{proposition:coh_fbd_34}
For $d \in \{3,4\}$ and for any algebraically closed $k$ over $\mathbb{S}_d$, \begin{enumerate}[(a)]
\item $\dim H^i((\mathcal{F}_b^d)_k,\mathbb{Q}_{\ell})$ is the coefficient of $q^{i/2}t^b$ in $I_d(q,t)$. Moreover, if $k = \overline{\mathbb{F}}_q$, all Frobenius eigenvalues of $H^i((\mathcal{F}_b^d)_{\overline{\mathbb{F}}_q},\mathbb{Q}_{\ell})$ with respect to $\mathbb{F}_q$ are $q^{i/2}$.

\item $\dim H_c^i((\mathcal{R}_{(\mathbb{P}^1;\infty),b}^d)_k,\mathbb{Q}_{\ell})$ is the coefficient of $q^{b - i/2}t^b$ in $I_d(q,t)$. Moreover, if $k = \overline{\mathbb{F}}_q$, all Frobenius eigenvalues of $H_c^i((\mathcal{R}_{(\mathbb{P}^1;\infty),b}^d)_{\overline{\mathbb{F}}_q},\mathbb{Q}_{\ell})$ with respect to $\mathbb{F}_q$ are $q^{i/2}$.
\end{enumerate}
In particular, if $i$ is odd, then the cohomology groups above are 0.
\end{proposition}

\begin{proof}
\begin{enumerate}[(a)]
\item If $k = \overline{\mathbb{F}}_q$, then the first claim follow from \Cref{corollary:rd_fbd_quotient_pure}(b), \Cref{corollary:pt_count_fbd_34}, and Sun's fixed point formula for stacks \cite[Theorem 1.1]{s_l_series}, since purity implies that only the cohomology group $H^i((\mathcal{F}_b^d)_k,\mathbb{Q}_{\ell})$ can contribute to the $q^{ir/2}t^b$ term in $\#\mathcal{F}_b^d(\mathbb{F}_{q^r})$. The second claim follows from the fact that the coefficient of $t^b$ in $I_d(q^r,t)$ is a sum of powers of $q^r$. The first claim then follows for general $k$ by a spreading out argument using Deligne's generic base change theorem \cite[Th. finitude, Th\'{e}or\`{e}me 1.9]{d_sga4.5}.

\item Both claims follow from the isomorphisms $$H_c^i((\mathcal{R}_{M,(\mathbb{P}^1;\infty),b}^d)_{\overline{\mathbb{F}}_q},\mathbb{Q}_{\ell}) \cong H^{2b - i}((\mathcal{R}_{M,(\mathbb{P}^1;\infty),b}^d)_{\overline{\mathbb{F}}_q},\mathbb{Q}_{\ell})^{\vee}(-b) \cong H^{2b - i}((\mathcal{F}_{M,b}^d)_{\overline{\mathbb{F}}_q},\mathbb{Q}_{\ell})^{\vee}(-b),$$ the first of which is Poincar\'{e} duality for the smooth algebraic space $(\mathcal{R}_{M,(\mathbb{P}^1;\infty),b}^d)_{\mathbb{F}_q}$ of dimension $r$ and the second of which is \Cref{corollary:rd_fbd_quotient_pure}(a).
\end{enumerate}
\end{proof}

In the following results, we include the case $d = 3$ for completeness, even though there is a much easier and shorter way to compute $N_{3,q,b}$ (see \Cref{remark:direct_counting}).

\begin{theorem}\label{theorem:fqt_alg_counts_34}
Let $q$ be a prime power coprime to 6.
\begin{enumerate}[(a)]
\item Let $N_{3,q,b}$ be the number of isomorphism classes of cubic $\mathbb{F}_q[t]$-algebras of discriminant $q^b$, inversely weighted by automorphisms over $\mathbb{F}_q[t]$. Then $$\sum_{b = 0}^{\infty}N_{3,q,b}t^b = I_3(q^{-1},qt) = \frac{1 + qt + q^2t^2 + q^3t^3 + q^4t^4}{(1 - q^2t^2)(1 - q^5t^6)}.$$

\item Let $N_{4,q,b}$ be the number of isomorphism classes of quartic $\mathbb{F}_q[t]$-algebras with cubic resolvent of discriminant $q^b$, inversely weighted by automorphisms over $\mathbb{F}_q[t]$. Then $$\sum_{b = 0}^{\infty}N_{4,q,b}t^b = I_4(q^{-1},qt) = \frac{f(q^{-1},qt)}{(1 - qt)(1 - q^2t^2)(1 - q^5t^6)(1 - q^6t^8)(1 - q^9t^{12})}.$$
\end{enumerate}
\end{theorem}

\begin{proof}
We claim that $N_{d,q,b} = \#\mathcal{R}_{M,(\mathbb{P}^1;\infty),b}^d(\mathbb{F}_q)$. Indeed, any morphism $\mathbb{A}_{\mathbb{F}_q}^1 \to \mathcal{X}_d$ sending the generic point to the $\mathcal{E}_d$ can be uniquely completed to some $\mathbb{P}_{r,\mathbb{F}_q}^1 \to \mathcal{X}_d$ sending the stacky point to $\mathcal{E}_d$ by \Cref{lemma:etale_orbi}, and for $d = 3$ (resp. $d = 4$), morphisms $\mathbb{A}_{\mathbb{F}_q}^1 \to \mathcal{X}_d$ correspond to cubic $\mathbb{F}_q[t]$-algebras (resp. quartic $\mathbb{F}_q[t]$-algebras with cubic resolvent) in a discriminant-preserving manner by \Cref{proposition:disc_agrees}.

The theorem then follows from \Cref{proposition:coh_fbd_34}(b) and Sun's fixed point formula for stacks \cite[Theorem 1.1]{s_l_series}.
\end{proof}

We get the following corollary by expanding the generating functions as power series.

\begin{corollary}\label{corollary:fqt_count_terms_34}
Let $q$ be a prime power coprime to 6.
\begin{enumerate}[(a)]
\item There are rational functions $F_b,G_b \in \mathbb{Q}(q^{1/6})$ depending only on $b \pmod{6}$ such that $$N_{3,q,b} = F_bq^b + G_bq^{5b/6}$$ for all $q,b$.

\item There are rational functions $A_b,B_b,C_b,D_b,E_b \in \mathbb{Q}(q^{1/12})$ depending only on $b \pmod{24}$ such that $$N_{4,q,b} = (A_bb + B_b)q^b + C_bq^{5b/6} + (D_bb + E_b)q^{3b/4}$$ for all $q,b$.
\end{enumerate}
\end{corollary}

\begin{remark}\label{remark:extending_counts_5}
All the work in this section is exactly the same in the case $d = 5$. Once we know the value of the Igusa zeta function for $(G_5,V_5)$, we can get a count for quintic $\mathbb{F}_q[t]$-algebras with sextic resolvents. It does not matter that we do not have an isomorphism between $\mathcal{X}_d$ and a stack of quintic covers with sextic resolvents because $\mathbb{F}_q[t]$ is a Dedekind domain and we can use O'Dorney's work \cite{o_quintic_dedekind}.
\end{remark}

\begin{remark}\label{remark:loc_glob_count_duality}
With cubic $\mathbb{F}_q[t]$-algebras, quartic $\mathbb{F}_q[t]$-algebras with cubic resolvents, and quintic $\mathbb{F}_q[t]$-algebras with sextic resolvents, the counting problems exhibit a kind of local-global duality in that the generating function $I_d(q,t)$ for the local problem (\Cref{corollary:local_cover_count}) and the generating function $I_d(q^{-1},qt)$ (\Cref{theorem:fqt_alg_counts_34}) for the global problem are exactly dual to one another. In particular, lower order terms for the global counts correspond to higher order terms for the local counts.

The geometric explanation of the local-global duality in this situation is rather simple: $\mathbb{G}_m$-localization and Poincar\'{e} duality. In fact, any counting problems over $\mathbb{A}^1$ that come from intermediate extensions of local systems coming from Hurwitz spaces (see \Cref{theorem:ks_nichols_ic}) will exhibit some kind of duality because the cohomology of any of these intermediate extensions satisfies Poincar\'{e} duality. If we follow the heuristics in \Cref{section:arith_nichols}, we can hope to find other counting problems with local-global duality by studying these intermediate extensions.
\end{remark}

\begin{remark}\label{remark:direct_counting}
In this remark, we describe the direct method for computing $N_{3,q,b}$ and explain why using it to compute $N_{4,q,b}$ is much more difficult. Under the parametrization given by the prehomogeneous vector space $(G_3,V_3) = (\GL_2,\Sym^3(\std_2) \otimes \det\std_2^{\vee})$, a generically \'{e}tale triple cover of $\mathbb{A}^1$ corresponds to a rank 2 bundle $\mathcal{E} = \mathcal{O}(c_1) \oplus \mathcal{O}(c_2)$ on $\mathbb{P}^1$ along with a section of $\Sym^3\mathcal{E} \otimes \det\mathcal{E}^{\vee}$ satisfying some local condition at $\infty$. For each rank 2 bundle $\mathcal{F} = \mathcal{O}(c_1) \oplus \mathcal{O}(c_2)$ over $\mathbb{P}^1$ where $c_1 + c_2 = b$, we can multiply the total number of sections of $\Sym^2(\mathcal{F}) \otimes \det\mathcal{F}^{\vee}$ by a local density to filter out the sections that correspond to a triple cover of $\mathbb{A}^1$ unramified over $\infty$ and then divide by the size of the automorphism group of $\mathcal{E}$. Similarly, if $c_1 + c_2 = b - 1$ or $c_1 + c_2 = b - 2$, we multiply by a local density to filter out triple covers simply ramified or totally ramified over $\infty$, respectively. These densities can only be nonzero if $c_1,c_2 \ge 0$, so it is clear which bundles we sum over to compute $N_{3,q,b}$.

In contrast, if we try to compute $N_{4,q,b}$ using the same straightforward approach, we run into an enormous amount of casework. In this case, a map $\mathbb{P}^1 \to [V_4/G_4]$ parametrizes a rank 3 bundle $\mathcal{F} = \mathcal{O}(c_1) \oplus \mathcal{O}(c_2) \oplus \mathcal{O}(c_3)$ and a rank 2 bundle $\mathcal{F}^{\prime} = \mathcal{O}(d_1) \oplus \mathcal{O}(d_2)$ such that $c_1 + c_2 + c_3 = d_1 + d_2$, along with a section of the associated rank $\Sym^2(\mathcal{F}) \otimes \mathcal{F}^{\prime\vee}$. While there are only a few possible cases\footnote{We just need to consider the inequality $2c_1 \ge c_2$ and the splitting type over $\infty$.} to consider when summing over bundles for degree 3 covers, there are many, many more cases for degree 4 covers where the local densities change depend on the signs of the linear combinations of the $c_1,c_2,c_3,d_1,d_2$ appearing in summands of $\Sym^2(\mathcal{F}) \otimes \mathcal{F}^{\prime\vee}$. These local densities can be very different from those calculated in \cite[\S4]{b_hcl_iii}, which apply only when all the linear combinations are positive. On the other hand, in our approach, we reduce the computation of $N_{4,q,b}$ to Igusa's computation of the local count \cite{i_stationary_phase}, which is technical but not dependent on as much casework.

There is another possible counting problem where we count quartic covers of $\mathbb{P}^1$ with cubic resolvent rather than $\mathbb{A}^1$. However, the direct approach to this problem also presents many difficulties because many sections will correspond to non-reduced covers. In the cubic case, it is not too difficult to subtract off these sections because they correspond to binary cubic forms with a square factor. In the quartic case, the bad sections correspond to pairs of conics over $\mathbb{F}_q(t)$ tangent at a point, and the inclusion-exclusion seems much more tedious than in the cubic case.
\end{remark}

\subsection{Relation with the cohomology of $\mathfrak{B}_d$}\label{section:rel_coh_bd}
The theorem of Kapranov-Schechtman (\Cref{theorem:ks_nichols_ic}) provides a geometric realization of the cohomology of certain Nichols algebras. In this section, we use our results on the geometry of $\br_M: \mathcal{R}_{M,(\mathbb{P}^1;\infty),b} \to \Sym^b(\mathbb{A}^1)$ to compute the $\mathbf{S}_d$-invariant part of the cohomology of $\mathfrak{B}_d$ for $d \in \{3,4\}$ and propose a way to compute the entire cohomology geometrically.

For this section, let $k$ be a field of characteristic 0.

\begin{theorem}\label{theorem:inv_coh_bd_34}
For $d \in \{3,4\}$, \begin{enumerate}[(a)]
\item There are $\mathbf{S}_d$-equivariant isomorphisms $$\Ext_{\mathfrak{B}_d}^{a,b}(k,k) \cong H^{b - a}(\mathcal{R}_{M,(\mathbb{P}^1;\infty),b}^d(\mathbb{C}),k)^{\vee} \cong H^{b - a}(\mathcal{F}_{M,b}^d(\mathbb{C}),k)^{\vee},$$ where the $\mathbf{S}_d$-action on the $\Ext$ groups is the geometric action.

\item Let $E^{a,b} = \dim\Ext_{\mathfrak{B}_d}^{a,b}(k,k)^{\mathbf{S}_d}$. Then $$\sum_{a = 0}^{\infty}\sum_{b = 0}^{\infty}E^{a,b}q^at^b = I_d(q^{-2},qt).$$
\end{enumerate}
\end{theorem}

\begin{proof}
\begin{enumerate}[(a)]
\item Recall from \Cref{theorem:ks_nichols_ic} that $$\Ext_{\mathfrak{B}_d}^{a,b}(k,k) \cong H^{b - a}(\Sym^b(\mathbb{A}^1)(\mathbb{C}),j_{!*}\pi_*\underline{k}_{\Hur_{\mathbf{S}_d,b}^{\tau_d}})^{\vee},$$ where $\pi$ is the covering map $\Hur_{\mathbf{S}_d,b}^{\tau_d} \to \Conf^b(\mathbb{A}^1)(\mathbb{C})$. By \Cref{lemma:ag_top_comp}, $\pi$ is identified with the map on $\mathbb{C}$-points induced by the finite \'{e}tale map $\pi: \mathcal{H}_{M,(\mathbb{P}^1;\infty),b}^d \to \Conf^b(\mathbb{A}^1)$. By \Cref{remark:analytic_ic}, $j_{!*}\pi_*\underline{k}_{\mathcal{H}_{M,(\mathbb{P}^1;\infty),b}^d(\mathbb{C})} \simeq R\br_{M*}\underline{k}_{\mathcal{R}_{M,(\mathbb{P}^1;\infty),b}^d(\mathbb{C})}$. Thus, we get the first isomorphism $$\Ext_{\mathfrak{B}_d}^{a,b}(k,k) \cong H^{b - a}(\mathcal{R}_{M,(\mathbb{P}^1;\infty),b}^d(\mathbb{C}),k)^{\vee}.$$ The second isomorphism follows from \Cref{corollary:rd_fbd_pure}(a) and the comparison theorem between singular and \'{e}tale cohomology.

\item This follows from taking $\mathbf{S}_d$-invariants of the groups in part (a) and applying \Cref{proposition:coh_fbd_34}(a).
\end{enumerate}
\end{proof}

\begin{corollary}\label{corollary:bosonization_even_b}
Let $E_{\#}^{a,b}(k,k) = \dim\Ext_{\mathfrak{B}_d \# k[\mathbf{S}_d]}^{a,b}(k,k)$, where the bosonization $\mathfrak{B}_d \# k[\mathbf{S}_d]$ is the graded Hopf algebra defined in \cite[\S2.16]{sv_cohomology_fk_alg}. Then for $d \in \{3,4\}$ and $b$ even, $E_{\#}^{a,b}$ is the coefficient of $q^at^b$ in $I_d(q^{-2},qt)$.
\end{corollary}

\begin{proof}
By \cite[Theorem 2.17]{sv_cohomology_fk_alg}, $\Ext_{\mathfrak{B}_d \# k[\mathbf{S}_d]}^{a,b}(k,k)$ is isomorphic to the invariants of the standard $\mathbf{S}_d$-action on $\Ext_{\mathfrak{B}_d}^{a,b}(k,k)$. Because the standard and geometric $\mathbf{S}_d$-actions on $\mathfrak{B}_d$ agree in even degrees, the actions on $\Ext_{\mathfrak{B}_d}^{a,b}(k,k)$ agree for even $b$, as we can compute this $\Ext$ group using the normalized bar complex of $\mathfrak{B}_d$ as in \cite[\S2.8]{sv_cohomology_fk_alg}. Thus, the corollary follows from \Cref{theorem:inv_coh_bd_34}(b).
\end{proof}

\begin{remark}\label{remark:sv_compare_future}
In \cite{sv_cohomology_fk_alg}, \c{S}tefan-Vay compute the entire cohomology ring of $\mathfrak{B}_3$, so our result \Cref{theorem:inv_coh_bd_34}(b) is not new in this case. However, we hope that the algebraically-inclined reader will find the connection with geometry and the $(G_3,V_3)$ Igusa zeta function interesting. On the other hand, our result is new for $\mathfrak{B}_4$.

We can actually apply $\mathbb{G}_m$-localization to our marked compactifications $\mathcal{R}_{M,(\mathbb{P}^1;\infty),b}^3$ to reprove the computation of \c{S}tefan-Vay, though without the multiplicative structure. The $k$-points of the marked central fiber $\mathcal{F}_{M,b}^3$ correspond to triple covers of $\Spec k[[t]]$ together with an isomorphism of the generic fiber with some \'{e}tale $k((t))$-algebra $S$. If we let $R$ be the integral closure of $k[[t]]$ in $S$, then providing an isomorphism of the generic fiber with $S$ is equivalent to providing an embedding of the triple. This means that counting $k$-points of $\mathcal{F}_{M,b}^3$ is equivalent to counting cubic $k[[t]]$-subalgebras of various $R$ (e.g. $k[[t]]^{\times 3}$), with some multiplicity depending on $R$. In \cite[\S4]{d_modular}, Deopurkar explicitly describes all the $k$-points of $\mathcal{F}_{M,b}^3$ for $k$ of characteristic 0, and his work equally applies to any $k$ of characteristic not 2 or 3. If we consider the point counts as $k$ ranges over $\mathbb{F}_q$ with $(q,6) = 1$, we can get the cohomology of $\mathcal{F}_{M,b}^3(\mathbb{C})$, which gives us the full cohomology of $\mathfrak{B}_3$ by $\mathbb{G}_m$-localization. Deopurkar even shows that $\mathcal{F}_{M,b}^3$ is stratified by affine spaces, so that the cohomology of $\mathfrak{B}_3$ is all pure Tate.

This approach also applies to $\mathfrak{B}_4$ and $\mathfrak{B}_5$ where the cohomology is unknown. However, the local problem of counting degree 4 or 5 $k[[t]]$-subalgebras of various $R$ (or even just the homogeneous subalgebras, which correspond to $\mathbb{G}_m$-fixed points) is much more complicated, especially coupled with the fact that the resolvents must also be counted.
\end{remark}



\bibliographystyle{alpha}
\nocite{*}
\bibliography{Bibliography}

\end{document}